\documentclass{amsart}

\usepackage{latexsym, amsfonts, amsthm, amsmath, amssymb, amscd, epsfig, amsrefs}
\usepackage{color}
\usepackage{mathrsfs}
\allowdisplaybreaks[1]

\newtheorem{lemma}{Lemma}[section]
\newtheorem{Add in proof}{Add in proof}[section]
\newtheorem{prop}{Proposition}[section]
\newtheorem{defn}{Definition}

\newtheorem{theorem}{Theorem}

\newcommand{\R}{\ensuremath{\mathbb{R}}}

 \def\p#1#2{\dfrac{\partial
#1}{\partial#2}}
\def \p{\partial}

 \def\a{\alpha}
\def\intave#1{-\kern-10.7pt\int_{\,#1}}
\def\b{\beta}
\def\D{\nabla}
\def\g{\gamma}

\def\<{\langle}
\def\>{\rangle}
\def\({\left(}
\def\){\right)}

\def\limsup{\operatornamewithlimits{lim\,sup}}
\def\intave#1{-\kern-10.7pt\int_{\,#1}}

\newcommand{\na}{\ensuremath{\nabla}}
\newcommand\alabel[1]{\addtocounter{equation}{1}\tag{\theequation}\label{#1}}

\begin{document}
\title{Convergence of  the Ginzburg-Landau approximation for the Ericksen-Leslie system}

\numberwithin{equation}{section}
\author{Zhewen Feng, Min-Chun Hong and Yu Mei}

\address{Zhewen Feng, Department of Mathematics, The University of Queensland\\
	Brisbane, QLD 4072, Australia}
\email{z.feng@uq.edu.au}

\address{Min-Chun Hong, Department of Mathematics, The University of Queensland\\
Brisbane, QLD 4072, Australia}
\email{hong@maths.uq.edu.au}

\address{Yu Mei, Department of Mathematics, The University of Queensland\\
Brisbane, QLD 4072, Australia}
\email{y.mei@uq.edu.au}

\begin{abstract}
	We establish the local well-posedness of  the general Ericksen-Leslie system in liquid crystals with the initial velocity and director field  in $H^1\times H^2_b$. In particular, we prove that the solutions of the Ginzburg-Landau  approximation system  converge smoothly to the solution of the Ericksen-Leslie system for any $t\in (0, T^*)$ with a maximal existence time $T^*$ of the Ericksen-Leslie system.
     \end{abstract}

 \maketitle

\pagestyle{myheadings} \markright {Convergence of the Ginzburg-Landau approximation}

\section{Introduction}
 In the 1960s, Ericksen \cite{Er} and Leslie \cite{Le} proposed a celebrated hydrodynamic theory to describe the behavior of liquid crystals. The   Ericksen-Leslie theory has  been widely accepted since then as one of the most successful theories for modeling liquid crystal flows (c.f. \cite{Ste}). Let $v=(v^1,v^2,v^3)$ be the
velocity vector of the fluid and $u=(u^1,u^2,u^3)\in S^2$ the unit
direction vector. Then the Ericksen-Leslie system in $\R^3\times
[0,\infty )$ is given by (c.f. \cites {LL1,Ste})
 \begin{align}
&\p_tv+v\cdot\D v+\D P=\D\cdot(\sigma^E+\sigma^L),\label{E-L1}\\
&\D\cdot v=0,\label{E-L2}\\
& u\times\left(\g_1 N+\g_2Au-h\right)=0\label{E-L3},
\end{align}
where   $P$ represents  the pressure and $\sigma^E$ denotes the Ericksen stress tensor  given by
\begin{equation}\label{E-T}
\sigma^E=-\D u^T\frac{\p W(u,\D u)}{\p (\D u)}.
\end{equation}
Here the Oseen-Frank density $W(u,\D u)$ takes the form
\begin{equation*}
W(u,\D u)=k_1(\text{div } u)^2+k_2 (u\cdot\text{curl }u)^2+k_3
|u\times \text{curl } u|^2+(k_2+k_4)[\text {tr} (\nabla u)^2-(\text {div
}u)^2],
\end{equation*}
where $k_1,k_2,k_3,k_4$ are Frank's elastic constants. The Leslie stress tensor $\sigma^L$ satisfies the constitutive relation
\begin{equation}\label{L-T}
\sigma^L=\alpha_1(u\otimes u:A)u\otimes u+\alpha_2 N\otimes u+\alpha_3 u\otimes N+\alpha_4 A+\alpha_5 (Au)\otimes u+\alpha_6 u\otimes (Au),
\end{equation}
where $\alpha_i$, $i=1,2,\cdots, 6$ are the Leslie coefficients. The co-rotational time derivative $N$ of $u$  is defined by
\begin{equation}\label{co-der}
N=\p_tu+v\cdot\D u-\Omega u.
\end{equation}
We denote by $\Omega$ and $A$ the skew-symmetric and symmetric parts of the tensor $\D v$ respectively; that is
\[ \Omega=\dfrac{1}{2}(\nabla v-(\nabla v)^T),\qquad A=\dfrac{1}{2}(\nabla v+(\nabla v)^T).\]
The molecular field $h$ in (\ref{E-L3}) is given by
\begin{equation}\label{M-F}
h=\D\cdot\left(\frac{\p W(u,\D u)}{\p (\D u)}\right)-\frac{\p W(u,\D u)}{\p u}.
\end{equation}

In the sequel, the following assumptions are introduced: Frank's elastic constants $k_1,k_2,k_3,k_4$ satisfy the strong Ericksen inequalities (c.f.  \cite{Ball})
\begin{equation}\label{Co1}k_1> 0, \,k_2 >|k_4|,\,k_3>0, \,2k_1> k_2+k_4. \end{equation}
Under which there are positive constants $a,C>0$ such that the density $W(u,\na u)$ is equivalent to a form that satisfies
\begin{equation}\label{Co2}
 a |p|^2\leq W(z,p)\leq C|p|^2,\quad
W_{p_i^kp_j^l}(z,p) \xi_i^k \xi_j^l\geq a |\xi|^2
\end{equation} for any $\xi\in\mathbb M^{3\times 3}$, any $z\in \R^3$ and any $p\in\mathbb M^{3\times 3}$ (c.f. \cites{HM,Ericksen}). The Leslie coefficients $\alpha_i$, $i=1,2,\cdots, 6$, are assumed to satisfy the following conditions:
\begin{equation}\label{ag2}
\g_1=\alpha_3-\alpha_2>0,\quad \g_2=\alpha_6-\alpha_5,
\quad \alpha_2+\alpha_3=\alpha_6-\alpha_5,
\end{equation}
where the last equation is called the \emph{Parodi} relation (c.f. \cite{Ste}). Further, suppose that
\begin{equation}\label{Leslie coe-2}
\alpha_1\geq 0,\quad \alpha_4>0,\quad \beta:=\alpha_5+\alpha_6-\frac{\g_2^2}{\g_1}\geq 0,
\end{equation}
which ensures the energy-dissipation law of the general Ericksen-Leslie system.

The  Ericksen-Leslie system \eqref{E-L1}-\eqref{E-L3}   has attracted much attention in recent years.
 For the two-dimensional case, Lin-Lin-Wang \cite{LLW} and Hong \cite{Ho3} independently proved global existence and partial regularity of weak solutions to the simplified system; that is a special case where Frank's elastic constants in the isotropic case satisfy $k_1=k_2=k_3=1,k_4=0$  and the Leslie tensor is ignored (other than $\alpha_4\neq 0$). Hong-Xin \cite{HX} generalized these results to any positive $k_1, k_2, k_3$, but without the Leslie tensor. Later, Huang-Lin-Wang \cite{HLW} and Wang-Wang \cite{WW} obtained similar results in $\R^2$ for the system  \eqref{E-L1}-\eqref{E-L3} with the Leslie tensor. Lin-Wang \cite{LW}, Li-Titi-Xin \cite{LTX} and Wang-Wang-Zhang \cite{WWZ}  established uniqueness of global weak solutions of the the system  \eqref{E-L1}-\eqref{E-L3}.

In three dimensions, the question on global existence of weak solutions to the Ericksen-Leslie system \eqref{E-L1}-\eqref{E-L3} remains open.   Wen-Ding \cite{WD} established the local well-posedness of strong solutions to the simplified system without the Leslie stress tensor in the isotropic case ($k_1=k_2=k_3=1, k_4=0$ and only $\alpha_4\neq0$). Later, Fan-Guo \cite{FG} and Huang-Wang \cite{HW2} studied the Serrin and BKM type blow-up criteria for this simplified system, using ideas originating from the celebrated Navier-Stokes equation.  For the Ericksen-Leslie system with general Oseen-Frank density and without the Leslie tensor, Hong-Li-Xin \cite{HLX} proved the local well-posedness and blow-up criterions of strong solutions with initial data $(v_0, u_0)\in H^1(\R^3,\R^3)\times H^2_b(\R^3, S^2)$. For rough initial data,  Hineman-Wang \cite{HW} established the local well-posedness of solutions to the simplified system with initial velocity $v_0$ and director $u_0$ in uniformly local  $L^3$-integrable spaces respectively. See also Wang \cite{W} for the case with initial data in $BMO^{-1}\times BMO$. Recently, Hong-Mei \cite{HM} generalized the result in \cite{HW} to the case of any positive $k_1, k_2, k_3$ with initial data in uniformly local $L^3$ spaces,   but without the effect of the Leslie tensor.

Now, we consider the effect of Leslie stress tensor for the general Ericksen-Leslie system in three dimensions.  Wang-Zhang-Zhang \cite{WZZ} and Wang-Wang \cite{WW} proved the local well-posedness of solutions to the general Ericksen-Leslie system with initial data $(v_0, u_0)\in H^{2s}(\R^3,\R^3)\times H^{2s}_b(\R^3, S^2)$ with $s\geq 2$.
 In this article, we investigate the local well-posedness of strong solutions to the general Ericksen-Leslie system \eqref{E-L1}-\eqref{E-L3} with initial data $(v_0, u_0)\in H^1(\R^3,\R^3)\times H^2_b(\R^3, S^2)$.

For a given unit vector $b\in S^2$ and $m\in\mathbb{N}$, we denote
\begin{align}
H^m_{b}(\R^3;S^2):=\{u:u-b\in H^{m}(\R^3;\R^3), |u|=1 \text{ a.e. in }\R^3\}.
\end{align}
\begin{defn}
    For any $T>0$, $(v,u)$ is called a strong solution to the system \eqref{E-L1}-\eqref{E-L3} in $\R^3\times (0,T)$ if it satisfies the system \eqref{E-L1}-\eqref{E-L3} for almost every $(x,t)\in \R^3\times(0,T)$ and
    \begin{align*}
    & v\in L^\infty(0,T;H^1(\R^3))\cap L^2(0,T;H^2(\R^3)),\quad \p_tv\in L^2(0,T;L^2(\R^3)),\quad \text{div }v=0,\\
    & u\in L^\infty(0,T;H^2_b(\R^3))\cap L^2(0,T; H^3_b(\R^3)),\quad\p_tu\in L^2(0,T;H^1(\R^3)),\quad |u|=1.
    \end{align*}
\end{defn}

Firstly, we prove the local well-posedness of strong solutions to \eqref{E-L1}-\eqref{E-L3}:
\begin{theorem} \label{thm1}
     For any  $v_0\in H^1(\R^3,\R^3)$ and $u_0\in H^2_b(\R^3, S^2)$ with $\text{div } v_0=0$, there is a unique strong
    solution $(v, u)$ to the system \eqref{E-L1}-\eqref{E-L3}  in $\R^3\times[0, T^*)$ with initial data $(v_0,u_0)$.
    Moreover, there are two positive constants $\varepsilon_0$ and $R_0$ such that at a singular point $x_i$, the maximal existence time $T^*$ satisfies
    \begin{equation*}
    \limsup_{t\rightarrow T^*}\int_{B_R(x_i)}|\D u(\cdot,t)|^3+|v(\cdot,t)|^3\;dx\geq \varepsilon_0,
    \end{equation*}
    for any $R>0$ with $R\leq R_0$.
\end{theorem}

In line with previous efforts, the proof of Theorem \ref{thm1} utilizes the Ginzburg-Landau approximation.  The Ginzburg-Landau functional  was introduced  in 1950 (\cite{GL}) to study the phase transition in superconductivity.  For a  parameter $\varepsilon>0$,  the Ginzburg-Landau functional of $u:\Omega\rightarrow \R^3$ is defined by
\begin{equation}\label{GL}
E_{\varepsilon}(u; \Omega):=\int_{\Omega}\left(\frac 12 |\nabla u|^2+\frac{1}{4\varepsilon^2}(1-|u|^2)^2\right)\;dx.
\end{equation}
  There are many impressive results concerning  convergence of the Ginzburg-Landau  approximation system as $\varepsilon \to 0$.  In \cite{CS}, Chen-Struwe proved global existence of weak solutions to the heat flow of harmonic maps using the Ginzburg-Landau approximation. See    a further result in  \cite{BOS} on the convergence of the gradient flow of the Ginzburg-Landau approximation. On the other hand,
  Bethuel, Brezis and H\'elein \cites{BBH,BBO}  proved asymptotic behavior for minimizers of $E_{\varepsilon}$ in two dimensional  star-shaped domains as $\varepsilon \to 0$  (see also
  \cite{Struwe} for the case of non-star-shaped domains).
Motivated by above results, Lin-Liu \cites{LL1,LL2}  introduced
the Ginzburg-Landau approximation system for the Ericksen-Leslie system
\begin{align}
&\p_t v_\varepsilon+v_\varepsilon\cdot\D v_\varepsilon+\D P_\varepsilon=\D\cdot(\sigma_\varepsilon^E+\sigma_\varepsilon^L),\label{G-L1}\\
&\nabla\cdot v_\varepsilon=0,\label{G-L2}\\
&\g_1N_\varepsilon+\g_2 A_\varepsilon u_\varepsilon=h_\varepsilon \label{G-L3}
\end{align}
for $\varepsilon>0$, where $u_\varepsilon$, $v_\varepsilon$ are the direction and velocity field of the Ginzburg-Landau system and $h_\varepsilon$ is given by
\begin{equation}\label{M-F1}
h^i_\varepsilon=\D_{\alpha} \left(\frac{\p W(u_\varepsilon,\D u_\varepsilon)}{\partial p^i_{\alpha}}\right)-\frac{\p W(u_\varepsilon,\D u_\varepsilon)}{\p u_\varepsilon^i}+\frac{1}{\varepsilon^2}u^i_\varepsilon(1-|u_\varepsilon|^2).
\end{equation}
  Lin-Liu \cites{LL1,LL2}  proved   global existence of classical solutions in two dimensions and weak solutions  in three dimensions to the Ginzburg-Landau system (see also \cite{CRW} for the $\gamma_2\neq 0$ case).     Lin-Liu  \cite {LL2} also
analyzed the limit of solutions
$(v_{\varepsilon},u_{\varepsilon})$ of the Ginzburg-Landau system  as
$\varepsilon\to 0$, but it is not clear that the limiting solution
satisfies the original Ericksen-Leslie system   with
$|u|=1$.  In the study of  numerical context,
 it is a widely used approach to handle the constraint $|u| = 1$ in the Ericksen-Leslie equations through the Ginzburg-Landau approximation system \eqref {G-L1}-\eqref{G-L3} (c.f. \cites{LW2,WN}).
Therefore, it is an important question  to study the convergence of weak solutions to the Ginzburg-Landau approximation system \eqref {G-L1}-\eqref{G-L3}.   Hong \cite{Ho3} and Hong-Xin \cite{HX} proved the local existence of weak solutions of the  Ericksen-Leslie system without the Leslie stress tensor in $\R^2$ using the Ginzburg-Landau approximation approach.  For the three-dimensional problem,  the convergence of solutions to the Ginzburg-Landau approximation system is challenging in the framework of weak solutions, owing to a lack of uniform a priori estimates (refer to \cites{LL2,LW}).   Hong-Li-Xin \cite{HLX} first justified the local (in time) convergence of strong solutions to \eqref{G-L1}-\eqref{G-L3} without Leslie stress tensor for initial velocity and director field in $H^1(\R^3)\times H^2_b(\R^3)$.  We would like to point out that by using the Ginzburg-Landau approximation,  Lin-Wang \cite{LW} proved global existence of weak solutions to the simplified system in dimension three with initial velocity and director in $L^2\times H^1$ and hemisphere condition on the director.
 Building on the ideas of \cites{Ho3, HX, HLX}, we prove Theorem  \ref{thm1} by establishing the convergence of strong solutions to \eqref{G-L1}-\eqref{G-L3}, when the Leslie stress tensor is present. One of the key ideas is that when $|u_\varepsilon|$ is close to $1$, we handle the singular term $\frac{1-|u_\varepsilon|^2}{\varepsilon^2}$ using \eqref{G-L3}.

Concerning the Lin-Liu problem on the convergence of the Ginzburg-Landau approximation,   Hong-Li-Xin \cite{HLX} proved the strong convergence of the Ginzburg-Landau approximation system up to the maximal existence time of the Ericksen-Leslie system without the  Leslie stress tensor.   In this paper,  we extend the result in \cite{HLX} to the general Ericksen-Leslie system \eqref{E-L1}-\eqref{E-L3} with the  Leslie stress tensor.

\begin{theorem}\label{thm2} For each $\varepsilon >0$, there is a unique strong  solution $(v_\varepsilon,u_\varepsilon)$ to the system \eqref{G-L1}-\eqref{G-L3} in  $\R^3\times[0,T^*_\varepsilon)$ with  initial data $(v_0, u_0) \in H^1(\R^3)\times H^2_b(\R^3)$ satisfying $div\;v_0=0$, where $T^*_\varepsilon$ is the maximal existence time.  Let $T^*$ be the maximal existence time  of the strong solution $(v, u)$ to the system \eqref{E-L1}-\eqref{E-L3}    with the same initial data $(v_0,u_0)$ in Theorem  \ref{thm1}. Then, we have $(\D u,v)\in C^\infty(\tau,T ;C_{loc}^\infty(\R^3))$ with any $(\tau, T)\subset(0, T^*)$. Moreover,
 for any $T\in(0, T^*)$,  there exists a small positive  $\varepsilon_T$ such that $T^\ast_\varepsilon\geq T$ for any  $\varepsilon\leq \varepsilon_T$, and  as $\varepsilon\to 0$,
\begin{equation}\label{con1}
(\D u_\varepsilon,v_\varepsilon)\rightarrow(\D u,v),\qquad\text{in }~~L^\infty(0,T;L^2(\R^3))\cap L^2(0,T;H^1(\R^3))
\end{equation}
and
\begin{equation}\label{con2}
(\D u_\varepsilon,v_\varepsilon)\rightarrow(\D u,v),\qquad\text{in }~~C^\infty(\tau,T;C_{loc}^\infty(\R^3)) \quad\text{for  any }\tau>0.
\end{equation}
\end{theorem}
We would like to point out  that   the  smooth convergence   in  \eqref {con2}   is a new result even for the Ericksen-Leslie system without the  Leslie stress tensor.
One of the key proofs to Theorem \ref{thm2} is to establish Proposition \ref{prop2.1} under the condition that $(v_{0,\varepsilon}, u_{0,\varepsilon})$ satisfies
\begin{equation}\label{con3}
\|u_{0,\varepsilon}-b\|_{H^2(\R^3)}^2+\|v_{0,\varepsilon}\|_{H^1(\R^3)}^2+\frac 1{\varepsilon^2}\|(1-|u_{0,\varepsilon}|^2)\|_{H^1(\R^3)}^2\leq M
\end{equation} for a positive constant $M$ independent of $\varepsilon$. Note that the condition \eqref{con3} does not involve any condition on $\|\p_t u_{0,\varepsilon}\|_{L^2(\R^3)}$, which differs from the one in \cite{HLX}. To prove Proposition \ref{prop2.1}, we establish a local estimate  on the  pressure in Lemma \ref{pressure estimate} and derive a local $L^3$-estimate using an interpolation inequality and a covering argument, which is similar to the argument in \cite{HM}. By applying Proposition \ref{prop2.1}, we prove that as $\varepsilon\to 0$, the solutions $(v_\varepsilon,u_\varepsilon)$ of \eqref{G-L1}-\eqref{G-L3} converge  strongly to the solution $(v,u)$ of the system \eqref{E-L1}-\eqref{E-L3}  in $\R^3\times (0, T_M]$  with  a uniform constant $T_M>0$ depending only on $M$.

The second key proof  to Theorem \ref{thm2} is to derive sophisticated higher order estimates of $(v_\varepsilon,\na u_\varepsilon)$ with uniform bounds in $\varepsilon$ in Lemma \ref{higher estimates lemma}, which  implies   the smooth convergence results of Ginzburg-Landau approximation systems  in $\R^3\times (0, T_M]$.
Let $T^\ast$  be the maximal existence time of the solution  $(v, u)$ to the Ericksen-Leslie system. For any $T<T^*$,     we choose
$M= 2\sup_{0\leq t\leq T}\|(\nabla u,v)\|_{H^1(\R^3)}^2$. Then we
combine the energy identities  in Lemma \ref{lem4.2}   with  the higher order estimates  to verify  that    $(v_\varepsilon, u_\varepsilon)$ satisfies \eqref {con3} at $t=T_M$. Therefore,  the solutions $(v_\varepsilon,u_\varepsilon)$ to the Ginzburg-Landau system  converge smoothly to the solution $(v,u)$  in $\R^3\times (0, 2T_M]$ for sufficiently small $\varepsilon$. Finally, we  establish the smooth convergence of solutions to  Ginzburg-Landau approximation systems  for any $T<T^*$.

  The   paper is organized as follows. In Section 2, we obtain some a priori estimates of the Ericksen-Leslie system \eqref{E-L1}-\eqref{E-L3}. In Section 3, we  establish Proposition \ref{prop2.1} and prove Theorem \ref{thm1}. In Section 4, we establish higher order estimates of the Ginzburg-Landau approximation system and prove Theorem \ref{thm2}.

\section{a priori estimates}
In this section, we derive  a priori estimates for strong solutions to the Ginzburg-Landau system \eqref{G-L1}-\eqref{G-L3}. First,
 we note that  the  equation   \eqref{E-L3} is equivalent to
\begin{equation}\label{E-L3-0}
\g_1 N+\g_2(Au-(u^TAu)u)=h-(u\cdot h)u
\end{equation}
 by taking the vector cross product to \eqref{E-L3} with $u$ and using the fact  that $|u|=1$.

Then we have the following basic energy identity:
\begin{lemma}\label{energy-identity}
	Let $(v_\varepsilon,u_\varepsilon)$ be a strong solution to the system \eqref{G-L1}-\eqref{G-L3} in $\R^3\times (0,T_\varepsilon)$. Then  for any $t\in(0,T_\varepsilon)$ we have
	\begin{align}\label{basic energy}
	&\frac{d}{dt}\int_{\R^3}\left(\frac{|v_\varepsilon|^2}{2}+W(u_\varepsilon,\D u_\varepsilon)+\frac{1}{4\varepsilon^2}(1-|u_\varepsilon|^2)^2\right)\,dx+\alpha_4\int_{\R^3}|A_\varepsilon|^2\,dx\\
	&+\alpha_1\int_{\R^3}|u_\varepsilon^TA_\varepsilon u_\varepsilon|^2\,dx+\beta\int_{\R^3}|A_\varepsilon u_\varepsilon|^2\,dx+\frac{1}{\g_1}\int_{\R^3}|h_\varepsilon|^2\,dx=0.\nonumber
	\end{align}
\end{lemma}
\begin{proof}
	Multiplying \eqref{G-L1} by $v_\varepsilon$, using \eqref{G-L2} and integrating by parts yield
	\begin{align}\label{2.2}
	\frac{1}{2}\frac{d}{dt}\int_{\R^3}|v_\varepsilon|^2\,dx+\int_{\R^3}\sigma^L_\varepsilon:\D v_\varepsilon \,dx=-\int_{\R^3}\sigma^E_\varepsilon:\D v_\varepsilon \,dx.
	\end{align}
	Since $A_\varepsilon$ is symmetric and $\Omega_\varepsilon$  is antisymmetric, it follows from \eqref{L-T}, \eqref{ag2} and \eqref{G-L3} that
	\begin{align}\label{2.3}
	&\int_{\R^3}\sigma^L_\varepsilon:\D v_\varepsilon \,dx=\int_{\R^3}\sigma^L_\varepsilon:(A_\varepsilon+\Omega_\varepsilon)\,dx\\
	=&\int_{\R^3}\big[\alpha_1|u_\varepsilon^TA_\varepsilon u_\varepsilon|^2+\alpha_4|A_\varepsilon|^2+(\alpha_5+\alpha_6)|A_\varepsilon u_\varepsilon|^2\nonumber\\
	&-(\g_1N_\varepsilon+\g_2 A_\varepsilon u_\varepsilon)\cdot(\Omega_\varepsilon  u_\varepsilon)+\g_2N_\varepsilon\cdot(A_\varepsilon u_\varepsilon)\big]\,dx\nonumber\\
	=&\alpha_1\int_{\R^3}|u_\varepsilon^TA_\varepsilon u_\varepsilon|^2\,dx+\alpha_4\int_{\R^3}|A_\varepsilon|^2\,dx+(\alpha_5+\alpha_6-\frac{\g_2^2}{\g_1})\int_{\R^3}|A_\varepsilon u_\varepsilon|^2\,dx\nonumber\\
	&\quad-\int_{\R^3}h_\varepsilon^T\Omega_\varepsilon u_\varepsilon \,dx+\frac{\g_2}{\g_1}\int_{\R^3}h_\varepsilon^TA_\varepsilon u_\varepsilon \,dx.\nonumber
	\end{align}
	Substituting \eqref{2.3} into \eqref{2.2} and using \eqref{E-T}, we have
	\begin{align}\label{2.4.0}
	&\frac{1}{2}\frac{d}{dt}\int_{\R^3}|v_\varepsilon|^2\,dx+\alpha_1\int_{\R^3}|u_\varepsilon^TA_\varepsilon u_\varepsilon|^2\,dx+\alpha_4\int_{\R^3}|A_\varepsilon|^2\,dx+\beta\int_{\R^3}|A_\varepsilon u_\varepsilon|^2\,dx\\
	=&\int_{\R^3}\D_iu_\varepsilon^kW_{p_j^k}(u_\varepsilon,\D u_\varepsilon)\D_j v_\varepsilon^i\,dx+\int_{\R^3}h_\varepsilon^T\Omega_\varepsilon u_\varepsilon \,dx-\frac{\g_2}{\g_1}\int_{\R^3}h_\varepsilon^TA_\varepsilon u_\varepsilon \,dx.\nonumber
	\end{align}
	
	On the other hand, multiplying \eqref{G-L3} by $\frac{1}{\g_1}h_\varepsilon$, integrating over $\R^3$ and using \eqref{co-der}, we have
	\begin{align}\label{2.4}
	&-\int_{\R^3}\p_tu_\varepsilon\cdot h_\varepsilon \,dx+\frac{1}{\g_1}\int_{\R^3}|h_\varepsilon|^2\,dx\\
	=&\int_{\R^3}(v_\varepsilon\cdot \D) u_\varepsilon\cdot h_\varepsilon \,dx-\int_{\R^3}h_\varepsilon^T\Omega_\varepsilon u_\varepsilon \,dx+\frac{\g_2}{\g_1}\int_{\R^3}h_\varepsilon^TA_\varepsilon u_\varepsilon \,dx.\nonumber
	\end{align}
	It follows from \eqref{M-F1} and integration by parts that
	\begin{align}\label{2.5}
	-\int_{\R^3}\p_tu_\varepsilon\cdot h_\varepsilon \,dx=&\int_{\R^3}\left(\p_t\D_\alpha u^i_\varepsilon W_{p^i_\alpha}(u_\varepsilon,\D u_\varepsilon)+\p_tu^i_{\varepsilon}W_{u_\varepsilon^i}(u_\varepsilon,\D u_\varepsilon)\right)\,dx\\
	&-\frac{1}{\varepsilon^2}\int_{\R^3}\p_t u_\varepsilon\cdot u_\varepsilon(1-|u_\varepsilon|^2)\,dx\nonumber\\
	=&\frac{d}{dt}\int_{\R^3}\left(W(u_\varepsilon,\D u_\varepsilon)+\frac{1}{4\varepsilon^2}(1-|u_\varepsilon|^2)^2\right)\,dx.\nonumber
	\end{align}
  Using \eqref{G-L2} and   integration by parts, we have
	\begin{align}\label{2.6}
	&\int_{\R^3}(v_\varepsilon\cdot \D) u_\varepsilon\cdot h_\varepsilon \,dx\\
	=&\frac{1}{\varepsilon^2}\int_{\R^3}(v_\varepsilon\cdot\D) u_\varepsilon\cdot u_\varepsilon(1-|u_\varepsilon|^2)\,dx \nonumber\\
	&+\int_{\R^3}v_\varepsilon^k\D_k u_\varepsilon^i(\D_{\alpha} W_{p_\alpha^i}(u_\varepsilon,\D u_\varepsilon)-W_{u_\varepsilon^i}(u_\varepsilon,\D u_\varepsilon))\,dx\nonumber\\
	=&-\frac{1}{4\varepsilon^2}\int_{\R^3}(v_\varepsilon\cdot\D)(1-|u_\varepsilon|^2)^2\,dx-\int_{\R^3}\D_\alpha v^k_\varepsilon\D_ku_\varepsilon^iW_{p_\alpha^i}(u_\varepsilon,\D u_\varepsilon)\,dx\nonumber\\
	&-\int_{\R^3}v_\varepsilon^k\left(\D_k\D_{\alpha}u_\varepsilon^i W_{p_\alpha^i}(u_\varepsilon,\D u_\varepsilon)-\D_k u_\varepsilon^iW_{u_\varepsilon^i}(u_\varepsilon,\D u_\varepsilon)\right)\,dx\nonumber\\
	=&-\int_{\R^3}\D_\alpha v^k_\varepsilon\D_ku_\varepsilon^iW_{p_\alpha^i}(u_\varepsilon,\D u_\varepsilon)\,dx.\nonumber
	\end{align}
	Plugging \eqref{2.5} into \eqref{2.4} gives
	\begin{align}\label{2.7}
	&\frac{d}{dt}\int_{\R^3}\left(W(u_\varepsilon,\D u_\varepsilon)+\frac{1}{4\varepsilon^2}(1-|u_\varepsilon|^2)^2\right)\,dx+\frac{1}{\g_1}\int_{\R^3}|h_\varepsilon|^2\,dx\\
	=&-\int_{\R^3}\D_\alpha v^k_\varepsilon\D_ku_\varepsilon^iW_{p_\alpha^i}(u_\varepsilon,\D u_\varepsilon)\,dx-\int_{\R^3}h_\varepsilon^T\Omega_\varepsilon u_\varepsilon \,dx+\frac{\g_2}{\g_1}\int_{\R^3}h_\varepsilon^TA_\varepsilon u_\varepsilon \,dx.\nonumber
	\end{align}
	Therefore, summing \eqref{2.4.0} with \eqref{2.7} yields \eqref{basic energy}.
	\end{proof}

The following lemma gives the local energy-dissipation law of the Ginzburg-Landau system \eqref{G-L1}-\eqref{G-L3}.
\begin{lemma}\label{local estimate}
	Let $(v_\varepsilon, u_\varepsilon)$ be a strong solution to the system \eqref{G-L1}-\eqref{G-L3} in $\R^3\times(0, T_\varepsilon)$. Assume that $\frac{1}{2}\leq|u_\varepsilon|\leq\frac{3}{2}$ in $\R^3\times(0, T_\varepsilon)$.  Then for any $\phi\in C_0^\infty(\R^3)$ and $s\in(0,T_\varepsilon)$, we obtain
	\begin{align}\label{local energy estimates}
	&\int_{\R^3}\left(|v_\varepsilon(x,s)|^2+|\D u_\varepsilon(x,s)|^2+\frac{(1-|u_\varepsilon(x,s)|^2)^2}{\varepsilon^2}\right)\phi^2\,dx\\
	&+\int_{0}^{s}\int_{\R^3}\left(|\D v_\varepsilon|^2+|\D^2u_\varepsilon|^2+|\p_t u_\varepsilon|^2+\frac{|\D(|u_\varepsilon|^2)|^2}{\varepsilon^2}\right)\phi^2\,dxdt\nonumber\\
	\leq& C\int_{\R^3}\left(|v_{0,\varepsilon}|^2+|\nabla u_{0,\varepsilon}|^2+\frac{(1-|u_{0,\varepsilon}|^2)^2}{\varepsilon^2}\right)\phi^2\,dx\nonumber\\
	&+C\int_{0}^{s}\int_{\R^3}(|P_\varepsilon-c_\varepsilon(t)|+|v_\varepsilon|^2)|v_\varepsilon||\D\phi|\phi \,dxdt\nonumber\\
	&+C\int_{0}^{s}\int_{\R^3}(|v_\varepsilon|^2+|\D u_\varepsilon|^2)|\D u_\varepsilon|^2\phi^2\,dxdt\nonumber\\
	&+C\int_{0}^{s}\int_{\R^3}(|\D u_\varepsilon|^2+|v_\varepsilon|^2)|\D\phi|^2\,dxdt,\nonumber
	\end{align}
	where $C$ is a positive constant independent of $\varepsilon$ and $c_\varepsilon(t)\in \R$. In particular, for any $s\in(0,T_\varepsilon)$, we have
	\begin{align}\label{fir-ord}
	&\int_{\R^3}\left(|v_\varepsilon(x,s)|^2+|\D u_\varepsilon(x,s)|^2+\frac{(1-|u_\varepsilon(x,s)|^2)^2}{\varepsilon^2}\right)\,dx\\
	&+ \int_{0}^{s}\int_{\R^3}\left(|\D v_\varepsilon|^2+|\D^2u_\varepsilon|^2+ |\p_t u_\varepsilon|^2+\frac{1}{\varepsilon^2}|\na(|u_\varepsilon|)^2|^2\right)\,dxdt\nonumber
	\\
	\leq& C\int_{\R^3}|v_{0,\varepsilon}|^2+|\nabla u_{0,\varepsilon}|^2+\frac{(1-|u_{0,\varepsilon}|^2)^2}{\varepsilon^2}\,dx+C\int_{0}^{s}\int_{\R^3}(|v_\varepsilon|^2+|\D u_\varepsilon|^2)|\D u_\varepsilon|^2\,dxdt.\nonumber
	\end{align}
	
\end{lemma}
\begin{proof}
	Multiplying \eqref{G-L1} by $v^i_\varepsilon\phi^2$, integrating over $\R^3$ and using the similar calculations in \eqref{2.2} and \eqref{2.3} yield
	\begin{align*}\alabel{2.48}
	&\frac{d}{dt}\int_{\R^3}\frac{|v_\varepsilon|^2}{2}\phi^2\,dx+\int_{\R^3}\alpha_1|u_\varepsilon^T A_\varepsilon u_\varepsilon|^2\phi^2+\alpha_4|A_\varepsilon|^2\phi^2+\beta|A_\varepsilon u_\varepsilon|^2\phi^2\,dx
	\\
	=&\int_{\R^3}h_\varepsilon^T\Omega_\varepsilon u_\varepsilon \phi^2\,dx-\frac{\g_2}{\g_1}\int_{\R^3}h_\varepsilon^T A_\varepsilon u_\varepsilon \phi^2\,dx+2\int_{\R^3}(P_\varepsilon-c_\varepsilon(t)) v_\varepsilon\cdot \D \phi\phi \,dx
	\\
	&+2\int_{\R^3}|v_\varepsilon|^2v_\varepsilon\cdot\nabla\phi \phi \,dx\nonumber-2\int_{\R^3}\sigma_\varepsilon^L:v_\varepsilon\otimes \D\phi \phi \,dx+\int_{\R^3}\D_i u_\varepsilon^kW_{p^k_j}\D_jv_\varepsilon^i\phi^2\,dx
	\\
	&+2\int_{\R^3}\D_i u_\varepsilon^k W_{p^k_j}v^i_\varepsilon\D_j\phi\phi \,dx.\nonumber
	\end{align*}
	Multiplying \eqref{G-L3} by $\frac{1}{\g_1}h_\varepsilon\phi^2$, integrating over $\R^3$ and using the similar calculations in \eqref{2.5}, one has
	\begin{align}\label{2.49}
	&\frac{d}{dt}\int_{\R^3}\left(W(u_\varepsilon,\nabla u_\varepsilon)+\frac{1}{4\varepsilon^2}(1-|u_\varepsilon|^2)^2\right)\phi^2\,dx+\frac{1}{\g_1}\int_{\R^3}|h_\varepsilon|^2\phi^2\,dx\\
	=&-\int_{\R^3}h_\varepsilon^T\Omega_\varepsilon u_\varepsilon\phi^2\,dx+\frac{\g_2}{\g_1}\int_{\R^3}h_\varepsilon^T A_\varepsilon u_\varepsilon\phi^2\,dx+\int_{\R^3}v_\varepsilon\cdot \D u_\varepsilon\cdot h_\varepsilon\phi^2\,dx\nonumber\\
	&-2\int_{\R^3}\p_tu^i_\varepsilon W_{p^i_\alpha}\D_\alpha\phi\phi \,dx.\nonumber
	\end{align}
Summing \eqref{2.49} with \eqref{2.48} and using Young's inequality yield
	\begin{align}\label{2.50}
	&\frac{d}{dt}\int_{\R^3}\left(\frac{|v_\varepsilon|^2}{2}+W(u_\varepsilon,\D u_\varepsilon)+\frac{1}{4\varepsilon^2}(1-|u_\varepsilon|^2)^2\right)\phi^2\,dx\\
	&+\frac{\alpha_4}{2}\int_{\R^3}|\D v_\varepsilon|^2\phi^2\,dx+\alpha_1\int_{\R^3}|u_\varepsilon^T A_\varepsilon u_\varepsilon|^2\phi^2\,dx+\beta\int_{\R^3}|A_\varepsilon u_\varepsilon|^2\phi^2\,dx\nonumber\\
	\leq& C\int_{\R^3}(|\D u_\varepsilon|^2|\D v_\varepsilon|+|v_\varepsilon||\D u_\varepsilon||h_\varepsilon|)\phi^2\,dx+C\int_{\R^3}|\p_t u_\varepsilon||\D u_\varepsilon||\D \phi||\phi|\,dx\nonumber\\
	&+C\int_{\R^3}(|\sigma_\varepsilon^L|+|\D u_\varepsilon|^2+|\D v_\varepsilon|+|P_\varepsilon-c_\varepsilon(t)|)|v_\varepsilon||\D\phi||\phi|\,dx\nonumber\\
	\leq& \eta\int_{\R^3}|\p_t u_\varepsilon|^2\phi^2\,dx+\frac{\alpha_4}{4}\int_{\R^3}|\D v_\varepsilon|^2\phi^2\,dx+C\int_{\R^3}(|\D u_\varepsilon|^2+|v_\varepsilon|^2)|\D\phi|^2 \,dx\nonumber\\
	&+C\int_{\R^3}(|\D u_\varepsilon|^2+|v_\varepsilon|^2)|\D u_\varepsilon|^2\phi^2\,dx+C\int_{\R^3}(|P_\varepsilon-c_\varepsilon(t)|+|v_\varepsilon|^2)|v_\varepsilon||\D\phi||\phi|dx,\nonumber
	\end{align}
	where $\eta$ will be chosen later and we have used the facts that
	\begin{align*}
	|\sigma^L_\varepsilon|+|h_\varepsilon|\leq C(|A_\varepsilon|+|\Omega_\varepsilon|+|N_\varepsilon|)\leq C(|\D v_\varepsilon|+|\p_t u_\varepsilon|+|v_\varepsilon||\D u_\varepsilon|)
	\end{align*}
	and
	\begin{align*}
	\alpha_4\int_{\R^3}|A_\varepsilon|^2\phi^2\;dx\geq\frac{\alpha_4}{2}\int_{\R^3}|\D v_\varepsilon|^2\phi^2\,dx-C\int_{\R^3}|v_\varepsilon||\D v_\varepsilon||\D\phi||\phi|\,dx
	\end{align*}
	which follows from integration by parts and using \eqref{G-L2}.
	
	In order to bound the term $\int_{\R^3}|\p_t u_\varepsilon|^2\phi^2\;dx$ on the right hand side of \eqref{2.50}, we multiply \eqref{G-L3} by $\partial_tu_\varepsilon^i\phi^2$ and then integrate it over $\R^3$ to obtain
	\begin{align*}
	&-\int_{\R^3}h_\varepsilon\cdot\partial_tu_\varepsilon\phi^2\,dx+\g_1\int_{\R^3}|\partial_tu_\varepsilon|^2\phi^2\,dx\nonumber\\
	=&-\g_1\int_{\R^3}v_\varepsilon\cdot\nabla u_\varepsilon^i\partial_t u_\varepsilon^i\phi^2\,dx+\int_{\R^3}(\g_1\Omega_\varepsilon u_\varepsilon-\g_2 A_\varepsilon u_\varepsilon)\cdot\p_t u_\varepsilon \phi^2 \,dx.\\
	\end{align*}
	It follows from   similar calculations in \eqref{2.5} that
	\begin{align}\label{2.52}
	&\frac{d}{dt}\int_{\R^3}\left(W(u_\varepsilon,\nabla u_\varepsilon)+\frac{1}{4\varepsilon^2}(1-|u_\varepsilon|^2)^2\right)\phi^2\,dx+\g_1\int_{\R^3}|\partial_tu_\varepsilon|^2\phi^2\,dx\\
	=&-\g_1\int_{\R^3}v_\varepsilon\cdot\nabla u_\varepsilon^i\partial_t u_\varepsilon^i\phi^2\,dx-2\int_{\R^3}\p_tu^iW_{p^i_\alpha}\D_\alpha\phi\phi \,dx\nonumber\\
	& +\int_{\R^3}(\g_1\Omega_\varepsilon u_\varepsilon-\g_2 A_\varepsilon u_\varepsilon)\cdot\p_t u_\varepsilon \phi^2 \,dx\nonumber\\
	\leq& \frac{\g_1}{4}\int_{\R^3}|\partial_tu_\varepsilon|^2\phi^2\;dx+C_1\int_{\R^3}|\D v_\varepsilon|^2\phi^2\,dx\nonumber\\
	&+C\int_{\R^3}|v_\varepsilon|^2|\D u_\varepsilon|^2\phi^2+|\D u_\varepsilon|^2|\D\phi|^2\,dx.\nonumber
	\end{align}
	
	To derive the estimate of $\int_{\R^3}|\nabla^2u_\varepsilon|^2\phi^2\;dx$, we multiply \eqref{G-L3} with $\frac{1}{\g_1}\Delta u^i_\varepsilon\phi^2$ and integrate the resulting equation over $\R^3$. Then, one has
	\begin{align}\label{2.53}
	&-\int_{\R^3}\p_t u_\varepsilon\cdot \Delta u_\varepsilon \phi^2\,dx+\frac{1}{\g_1}\int_{\R^3}h_\varepsilon\cdot \Delta u_\varepsilon\phi^2\,dx\\
	=&\int_{\R^3}v_\varepsilon\cdot\D u_\varepsilon\Delta u_\varepsilon\phi^2\,dx+\int_{\R^3}\left(-\Omega_\varepsilon u_\varepsilon+\frac{\g_2}{\g_1}A_\varepsilon u_\varepsilon\right)\cdot\Delta u_\varepsilon\phi^2\,dx.\nonumber
	\end{align}
	It follows from integration by parts that
	\begin{align}\label{2.54}
	&\int_{\R^3}h_\varepsilon\cdot \Delta u_\varepsilon\phi^2\,dx=\int_{\R^3}\left(\D_\alpha W_{p_\alpha^i}-W_{u_\varepsilon^i}+\frac{1}{\varepsilon^2}u_\varepsilon^i(1-|u_\varepsilon|^2)\right)\Delta u_\varepsilon^i\phi^2\,dx\\
	=&\int_{\R^3} W_{p_\alpha^ip_\gamma^j}\D_{\beta\alpha}u_\varepsilon^i\D_{\beta\gamma}u_\varepsilon^j\phi^2\,dx+\int_{\R^3} W_{p_\alpha^iu^j}\nabla_\beta u_\varepsilon^j\D_{\beta\alpha}u_\varepsilon^i\phi^2\,dx\nonumber\\
   &+2\int_{\R^3}W_{p_\alpha^i}(\D_{\beta\alpha}u_\varepsilon^i\D_\beta\phi-\Delta u_\varepsilon^i\D_\alpha\phi)\phi \,dx\nonumber\\
	&-\int_{\R^3}W_{u_\varepsilon^i}\Delta u_\varepsilon^i\phi^2\;dx+\frac{1}{2\varepsilon^2}\int_{\R^3}|\D_\beta(|u_\varepsilon|^2)|^2\phi^2\,dx\nonumber\\
	&-\int_{\R^3}\frac{(1-|u_\varepsilon|^2)}{\varepsilon^2}(|\D u_\varepsilon|^2\phi^2+\D(|u_\varepsilon|^2)\D\phi\phi)\,dx \nonumber
	\end{align}
	and
	\begin{align}\label{2.55}
	-\int_{\R^3}\p_t u_\varepsilon\cdot \Delta u_\varepsilon \phi^2\,dx=\frac{1}{2}\frac{d}{dt}\int_{\R^3}|\D u_\varepsilon|^2\phi^2\,dx+2\int_{\R^3}\p_t u_\varepsilon^i\cdot\D u_\varepsilon^i\cdot\D\phi\phi \,dx.
	\end{align}
	Collecting \eqref{2.53}-\eqref{2.55} and using \eqref{Co2} give
	\begin{align}\label{2.56}
	&\frac{1}{2}\frac{d}{dt}\int_{\R^3}|\nabla u_\varepsilon|^2\phi^2 \,dx+\int_{\R^3}\left(\frac{a}{\g_1}|\D^2u_\varepsilon|^2+\frac{1}{2\varepsilon^2}|\nabla|u_\varepsilon|^2|^2\right)\phi^2\,dx\\
	\leq &\int_{\R^3}v_\varepsilon\cdot\D u_\varepsilon\Delta u_\varepsilon\phi^2\,dx+\int_{\R^3}\left(-\Omega_\varepsilon u_\varepsilon+\frac{\g_2}{\g_1}A_\varepsilon u_\varepsilon\right)\cdot\Delta u_\varepsilon\phi^2\,dx\nonumber\\
	&-2\int_{\mathbb{R}^3}\partial_tu_\varepsilon\nabla u_\varepsilon^i\phi\nabla\phi \,dx-\frac{1}{\g_1}\int_{\R^3}W_{p_\alpha^iu^j}(u_\varepsilon,\nabla u_\varepsilon)\nabla_\beta u_\varepsilon^j\nabla_{\beta\alpha}u_\varepsilon^i\phi^2\,dx\nonumber\\
	&-\frac{2}{\g_1}\int_{\R^3}W_{p_\alpha^i}(u,\nabla u)(\nabla_{\beta\alpha}^2u_\varepsilon^i\nabla_\beta\phi-\Delta u_\varepsilon^i\nabla\phi)\phi \,dx+\frac{1}{\g_1}\int_{\R^3}W_{u_\varepsilon^i}\Delta u_\varepsilon^i\phi^2\,dx\nonumber\\
	&+\frac{1}{\g_1}\int_{\R^3}\frac{(1-|u_\varepsilon|^2)}{\varepsilon^2}(|\D u_\varepsilon|^2\phi^2+\D(|u_\varepsilon|^2)\D\phi\phi)\,dx\nonumber\\
	\leq& \frac{a}{2\g_1}\int_{\R^3}|\D^2 u_\varepsilon|^2\phi^2\,dx+\eta\int_{\R^3}|\p_t u_\varepsilon|^2\phi^2\,dx+C_2\int_{\R^3}|\D v_\varepsilon|^2\,dx\nonumber\\
	&+C\int_{\R^3}(|v_\varepsilon|^2+|\D u_\varepsilon|^2)|\D u_\varepsilon|^2\phi^2\,dx+C\int_{\R^3}|\D u_\varepsilon|^2|\D\phi|^2\,dx,\nonumber
	\end{align}
	where we have used
	\begin{align}\label{estimate1}
	\left|\frac{1-|u_\varepsilon|^2}{\varepsilon^2}\right|\leq C(|\p_t u_\varepsilon|+|v_\varepsilon||\D u_\varepsilon|+|\D v_\varepsilon|+|\D u_\varepsilon|^2+|\D^2 u_\varepsilon|),
	\end{align}
	which follows from \eqref{G-L3} and the assumption $\frac{1}{2}\leq|u_\varepsilon|\leq\frac{3}{2}$.
	
	Multiplying \eqref{2.50} by $C_3:=4\alpha_4^{-1}(C_1+C_2+1)$, summing with \eqref{2.52}, \eqref{2.56}, choosing $\eta=\g_1(4(C_3+1))^{-1}$, and integrating in $[0,s]$, we have \eqref{local energy estimates} following from \eqref{Co2}. By using the same argument with $\phi\equiv 1$, we obtain the desired estimate \eqref{fir-ord}.
\end{proof}

Second order estimates of the Ginzburg-Landau system \eqref{G-L1}-\eqref{G-L3} are given in the following lemma.
\begin{lemma}\label{second order estimates}
	Let $(u_\varepsilon,v_\varepsilon)$ be a strong solution to the system \eqref{G-L1}-\eqref{G-L3} on $\R^3\times (0,T_\varepsilon)$. Assume that $\frac 12\leq |u_\varepsilon|\leq \frac 32$ on $\R^3\times (0,T_\varepsilon)$. Then for any $\phi\in C^\infty_0(\R^3)$ and any $s\in (0,T_\varepsilon)$ we have the estimate
	\begin{align*}
	&\int_{\R^3}\left(|\na v_\varepsilon(x,s)|^2+|\na^2 u_\varepsilon(x,s)|^2+\frac{|\na(|u_\varepsilon(x,s)|^2)|^2}{\varepsilon^2}\right)\phi^2\,dx\alabel{second order estimates.1}
	\\
	&+\int_{0}^{s}\int_{\R^3}\left(|\na^3u_\varepsilon|^2+|\na^2v_\varepsilon|^2+|\na \p_t u_\varepsilon|^2+\frac{|\nabla^2(|u_\varepsilon|^2)|^2}{\varepsilon^2}\right) \phi^2\,dxdt
	\\
	\leq& C\int_{\R^3}\left(|\na^2 u_{0,\varepsilon}|^2+|\na v_{0,\varepsilon}|^2+\frac{|\na|u_{0,\varepsilon}|^2|^2}{\varepsilon^2}\right)\phi^2\,dx
	\\
	&+C\int_{0}^{s}\int_{\R^3}(|\na u_\varepsilon|^2+|v_\varepsilon|^2)(|\na^2 u_\varepsilon|^2+|\na v_\varepsilon|^2+|\p_t u_\varepsilon|^2)\phi^2\,dxdt
	\\
	&+C\int_{0}^{s}\int_{\R^3}(|\na u_\varepsilon|^4+|v_\varepsilon|^4+|\na^2 u_\varepsilon|^2+|\na v_\varepsilon|^2+|\p_t u_\varepsilon|^2)(|\na \phi|^2+|\na^2\phi||\phi|)\,dxdt
	\\
	&+C\int_{0}^{s}\int_{\R^3}|P_\varepsilon-c_\varepsilon(t)|^2(|\na\phi|^2+|\na^2\phi||\phi|)\,dxdt,
	\end{align*}
	where $C$ is a positive constant independent of $\varepsilon$ and $c_\varepsilon(t)\in\R$. In particular, for any $s\in(0,T_\varepsilon)$, we have
	\begin{align*}
	&\int_{\R^3}\left(|\na^2 u_\varepsilon(x,s)|^2+|\na v_\varepsilon(x,s)|^2+\frac{|\na(|u_\varepsilon(x,s)|^2)|^2}{\varepsilon^2}\right)\,dx\alabel{second order estimates.2}\\
	&+\int_{0}^{s}\int_{\R^3}\left(|\na^3u_\varepsilon|^2+|\na^2v_\varepsilon|^2+|\na \p_t u_\varepsilon|^2+\frac{|\nabla^2|u_\varepsilon|^2|^2}{\varepsilon^2}\right) \,dxdt
	\\
	\leq& C\int_{\R^3}\left(|\na^2 u_{0,\varepsilon}|^2+|\na v_{0,\varepsilon}|^2+\frac{|\na|u_{0,\varepsilon}|^2|^2}{\varepsilon^2}\right)\,dx
	\\
	&+C\int_{0}^{s}\int_{\R^3}(|v_\varepsilon|^2+|\na u_\varepsilon|^2)(|\na v_\varepsilon|^2+|\na^2 u_\varepsilon|^2+|\p_t u_\varepsilon|^2)\,dxdt.
	\end{align*}
\end{lemma}
\begin{proof}
	For simplicity, denote
	\begin{align*}
	g_1=:&(|v_\varepsilon|^2+|\na u_\varepsilon|^2)(|\na v_\varepsilon|^2+|\na^2 u_\varepsilon|^2+|\p_t u_\varepsilon|^2+|\nabla u_\varepsilon|^4)\phi^2,
	\\
	g_2=:&(|\na u_\varepsilon|^4+|v_\varepsilon|^4+|\p_t u_\varepsilon|^2+|\na^2 u_\varepsilon|^2+|\na v_\varepsilon|^2)(|\na \phi|^2+|\na^2\phi||\phi|).
	\end{align*}
Multiplying equation \eqref{G-L1} by $\Delta v_\varepsilon\phi^2$, using \eqref{G-L2} and integrating over $\R^3$ yield
\begin{align*}
\frac 12\frac{d}{dt}\int_{\R^3}|\na v_\varepsilon|^2\phi^2\,dx
=&\int_{\R^3}\sigma_\varepsilon^L:\na\Delta v_\varepsilon\phi^2\,dx-2\int_{\R^3}\p_t v_\varepsilon\na v_\varepsilon\cdot\na \phi\phi\,dx
\\
&-2\int_{\R^3}(P_\varepsilon-c_\varepsilon(t))\Delta v_\varepsilon\na \phi\phi)\,dx\alabel{2.14}
\\
&+\int_{\R^3}\Delta v_\varepsilon\cdot(v_\varepsilon\cdot\na v_\varepsilon\phi^2+2\sigma_\varepsilon^L\cdot\na\phi\phi-\na\cdot\sigma_\varepsilon^E\phi^2)\,dx
\\
=&:I_1+I_2+I_3+I_4.
\end{align*}
By a similar argument to the one in \eqref{2.3}, one has
\begin{align*}
I_1=&\alpha_1\int_{\R^3}  (u_\varepsilon^TA_\varepsilon u_\varepsilon)u_\varepsilon\otimes u_\varepsilon:\Delta A_\varepsilon\phi^2\,dx
+\alpha_4\int_{\R^3}A_\varepsilon:\Delta A_\varepsilon\phi^2\,dx
\\
&+\beta\int_{\R^3}A_\varepsilon u_\varepsilon\otimes u_\varepsilon:\Delta A_\varepsilon \phi^2\,dx+\frac{\g_2}{\g_1}\int_{\R^3}h_\varepsilon^T\Delta A_\varepsilon u_\varepsilon\phi^2 \,dx-\int_{\R^3}h_\varepsilon^T\Delta \Omega_\varepsilon u_\varepsilon\phi^2\,dx\\
=&-\alpha_1\int_{\R^3}|u_\varepsilon^iu_\varepsilon^j\nabla A_{\varepsilon ij}|^2\phi^2\,dx-\alpha_4\int_{\R^3}|\nabla A_\varepsilon|^2\phi^2\,dx-\beta\int_{\R^3}|\nabla A_{\varepsilon ij}u_\varepsilon^j|^2\phi^2\,dx\\
&-\alpha_1\int_{\R^3}A_{\varepsilon kl}\nabla A_{\varepsilon ij}\cdot\nabla(u_\varepsilon^iu_\varepsilon^ju_\varepsilon^ku_\varepsilon^l\phi^2)\,dx-2\alpha_4\int_{\R^3}A_{\varepsilon ij}\nabla A_{\varepsilon ij}\cdot \nabla\phi\phi\,dx\\
&-\beta\int_{\R^3}A_{\varepsilon ik}\nabla(u_\varepsilon^ku_\varepsilon^i\phi^2)\nabla A_{\varepsilon ij}\,dx+\frac{\g_2}{\g_1}\int_{\R^3}h_\varepsilon^T\Delta A_\varepsilon u_\varepsilon\phi^2 \,dx-\int_{\R^3}h_\varepsilon^T\Delta \Omega_\varepsilon u_\varepsilon\phi^2\,dx\\
\leq&-\frac{\alpha_4}{2}\int_{\R^3}|\nabla^2 v_\varepsilon|^2\phi^2\,dx-\alpha_1\int_{\R^3}|u_\varepsilon^iu_\varepsilon^j\nabla A_{\varepsilon ij}|^2\phi^2\,dx-\beta\int_{\R^3}|\nabla A_{\varepsilon ij}u_\varepsilon^j|^2\phi^2\,dx\\
&+\frac{\g_2}{\g_1}\int_{\R^3}h_\varepsilon^T\Delta A_\varepsilon u_\varepsilon\phi^2\,dx-\int_{\R^3}h_\varepsilon^T\Delta \Omega_\varepsilon u_\varepsilon\phi^2\,dx+C\int_{\R^3}g_1+g_2\,dx.
\end{align*}
In the view of \eqref{G-L1}, one has
\begin{align*}
I_2=&-2\int_{\R^3}(\na \cdot(\sigma_\varepsilon^E+\sigma_\varepsilon^L)-(v_\varepsilon\cdot\na)v_\varepsilon-\na P_\varepsilon)\na v_\varepsilon\cdot\na \phi\phi\,dx
\\
\leq &C\int_{\R^3}(|\na^2u_\varepsilon||\na u_\varepsilon||\na v_\varepsilon|+|\na u_\varepsilon|^3|\na v_\varepsilon|)|\na\phi||\phi|\,dx-2\int_{\R^3}\na \cdot\sigma_\varepsilon^L\na v_\varepsilon\cdot\na \phi\phi\,dx
\\
&+C\int_{\R^3}|v_\varepsilon||\na v_\varepsilon|^2|\na \phi||\phi|\,dx+C\int_{\R^3}|P_\varepsilon-c_\varepsilon(t)||\na v_\varepsilon||\na(\na\phi\phi)|\,dx
\\
\leq&\frac{\alpha_4}{16}\int_{\R^3}|\na^2 v_\varepsilon|^2\phi^2\,dx+\eta_1\int_{\R^3}|\na \p_t u_\varepsilon|^2\phi^2\,dx
\\
&+C\int_{\R^3}|P_\varepsilon-c_\varepsilon(t)|^2(|\na^2\phi||\phi|+|\na \phi|^2)\,dx+C\int_{\R^3}g_1+g_2\,dx,
\end{align*}
where $\eta_1$ will be chosen later and we utilized the fact that
\begin{equation}
|\na\cdot \sigma^L_\varepsilon|\leq C(|\na u_\varepsilon||\na v_\varepsilon|+|\na^2 v_\varepsilon|+|\na \p_t u_\varepsilon|+|\p_t u_\varepsilon||\na u_\varepsilon|+|v_\varepsilon||\na^2u_\varepsilon|).
\end{equation}
It follows from Young's inequality that
\begin{align*}
I_3+I_4\leq&\frac{\alpha_4}{16} \int_{\R^3}|\na^2 v_\varepsilon|^2\phi^2\,dx+C \int_{\R^3}|P_\varepsilon-c_\varepsilon(t)|^2|\na \phi|^2\,dx+C \int_{\R^3}g_1+g_2\,dx.
\end{align*}
Substituting $I_1,I_2,I_3$ and $I_4$ into \eqref{2.14} leads us to
\begin{align*}
& \frac{1}{2}\frac{d}{dt}\int_{\R^3}|\na v_\varepsilon|^2\phi^2\,dx+\frac{3\alpha_4}{8}\int_{\R^3}|\na^2 v_\varepsilon|^2\phi^2\,dx\alabel{2.19}
\\
\leq& \eta_1\int_{\R^3}|\na \p_t u_\varepsilon|^2\phi^2\,dx+C\int_{\R^3}|P_\varepsilon-c_\varepsilon(t)|^2(|\na^2\phi||\phi|+|\na \phi|^2)\,dx
\\
&+\frac{\g_2}{\g_1}\int_{\R^3}h_\varepsilon^T\Delta A_\varepsilon u_\varepsilon\phi^2\,dx-\int_{\R^3}h_\varepsilon^T\Delta \Omega_\varepsilon u_\varepsilon\phi^2\, dx+C\int_{\R^3}g_1+g_2\,dx.
\end{align*}
Differentiating \eqref{G-L3} in $x_\beta$, multiplying the resulting equation by $\frac{1}{\g_1}\D_\beta h_\varepsilon\phi^2$ and integrating over $\R^3$, we have
\begin{align*}\alabel{2.20}
&-\int_{\R^3}\p_t \na_\beta  u_\varepsilon\cdot \na_\beta h_\varepsilon\phi^2\,dx+\frac{1}{\gamma_1}\int_{\R^3} |\na h_\varepsilon|^2\phi^2\,dx
\\
=&\frac{\g_2}{\g_1}\int_{\R^3}\D_\beta(A_\varepsilon u_\varepsilon)\cdot\D_\beta h_\varepsilon\phi^2 \,dx-\int_{\R^3} \D_\beta(\Omega_\varepsilon u_\varepsilon)\cdot\D_\beta h_\varepsilon\phi^2 \,dx\\
&+\int_{\R^3}\D_\beta(v_\varepsilon\cdot\D u_\varepsilon)\cdot\D_\beta h_\varepsilon\phi^2\,dx=:J_1+J_2+J_3.
\end{align*}
For $J_1$, it follows from integration by parts that
\begin{align*}\alabel{2.21}
J_1=&-\frac{\g_2}{\g_1}\int_{\R^3}h_\varepsilon^T\Delta A_\varepsilon u_\varepsilon\phi^2 \,dx-\frac{\g_2}{\g_1}\int_{\R^3}\D_{\beta}A_\varepsilon\D_{\beta}(u_\varepsilon\phi^2)  h_\varepsilon  \,dx\\
&+\frac{\g_2}{\g_1}\int_{\R^3}A_{\varepsilon} \D_{\beta} u_\varepsilon\cdot \D_{\beta}h_\varepsilon \phi^2\,dx\\
\leq&-\frac{\g_2}{\g_1}\int_{\R^3}h_\varepsilon^T\Delta A_\varepsilon u_\varepsilon\phi^2 \,dx+\frac{\alpha_4}{16}\int_{\R^3}|\D^2 v_\varepsilon|^2\phi^2\,dx
\\
&+\frac{1}{8\g_1}\int_{\R^3}|\D h_\varepsilon|^2\phi^2\,dx
+C\int_{\R^3}g_1+g_2\,dx,
\end{align*}
where we have used that $|h_\varepsilon|\leq C(|\p_t u_\varepsilon|+|v_\varepsilon||\na u_\varepsilon|+|\na v_\varepsilon|)$. Similarly, we have
\begin{align}\label{2.22}
J_2\leq&\int_{\R^3}h_\varepsilon^T\Delta \Omega_\varepsilon u_\varepsilon\phi^2\,dx+\frac{\alpha_4}{16}\int_{\R^3}|\D^2 v_\varepsilon|^2\phi^2\,dx
\\
&+\frac{1}{4\g_1}\int_{\R^3}|\D h_\varepsilon|^2\phi^2\,dx+C\int_{\R^3}g_1+g_2\,dx.\nonumber
\end{align}
It follows from Young's inequality that
\begin{align}\label{2.23}
J_3\leq\frac{1}{8\g_1}\int_{\R^3}|\D h_\varepsilon|^2\phi^2\,dx+C\int_{\R^3}g_1+g_2\,dx.
\end{align}
Next we estimate $J_0:=-\int_{\R^3}\p_t\D_\beta u_\varepsilon\cdot\D_\beta h_\varepsilon \phi^2\;dx$ on the left hand side of \eqref{2.20}. By using \eqref{M-F1}, one has
\begin{align*}\alabel{2.24}
J_0=&-\int_{\R^3}  \na_\beta \p_t u_\varepsilon^i\na^2_{\alpha\beta} W_{p^i_\alpha}\phi^2\,dx+\int_{\R^3}\na_\beta \p_t u_\varepsilon^i\na_\beta W_{u_\varepsilon^i}\phi^2\,dx
\\
&-\int_{\R^3}\na_\beta \p_t u_\varepsilon^i\na_\beta\left(\frac{u_\varepsilon^i(1-|u_\varepsilon|^2)}{\varepsilon^2}\right)\phi^2\,dx=:J_{0,1}+J_{0,2}+J_{0,3}.
\end{align*}
It follows from integration by parts and \eqref{Co2} that
\begin{align*}
J_{0,1}=&\int_{\R^3}\p_t\D_{\alpha\beta}^2u_\varepsilon^iW_{p_\alpha^ip_\g^j}\D_{\beta\g}^2u_\varepsilon^j\phi^2\,dx+\int_{\R^3}\p_t\D_{\alpha\beta}^2u_\varepsilon^iW_{p_\alpha^iu_\varepsilon^j}\D_{\beta}u_\varepsilon^j\phi^2\,dx\alabel{2.25}
\\
&+2\int_{\R^3}\p_t\D_{\beta}u_\varepsilon^i\na_\beta W_{p_\alpha^i}\na_\alpha\phi\phi\,dx
\\
=&\frac{1}{2}\frac{d}{dt}\int_{\R^3}\D_{\alpha\beta}^2u_\varepsilon^iW_{p_\alpha^ip_\g^j}\D_{\beta\g}^2u_\varepsilon^j\phi^2\,dx+\frac12\int_{\R^3}\D_{\beta}u_\varepsilon^i\D_{\alpha}(\p_tW_{p_\alpha^ip_\g^j}\D_{\beta\g}^2u_\varepsilon^j\phi^2)\,dx\nonumber
\\
&-\int_{\R^3}\p_t\D_{\beta}u_\varepsilon^i\D_{\alpha}\left(W_{p_\alpha^iu_\varepsilon^j}\D_{\beta}u_\varepsilon^j\phi^2\right)\,dx+2\int_{\R^3}\p_t\D_{\beta}u_\varepsilon^i\na_\beta W_{p_\alpha^i}\na_\alpha\phi\phi\,dx
\\
\geq& \frac{1}{2}\frac{d}{dt}\int_{\R^3}\D_{\alpha\beta}^2u_\varepsilon^iW_{p_\alpha^ip_\g^j}\D_{\beta\g}^2u_\varepsilon^j\phi^2\,dx-{\eta_1}\int_{\R^3}|\p_t\D u_\varepsilon|^2\phi^2\,dx
\\
&-{\eta}_2\int_{\R^3}|\D^3u_\varepsilon|^2\phi^2\,dx-C\int_{\R^3}g_1+g_2\,dx.
\end{align*}
For $J_{0,2}$, Young's inequality implies
\begin{align}\label{2.26}
J_{0,2}\geq-{\eta_1}\int_{\R^3}|\p_t\D u_\varepsilon|^2\phi^2\,dx-C\int_{\R^3}g_1+g_2\,dx.
\end{align}
For $J_{0,3}$, noting that \eqref{estimate1} and the inequality
\begin{align}\label{estimate2}
\left|\frac{\nabla (1-|u_\varepsilon|^2)}{\varepsilon^2}\right|\leq &C(|\nabla\p_tu_\varepsilon|+|\nabla^2v_\varepsilon|+|\nabla ^3u_\varepsilon|)+C|v_\varepsilon|(|\nabla u_\varepsilon|^2+|\nabla^2u_\varepsilon|)\\
&+C|\nabla u_\varepsilon|(|\nabla u_\varepsilon|^2+|\nabla^2u_\varepsilon|+|\p_t u_\varepsilon|+|\nabla v_\varepsilon|)\nonumber
\end{align}
from \eqref{G-L3} and the assumption $\frac{1}{2}\leq|u_\varepsilon|\leq\frac{3}{2}$,
 one has
\begin{align*}
J_{0,3}=&\int_{\R^3}\frac{\D_{\beta}(|u_\varepsilon|^2)}{\varepsilon^2}\p_t(u_\varepsilon^i\D_{\beta}u_\varepsilon^i)\phi^2\,dx+\int_{\R^3}\frac{\D_\beta(1-|u_\varepsilon|^2)}{\varepsilon^2}\p_t u_\varepsilon^i\D_\beta u_\varepsilon^i\phi^2\,dx\alabel{2.27}
\\
&-\int_{\R^3}\frac{1-|u_\varepsilon|^2}{\varepsilon^2}\D_{\beta}u_\varepsilon^i\p_t\D_{\beta}u_\varepsilon^i\phi^2\,dx\\
\geq&\frac{1}{4\varepsilon^2}\frac{d}{dt}\int_{\R^3}|\D(|u_\varepsilon|^2)|^2\phi^2\,dx-\frac{\alpha_4}{16}\int_{\R^3}|\nabla^2 v_\varepsilon|\phi^2\,dx\\
&-\eta_1\int_{\R^3}|\p_t\D u_\varepsilon|^2\phi^2\,dx-{\eta}_2\int_{\R^3}|\D^3u_\varepsilon|^2\phi^2\,dx-C\int_{\R^3}g_1+g_2\,dx.\nonumber
\end{align*}
Plugging \eqref{2.25}, \eqref{2.26} and \eqref{2.27} into \eqref{2.24}, one has
\begin{align}\label{2.31}
J_0\geq&\frac{d}{dt}\int_{\R^3}\left(\frac{1}{2}\D_{\alpha\beta}^2u_\varepsilon^iW_{p_\alpha^ip_\g^j}\D_{\beta\g}^2u_\varepsilon^j+\frac{|\D(|u_\varepsilon|^2)|^2}{4\varepsilon^2}\right)\phi^2\,dx\\
&-\int_{\R^3}\left(3{\eta_1}|\p_t\D u_\varepsilon|^2+2{\eta}_2|\D^3u_\varepsilon|^2\right)\phi^2\,dx-C\int_{\R^3}g_1+g_2\,dx.\nonumber
\end{align}
Substituting \eqref{2.21}, \eqref{2.22}, \eqref{2.23}, \eqref{2.31} into \eqref{2.20}, we have
\begin{align*}
&\frac{d}{dt}\int_{\R^3}\left(\frac12\D_{\alpha\beta}^2u_\varepsilon^iW_{p_\alpha^ip_\g^j}\D_{\beta\g}^2u_\varepsilon^j+\frac{1}{4\varepsilon^2}|\D|u_\varepsilon|^2|^2\right)\phi^2\,dx+\frac{1}{2\g_1}\int_{\R^3}|\D h_\varepsilon|^2\phi^2 \,dx\\
\leq&\frac{3\alpha_4}{16}\int_{\R^3}|\D^2 v_\varepsilon|^2\phi^2 \,dx+3\eta_1\int_{\R^3}|\D\p_t u_\varepsilon|^2\phi^2 \,dx+2{\eta}_2\int_{\R^3}|\D^3 u_\varepsilon|^2 \phi^2\,dx\nonumber\\
&-\frac{\g_2}{\g_1}\int_{\R^3}h_\varepsilon^T\Delta A_\varepsilon u_\varepsilon\phi^2\,dx+\int_{\R^3}h_\varepsilon^T\Delta \Omega_\varepsilon u_\varepsilon\phi^2\,dx+C\int_{\R^3}g_1+g_2\,dx.
\end{align*}
which, summing with \eqref{2.19}, integrating over $[0,s]$ and then using \eqref{Co2}, yields
\begin{align*}\alabel{2.33}
&\int_{\R^3}\left(\frac{1}{2}|\D v_\varepsilon(x,s)|^2+\frac{a}{2}|\D^2 u_\varepsilon(x,s)|^2
+\frac{1}{4\varepsilon^2}|\D(|u_\varepsilon(x,s)|^2)|^2\right)\phi^2\,dx\\
&+\int_{0}^{s}\int_{\R^3}\left(\frac{\alpha_4}{8}|\D^2v_\varepsilon|^2 +\alpha_1|u^i_\varepsilon\D A_{\varepsilon ij}u^j_\varepsilon|^2+\beta|\D A_{\varepsilon ij}u^j_\varepsilon|^2\right)\phi^2\,dxdt\nonumber\\
\leq &4\eta_1\int_{0}^{s}\int_{\R^3}|\D\p_t u_\varepsilon|^2\phi^2 \,dxdt+2{\eta}_2\int_{0}^{s}\int_{\R^3}|\D^3 u_\varepsilon|^2\phi^2 \,dxdt+C\int_{0}^{s}\int_{\R^3}g_1+g_2\,dxdt\nonumber\\
&+C\int_{\R^3}\left(|\D v_{0,\varepsilon}|^2+|\D^2 u_{0,\varepsilon}|^2
+\frac{|\D(|u_{0,\varepsilon}|^2)|^2}{\varepsilon^2}\right)\phi^2\,dx.\nonumber
\end{align*}
In view of \eqref{2.33}, it remains to estimate $|\D\p_t u_\varepsilon|^2$ and $|\D^3 u_\varepsilon|^2.$
Differentiating \eqref{G-L3} in $x_\beta$, multiplying the resulting equation by $\nabla_\beta\partial_tu_\varepsilon\phi^2$ and integrating over $\mathbb R^3$, we have
\begin{align}\label{2.34}
&-\int_{\R^3}\D_\beta h_\varepsilon\cdot\D_\beta \p_t u_\varepsilon\phi^2 \,dx+\g_1\int_{\R^3}|\D\p_t u_\varepsilon|^2\phi^2\,dx\\
=&\int_{\R^3}\left(\D_\beta(\g_1\Omega_\varepsilon u_\varepsilon-\g_2A_\varepsilon u_\varepsilon)-\g_1\D_\beta(v_\varepsilon\cdot \D u_\varepsilon)\right)\cdot \D_\beta\p_t u_\varepsilon \phi^2\,dx\nonumber\\
\leq& \eta_2\int_{\R^3}|\D\p_t u_\varepsilon|^2\phi^2 \,dx+{C}_1\int_{\R^3}|\D^2 v_\varepsilon|^2 \phi^2\,dx+C\int_{\R^3}g_1+g_2\,dx.\nonumber
\end{align}
Plugging \eqref{2.31} into \eqref{2.34} yields
\begin{align}\label{2.35}
&\frac{d}{dt}\int_{\R^3}\left(\frac12\D_{\alpha\beta}^2u_\varepsilon^iW_{p_\alpha^ip_\g^j}\D_{\beta\g}^2u_\varepsilon^j+\frac{|\D(|u_\varepsilon|^2)|^2}{4\varepsilon^2}\right)\phi^2\,dx+{\g_1}\int_{\R^3}|\D\p_t u_\varepsilon|^2\phi^2\,dx
\\
\leq\int_{\R^3}(3&\eta_1|\D\p_t u_\varepsilon|^2+3\eta_2|\D^3u_\varepsilon|^2)\phi^2\,dx+{C}_1\int_{\R^3}|\D^2 v_\varepsilon|^2 \phi^2\,dx+C\int_{\R^3}g_1+g_2\,dx.\nonumber
\end{align}
On the other hand,
taking a derivative $\D_{\beta}$ of \eqref{G-L3},  multiplying the resulting equation by $\frac{1}{\g_1}\D_\beta\Delta u_\varepsilon\phi^2$ and integrating over $\R^3$,  we have
\begin{align}\label{2.36}
&\frac{1}{2}\frac{d}{dt}\int_{\R^3}|\D^2u_\varepsilon|^2\phi^2\,dx+\frac{1}{\g_1}\int_{\R^3}\D_{\beta}h_\varepsilon\cdot\D_{\beta}\Delta u_\varepsilon\phi^2 \,dx\\
=&\int_{\R^3}\left(\D_{\beta}(v_\varepsilon\cdot\D u_\varepsilon-\Omega_\varepsilon u_\varepsilon)+\frac{\g_2}{\g_1}\D_{\beta}(A_\varepsilon u_\varepsilon)\right)\cdot\D_\beta\Delta u_\varepsilon\phi^2 \,dx\nonumber\\
\leq&\frac{a}{8\g_1}\int_{\R^3}|\D^3u_\varepsilon|^2\phi^2\,dx+{C}_2\int_{\R^3}|\D^2 v_\varepsilon|^2\phi^2 \,dx+C\int_{\R^3}g_1+g_2\,dx.\nonumber
\end{align}
For the term $K_0:=\int_{\R^3}\D_{\beta}h_\varepsilon\cdot\D_{\beta}\Delta u_\varepsilon\phi^2 \;dx$,  it follows from \eqref{M-F1} that
\begin{align}\label{2.37}
K_0=&\int_{\R^3}\D_{\beta\alpha}^2W_{p_\alpha^i}\D_{\beta}\Delta u_\varepsilon^i \phi^2\,dx-\int_{\R^3}\D_{\beta}W_{u_\varepsilon^i}\D_{\beta}\Delta u_\varepsilon^i\phi^2 \,dx\\
&+\int_{\R^3}\D_{\beta}\left(\frac{1-|u_\varepsilon|^2}{\varepsilon^2}u_\varepsilon^i\right)\D_{\beta}\Delta u_\varepsilon^i\phi^2\,dx=:K_{0,1}+K_{0,2}+K_{0,3}.\nonumber
\end{align}
Note that
\begin{align*}
\nabla^2_{\g \b}W_{p_{\a}^i} (u_\varepsilon,\nabla u_\varepsilon)
&=W_{p_l^jp_{\a}^i}(u_\varepsilon,\nabla u_\varepsilon) \nabla^3_{\b \g l}
u^j_\varepsilon+W_{u_\varepsilon^jp_{\a}^i}(u_\varepsilon,\nabla u_\varepsilon) \nabla^2_{\g\b}
u^j_\varepsilon\nonumber\\
&+W_{u_\varepsilon^kp_l^jp_{\a}^i}(u_\varepsilon,\nabla u_\varepsilon)\nabla_\g u^k_\varepsilon\nabla_{\b l}^2
u^j_\varepsilon+W_{u_\varepsilon^ku_\varepsilon^j p_{\a}^i}(u_\varepsilon,\nabla u_\varepsilon)\nabla_\g u^k_\varepsilon\nabla_{\b}
u^j_\varepsilon\\
&+W_{p_l^ku_\varepsilon^jp_{\a}^i}(u_\varepsilon,\nabla u_\varepsilon) \nabla_{\g l}^2 u^k_\varepsilon\nabla_\beta u^j_\varepsilon.
\end{align*}
Using integration by parts twice and using \eqref{Co2} yield
\begin{align}\label{2.38}
K_{0,1}\geq&\int_{\R^3}\D_{\beta\g}^2W_{p_\alpha^i}\D_{\alpha\beta\g}^3u_\varepsilon^i\phi^2\,dx-C\int_{\R^3}|\na W_{p}||\na^3u_\varepsilon||\na\phi||\phi|\,dx\\
\geq& \int_{\R^3}W_{p_l^jp_{\a}^i}(u_\varepsilon,\nabla u_\varepsilon) \nabla^3_{\b \g l}
u^j_\varepsilon \D_{\alpha\beta\g}^3u_\varepsilon^i\phi^2\,dx\nonumber
\\
&-\frac{a}{4}\int_{\R^3}|\D^3u_\varepsilon|^2\phi^2\,dx-C\int_{\R^3}g_1+g_2\,dx\nonumber\\
\geq& \frac{3a}{4}\int_{\R^3}|\D^3u_\varepsilon|^2\phi^2\,dx-C\int_{\R^3}g_1+g_2\,dx.\nonumber
\end{align}
For $K_{0,2}$, it follows from Young's inequality that
\begin{align}\label{2.39}
K_{0,2}\geq -\frac{a}{8}\int_{\R^3}|\D^3u_\varepsilon|^2\phi^2\,dx-C\int_{\R^3}(|\D u_\varepsilon|^2|\D^2u_\varepsilon|^2+|\D u_\varepsilon|^6)\phi^2\,dx.
\end{align}
The term $K_{0,3}$ can be controlled as follows. Since
\begin{align*}
u_\varepsilon\cdot\nabla_\beta\Delta u_\varepsilon=\frac{1}{2}\nabla_\beta\Delta(|u_\varepsilon|^2)-\nabla_\beta(|\nabla u_\varepsilon|^2)-\nabla_\beta u_\varepsilon\cdot \Delta u_\varepsilon,
\end{align*}
we obtain from integration by parts that
\begin{align*}
K_{0,3}=&\int_{\mathbb R^3}\nabla_\beta\left(\frac{1-|u_\varepsilon|^2}{\varepsilon^2}\right)u_\varepsilon^i\nabla_\beta\Delta u_\varepsilon^i \phi^2\,dx+\int_{\mathbb R^3}\frac{1-|u_\varepsilon|^2}{\varepsilon^2}\nabla_\beta u_\varepsilon^i\na_\beta\Delta u_\varepsilon^i\phi^2\,dx
\\
=&\int_{\mathbb R^3}\nabla_\beta\left(\frac{1-|u_\varepsilon|^2}{\varepsilon^2}\right)\left(\frac{1}{2}\nabla_\beta\Delta(|u_\varepsilon|^2)-\nabla_\beta(|\nabla u_\varepsilon|^2)-\nabla_\beta u_\varepsilon\cdot \Delta u_\varepsilon\right)\phi^2\,dx
\\
&+\int_{\mathbb R^3}\frac{1-|u_\varepsilon|^2}{\varepsilon^2}\nabla_\beta u_\varepsilon^i\na_\beta\Delta u_\varepsilon^i\phi^2\,dx\\
=&\int_{\mathbb R^3}\frac{|\nabla^2(|u_\varepsilon|^2)|^2}{2\varepsilon^2}\phi^2 \,dx-\int_{\R^3}\nabla_\beta\left(\frac{1-|u_\varepsilon|^2}{\varepsilon^2}\right)\nabla_{\beta\alpha}^2(|u_\varepsilon|^2)\nabla_\alpha\phi\phi\,dx\\
&-\int_{\mathbb R^3}\nabla_\beta\left(\frac{1-|u_\varepsilon|^2}{\varepsilon^2}\right)\left(\nabla_\beta(|\nabla u_\varepsilon|^2)+\nabla_\beta u_\varepsilon\cdot \Delta u_\varepsilon\right)\phi^2\,dx\\
&+\int_{\mathbb R^3}\frac{1-|u_\varepsilon|^2}{\varepsilon^2}\nabla_\beta u_\varepsilon^i\na_\beta\Delta u_\varepsilon^i\phi^2\,dx
\end{align*}
Then, by using \eqref{estimate1}, \eqref{estimate2} and Young's inequality,  it is clear that
\begin{align}\label{2.42}
K_{0,3}\geq&\int_{\mathbb R^3}\frac{|\nabla^2(|u_\varepsilon|^2)|^2}{2\varepsilon^2} \phi^2\,dx-\frac{a}{8}\int_{\R^3}|\D^3u_\varepsilon|^2\phi^2\,dx-\eta_1\int_{\R^3}|\nabla\p_tu_\varepsilon|^2\phi^2\,dx\\
&- C_2\int_{\R^3}|\nabla^2v_\varepsilon|^2\phi^2\,dx-C\int_{\R^3}g_1+g_2\,dx.\nonumber
\end{align}
Substituting \eqref{2.38}-\eqref{2.39} and \eqref{2.42} into \eqref{2.37}, we have
\begin{align}\label{2.43}
K_0\geq& \frac{a}{2}\int_{\R^3}|\D^3u_\varepsilon|^2\phi^2\,dx+\int_{\mathbb R^3}\frac{|\nabla^2(|u_\varepsilon|^2)|^2}{2\varepsilon^2} \phi^2\,dx\\
&-\eta_1\int_{\R^3}|\nabla\p_tu_\varepsilon|^2\phi^2\,dx- C_2\int_{\R^3}|\nabla^2v_\varepsilon|^2\phi^2\,dx-C\int_{\R^3}g_1+g_2\,dx.\nonumber
\end{align}
Collecting \eqref{2.43} with \eqref{2.36}, one has
\begin{align*}\alabel{2.44}
&\frac{1}{2}\frac{d}{dt}\int_{\R^3}|\D^2u_\varepsilon|^2\phi^2\,dx+\frac{3a}{8\g_1}\int_{\R^3}|\D^3u_\varepsilon|^2\phi^2\,dx+\frac{1}{2\g_1}\int_{\mathbb R^3}\frac{|\nabla^2(|u_\varepsilon|^2)|^2}{\varepsilon^2}\phi^2 \,dx\\
\leq&\eta_1\int_{\R^3}|\na \p_t u_\varepsilon|^2\phi^2\,dx+2{C}_2\int_{\R^3}|\D^2 v_\varepsilon|^2\phi^2 \,dx+C\int_{\R^3}g_1+g_2\,dx.\nonumber
\end{align*}
Summing \eqref{2.44} with \eqref{2.35}, integrating over $[0,s]$ and using \eqref{Co2} yield
\begin{align*}\alabel{2.45}
&\int_{\R^3}\left(\frac{1+a}{2}|\D^2u_\varepsilon(x,s)|^2+\frac{|\D(|u_\varepsilon(x,s)|^2)|^2}{4\varepsilon^2}
\right)\phi^2\,dx
\\
&+\int_{0}^{s}\int_{\R^3}\left(\g_1|\D\p_t u_\varepsilon|^2 +\frac{3a}{8\g_1}|\D^3u_\varepsilon|^2+\frac{1}{2\g_1}\frac{|\nabla^2(|u_\varepsilon|^2)|^2}{\varepsilon^2}\right) \phi^2\,dxdt\nonumber
\\
\leq&C\int_{\R^3}\left(|\D^2 u_{0,\varepsilon}|^2
+\frac{|\D(|u_{0,\varepsilon}|^2)|^2}{\varepsilon^2}\right)\phi^2\,dx\nonumber
\\
&+4\eta_1\int_{0}^{s}\int_{\R^3}|\na\p_tu_\varepsilon|^2\phi^2\,dxdt+3\eta_2\int_{0}^{s}\int_{\R^3}|\na^3 u_\varepsilon|^2\phi^2\,dxdt
\\
&+({C}_1+2{C}_2)\int_{0}^{s}\int_{\R^3}|\D^2v_\varepsilon|^2\phi^2\,dxdt+C\int_{0}^{s}\int_{\R^3}g_1+g_2\,dxdt.
\end{align*}
Multiplying \eqref{2.33} by $ C_3=8\alpha_4^{-1}({C}_1+2 C_2)$, summing with \eqref{2.45}, and then choosing small constants $\eta_1=\g_1(8( C_3+1))^{-1}$ and ${\eta}_2=a(8\g_1(2 C_3+3))^{-1}$, we obtain
\begin{align}\label{2.46}
&\int_{\R^3}(|\D v_\varepsilon(x,s)|^2+|\D^2 u_\varepsilon(x,s)|^2+\frac{1}{\varepsilon^2}|\D(|u_\varepsilon(x,s)|^2)|^2)\phi^2\,dx\\
&+\int_{0}^{s}\int_{\R^3}\left(|\D^2v_\varepsilon|^2+|\p_t\D u_\varepsilon|^2+|\D^3u_\varepsilon|^2+\frac{1}{\varepsilon^2}|\nabla^2|u_\varepsilon|^2|^2\right) \phi^2\,dxdt\nonumber
\\
\leq C&\int_{\R^3}(|\D v_{0,\varepsilon}|^2+|\D^2 u_{0,\varepsilon}|^2+\frac{1}{\varepsilon^2}|\D(|u_{0,\varepsilon}|^2)|^2)\phi^2\,dx+C\int_{0}^{s}\int_{\R^3}g_1+g_2\,dxdt.\nonumber
\end{align}
Note that integration by parts and Young's inequality yield
\begin{align*}
&\int_{\R^3}(|v_\varepsilon|^2+|\nabla u_\varepsilon|^2)|\na u_\varepsilon|^4\phi^2\,dx\alabel{FHM L6}\\
=&-\int_{\R^3}(u_\varepsilon^j-b)\na_i(\na_iu_\varepsilon^j|\na u_\varepsilon|^4\phi^2+\na_iu_\varepsilon^j|v_\varepsilon|^2|\nabla u_\varepsilon|^2\phi^2)\,dx
\\
\leq&C\int_{\R^3}(|\na^2 u_\varepsilon||\na u_\varepsilon|^4+|v_\varepsilon||\nabla v_\varepsilon||\nabla u_\varepsilon|^3+|v_\varepsilon|^2|\nabla u_\varepsilon|^2|\nabla^2 u_\varepsilon|)\phi^2\,dx\\
&+C\int_{\R^3}(|\na u_\varepsilon|^5|+|v_\varepsilon|^2|\nabla u_\varepsilon|^3)|\na \phi||\phi|\,dx
\\
\leq&\frac 12\int_{\R^3}(|v_\varepsilon|^2+|\nabla u_\varepsilon|^2)|\na u_\varepsilon|^4\phi^2\,dx+C\int_{\R^3}|\nabla u_\varepsilon|^2(|\nabla^2 u_\varepsilon|^2+|\nabla v_\varepsilon|^2)\phi^2\,dx\\
&+C\int_{\R^3}|v_\varepsilon|^2|\nabla^2u_\varepsilon|^2\phi^2\;dx+C\int_{\R^3}(|v_\varepsilon|^2+|\nabla u_\varepsilon|^2)|\nabla u_\varepsilon|^2|\nabla \phi|^2\,dx.
\end{align*}
Therefore,  this lemma follows from \eqref{2.46} and \eqref{FHM L6}.
\end{proof}

The following lemma gives a local estimate of pressure under a smallness assumption, see \cites{HW, HM} for similar arguments.
\begin{lemma}\label{pressure estimate}
	Let $(u_\varepsilon, v_\varepsilon)$ be a strong solution to \eqref{G-L1}-\eqref{G-L3} in $\R^3\times(0, T_\varepsilon)$ and $\phi$ be a cut-off function satisfying $0\leq \phi\leq 1$, $\text{supp } \phi\subset B_{2R}(x_0)$ for some $x_0\in \R^3$ and $|\na \phi|\leq\frac{C}{R}$. For any $s\in (0, T_\varepsilon)$, assume that $\frac 12\leq |u_\varepsilon|\leq \frac32$ and
		\begin{equation}\label{pressure estimate 1}
		\sup_{0\leq t\leq s,x_0\in\R^3}\int_{B_R(x_0)}|\na u_\varepsilon(x,t)|^3+|v_\varepsilon(x,t)|^3dx\leq \delta^3.
		\end{equation}
		Then for any $t\in(0,T_\varepsilon)$, there exists a $c_\varepsilon(t)\in \R$ such that the pressure $P_\varepsilon$ satisfies the following estimate
	\begin{align*}
	&\int_{0}^{s}\int_{\R^3}|P_\varepsilon-c_\varepsilon(t)|^2\phi^2\,dxdt\alabel{pressure estimate.1}
	\\
	\leq&C\sup_{x_0\in\R^3}\int_{0}^{s}\int_{B_{R}(x_0)}\frac{\delta^2}{R^2}(|\na u_\varepsilon|^2+|v_\varepsilon|^2)+\delta^2(|\na^2 u_\varepsilon|^2+|\na v_\varepsilon|^2)\,dxdt
	\\
	&+C\sup_{x_0\in\R^3}\int_{0}^{s}\int_{B_{R}(x_0)}(|\p_t u_\varepsilon|^2+|\na v_\varepsilon|^2)\,dxdt.
	\end{align*}
\end{lemma}

\begin{proof}
	Adapting the procedure from Lemma 2.3 \cite{HM}, we take divergence on both sides of \eqref{G-L1} then the pressure $P_\varepsilon$ satisfies the elliptic equation
	\begin{equation*}\alabel{pressure estimateP1.0}
	-\Delta P_\varepsilon=\nabla_{ij}^2[\nabla_iu_\varepsilon^kW_{p_j^k}(u_\varepsilon,\nabla u_\varepsilon)+v_\varepsilon^jv_\varepsilon^i-\sigma_\varepsilon^L]\quad\text{on}~~\R^3\times[0,T_\varepsilon],
	\end{equation*}
	which implies
	$$P_\varepsilon=\mathcal{R}_i\mathcal{R}_j(F^{ij}),\quad F^{ij}=:\nabla_iu_\varepsilon^kW_{p_j^k}(u_\varepsilon,\nabla u_\varepsilon)+v_\varepsilon^jv_\varepsilon^i-\sigma_\varepsilon^L,$$
	where $\mathcal{R}_i$ is the $i$-th Riesz transform on $\R^3$. Then  we have
	\begin{equation}\label{pressure estimateP1.1}
	(P_\varepsilon-c_\varepsilon(t))\phi=\mathcal{R}_i\mathcal{R}_j(F^{ij}\phi)+[\phi,\mathcal{R}_i\mathcal{R}_j](F^{ij})-c_\varepsilon\phi
	\end{equation}for a cut-off function $\phi$,
	where the commutator $[\phi,\mathcal{R}_i\mathcal{R}_j]$ is defined by
	$$[\phi,\mathcal{R}_i\mathcal{R}_j](\cdot)=\phi\mathcal{R}_i\mathcal{R}_j(\cdot)-\mathcal{R}_i\mathcal{R}_j(\cdot\,\phi).$$
	Since $$|F^{ij}|\leq C(|\na u_\varepsilon|^2+|v_\varepsilon|^2+|\p_t u_\varepsilon|+|\na v_\varepsilon|)$$ and the Riesz operator maps $L^q$ into $L^q$ spaces for any $1<q<+\infty$, we have
	\begin{align*}\alabel{pressure estimateP1.2}
	&\int_{0}^{s}\int_{\R^3}|\mathcal{R}_i\mathcal{R}_j(F^{ij}\phi)|^2\,dxdt\\
	\leq C&\int_{0}^{s}\int_{\R^3}(|\na u_\varepsilon|^4+|v_\varepsilon|^4+|\p_t u_\varepsilon|^2+|\na v_\varepsilon|^2)\phi^2\,dxdt
	\\
		\leq C&\delta^2\int_{0}^{s}\int_{B_{2R(x_0)}}|\na^2 u_\varepsilon|^2+|\na v_\varepsilon|^2dxdt+\frac{C\delta^2}{R^2}\int_{0}^{s}\int_{B_{2R(x_0)}}|\na u_\varepsilon|^2+|v_\varepsilon|^2\,dxdt\\
		&+C\int_{0}^{s}\int_{B_{2R(x_0)}}|\p_t u_\varepsilon|^2+|\na v_\varepsilon|^2\,dxdt,
	\end{align*}
	where we have utilized the following estimate that
	\begin{align*}\alabel{pressure estimateP1.3}
	&\int_{0}^{s}\int_{\R^3}(|\na u_\varepsilon|^4+|v_\varepsilon|^4)\phi^2\,dxdt
	\\
	\leq \int_{0}^{s}&\left(\sup_{0\leq t\leq s,x_0\in\R^3}\int_{B_{R(x_0)}}|\na u_\varepsilon|^3+|v_\varepsilon|^3\,dx\right)^{\frac 23}\left(\int_{\R^3}|\na u_\varepsilon\phi|^6+|v_\varepsilon\phi|^6\,dx\right)^{\frac13}dt
	\\
	\leq C\delta^2&\int_{0}^{s}\int_{B_{2R(x_0)}}\hspace{-1em}|\na^2 u_\varepsilon|^2+|\na v_\varepsilon|^2\,dxdt+\frac{C\delta^2}{R^2}\int_{0}^{s}\int_{B_{2R(x_0)}}|\na u_\varepsilon|^2+|v_\varepsilon|^2\,dxdt.
	\end{align*}
	Since $\text{supp}\, \phi\subset B_{2R(x_0)}$, the commutator can be expressed as
	\begin{align*}\alabel{pressure estimateP2.1}
	&[\phi,\mathcal{R}_i\mathcal{R}_j](F^{ij})(x,t)-c_\varepsilon(t)\phi(x)\\
	=&\int_{\R^3}\frac{(\phi(x)-\phi(y))(x_i-y_i)(x_j-y_j)}{|x-y|^5}F^{ij}(y,t)\,dy-c_\varepsilon(t)\phi(x)\\
	=&\int_{B_{4R}(x_0)}\frac{(\phi(x)-\phi(y))(x_i-y_i)(x_j-y_j)}{|x-y|^5}F^{ij}(y,t)\,dy\\
	&\quad+\phi(x)\left[\int_{\R^3\backslash B_{4R}(x_0)}\frac{(x_i-y_i)(x_j-y_j)}{|x-y|^5}F^{ij}(y,t)\,dy-c_\varepsilon(t)\right]\\
	=&:f_1(x,t)+f_2(x,t).
	\end{align*}
	Note that
	\begin{align*}
	|f_1(x,t)|\leq \frac CR\int_{\R^3}\frac{(|\nabla u_\varepsilon|^2+|v_\varepsilon|^2+|\p_t u_\varepsilon|+|\na v_\varepsilon|)\chi_{B_{4R(x_0)}}}{|x-y|^2}\,dy
	\end{align*}
	and the Hardy-Littlewood-Sobolev inequality holds by (c.f. \cite{HM})
	$$\|I_\alpha(f)\|_{L^q(\R^n)}\leq C\|f\|_{L^r(\R^n)},\quad\frac{1}{q}=\frac{1}{r}-\frac{\alpha}{n}, $$
	where $I_\alpha(f)=:\int_{\R^n}\frac{f(y)}{|x-y|^{n-\alpha}}\,dy$.
	Then it follows from H\"{o}lder's inequality and standard covering arguments that
	\begin{align*}\alabel{pressure estimateP2.2}
	&\int_{0}^{s}\int_{\R^3}|f_1(x,s)|^2\;dxdt\leq CR^{-2}\int_{0}^{s}\|(F^{ij})\chi_{B_{4R(x_0)}}\|_{L^\frac{6}{5}(\R^3)}^2\,dt
	\\
	\leq \frac C{R^2}&\int_{0}^{s}\|(|\na u_\varepsilon|+|v_\varepsilon|)\chi_{B_{4R}(x_0)}\|_{L^3(\R^3)}^2\|(|\na u_\varepsilon|+|v_\varepsilon|)\chi_{B_{4R}(x_0)}\|_{L^2(\R^3)}^2\,dt
	\\
	+ \frac C{R^2}&\int_{0}^{s}\|\chi_{B_{4R}(x_0)}\|^2_{L^3(\R^3)}\|(|\p_t u_\varepsilon|+|\na v_\varepsilon|)\chi_{B_{4R}(x_0)}\|_{L^2(\R^3)}^2\,dt
	\\
	&\hspace{-2.5em}\leq\frac{C\delta^2}{R^2}\int_{0}^{s}\int_{B_{4R}(x_0)}|\na u_\varepsilon|^2+|v_\varepsilon|^2\,dxdt+C\int_{0}^{s}\int_{B_{4R}(x_0)}|\p_t u_\varepsilon|^2+|\na v_\varepsilon|^2\,dxdt,
	\end{align*}
	where $\chi_{B_{4R}(x_0)}(x)=1$ for $x\in B_{4R}(x_0)$ and 0 otherwise. As in Lemma 3.2 of \cite{HW}, to estimate the term involving $f_2(x,t)$, we choose
	$$c_\varepsilon(t)=\int_{\R^3\backslash B_{4R}(x_0)}\frac{(x_{0i}-y_i)(x_{0j}-y_j)}{|x_0-y|^5}F^{ij}(y,t)\,dy,$$
	which is finite for any approximation data $(v, u,\p_t u)\in \dot{H}^1(\R^3)\times \dot{H}^2(\R^3)\times L^2(\R^3)$. Then, one has
	\begin{align*}
	|f_2(x,t)|\leq \left|CR\phi(x)\int_{\R^3\backslash B_{4R}(x_0)}\frac{(|\nabla u_\varepsilon|^2+|v_\varepsilon|^2+|\p_t u_\varepsilon|+|\na v_\varepsilon|)(y)}{|x_0-y|^4}\,dy\right|,
	\end{align*}
	due to the fact (c.f. \cite{Ste}) that
	\begin{align*}
	\left|\frac{(x_{i}-y_i)(x_{j}-y_j)}{|x-y|^5}-\frac{(x_{0i}-y_i)(x_{0j}-y_j)}{|x_0-y|^5}\right|\leq C \frac{|x_0-x|}{|x_0-y|^4}.
	\end{align*}
	Upon relabeling and using H\"{o}lder's inequality we observe
	\begin{align*}
	&\int_{0}^{s}\int_{\R^3}|f_2(z,s)|^2\,dzdt
	\alabel{pressure estimateP2.3}
	\\
	\leq&CR^{5}\int_{0}^{s}\left|\sum^\infty_{k=4}\frac{C}{(kR)^4}\int_{B_{(k+1)R(x_0)}\backslash B_{kR(x_0)}}F^{ij}(x,t)\,dx\right|^2 \,dt
	\\
	\leq&CR^{5}\int_{0}^{s}\sum^\infty_{k=4}\frac{C}{(kR)^8}\int_{B_{(k+1)R(x_0)}\backslash B_{kR(x_0)}}|F^{ij}|^2\; dx\cdot|B_{(k+1)R}\backslash B_{kR}| \, dt
	\\
	\leq&C\sup_{x_0\in\R^3}\int_{0}^{s}\sum^\infty_{k=4}k^{-4}\int_{B_{R(x_0)}}|F^{ij}|^2\,dxdt
	\\
	\leq&\sup_{x_0\in\R^3}\int_{0}^{s}\int_{B_{R}(x_0)}\frac{C\delta^2}{R^2}(|\na u_\varepsilon|^2+|v_\varepsilon|^2)+C\delta^2(|\na^2 u_\varepsilon|^2+|\na v_\varepsilon|^2)\,dxdt
	\\
	&+C\sup_{x_0\in\R^3}\int_{0}^{s}\int_{B_{R}(x_0)}(|\p_t u_\varepsilon|^2+|\na v_\varepsilon|^2)\,dxdt,
	\end{align*}
	where the last step follows from \eqref{pressure estimateP1.3}.
	Now combine \eqref{pressure estimateP1.2}, \eqref{pressure estimateP2.2} and \eqref{pressure estimateP2.3}, then apply standard covering arguments to complete the proof.
\end{proof}

\section{Local existence}
In this section, we prove the local well-posedness of the general Ericksen-Leslie system \eqref{E-L1}-\eqref{E-L3} with initial data $(v_0,u_0)\in H^1(\R^3,\R^3)\times H_b^2(\R^3,S^2)$ by using the Ginzburg-Landau approximation approach. The following lemma states the local well-posedness of the Ginzburg-Landau approximation system \eqref{G-L1}-\eqref{G-L3}.
\begin{lemma}\label{existence of G-L}
	Let $(v_{0,\varepsilon},u_{0,\varepsilon})$ be the initial data satisfying
	\begin{align}\label{initial data}
	u_{0,\varepsilon}-b\in H^2(\R^3),\quad v_{0,\varepsilon}\in H^1(\R^3),\quad\text{div }\, v_{0,\varepsilon}=0,
	\end{align}
	where $b$ is a constant unit vector. Then there exists a constant $T_\varepsilon>0$ such that the  system \eqref{G-L1}-\eqref{G-L3} with initial data $(v_{0,\varepsilon},u_{0,\varepsilon})$ has a unique strong solution $(v_\varepsilon,u_\varepsilon)$ in $\R^3\times(0,T_\varepsilon)$ satisfying
	\begin{align*}
	&v_\varepsilon\in L^\infty(0,T_\varepsilon,H^1(\R^3))\cap L^2(0,T_\varepsilon;H^2(\R^3)),\quad \p_tu_\varepsilon\in L^2(\R^3\times(0,T_\varepsilon)),\\
	&u_\varepsilon\in L^\infty(0,T_\varepsilon,H_b^2(\R^3))\cap L^2(0,T_\varepsilon;H_b^3(\R^3)),\quad \p_tu_\varepsilon\in L^2(0,T_\varepsilon;H^1(\R^3)).
	\end{align*}
\end{lemma}
\begin{proof}
	The local well-posedness of the system \eqref{G-L1}-\eqref{G-L3} follows from the standard contraction mapping principle. We omit the proof and refer to the similar argument in Lin-Liu \cite{LL2}.
\end{proof}

The following proposition gives the uniform estimates of solutions $(u_\varepsilon,v_\varepsilon)$ in Lemma \ref{existence of G-L}.
\begin{prop}\label{prop2.1}
	Let $(v_{0,\varepsilon},u_{0,\varepsilon})$ be the initial data in Lemma \ref{existence of G-L} satisfying
	\begin{equation}\label{initial data 2}
	\frac{3}{4}\leq |u_{0,\varepsilon}|\leq \frac{5}{4},\quad\|u_{0,\varepsilon}-b\|_{H^2}^2+\|v_{0,\varepsilon}\|_{H^1}^2+\varepsilon^{-2}\|(1-|u_{0,\varepsilon}|^2)\|_{H^1}^2\leq M,
	\end{equation}
	for some constant $M$ independent of $\varepsilon$. Then there is a uniform constant $T_M$ in $\varepsilon$ such that the system \eqref{G-L1}-\eqref{G-L3} with initial data $(u_{0,\varepsilon},v_{0,\varepsilon})$ has a unique strong solution $(u_\varepsilon,v_\varepsilon)$ in $\R^3\times[0,T_M]$ satisfying
	\begin{equation}\label{2.66}
	\frac{3}{4}\leq |u_\varepsilon|\leq\frac{5}{4}\quad\text{in  }\R^3\times[0,T_M]
	\end{equation}
	and
	\begin{align}\label{2.67}
	&\sup_{0\leq t\leq T_M}\left(\|v_\varepsilon\|_{H^1}^2+\|\D u_\varepsilon\|_{H^1}^2+\varepsilon^{-2}\|(1-|u_\varepsilon|^2)\|_{H^1}^2\right)+\|\D v_\varepsilon\|_{L^2(0,T_M;H^1)}^2\\
	&+\|\D^2 u_\varepsilon\|_{L^2(0,T_M;H^1)}^2+\|\p_t u_\varepsilon\|_{L^2(0,T_M;H^1)}^2+\varepsilon^{-2}\|\D(|u_\varepsilon|^2)\|_{L^2(0,T_M;H^1)}^2\leq C_M,\nonumber
	\end{align}
	provided $\varepsilon\leq \varepsilon_M$, where $T_M$, $\varepsilon_M$ and $C_M$ are positive constants only depend on $M$.
\end{prop}
\begin{proof}
	For the initial data $(u_{0,\varepsilon},v_{0,\varepsilon})$ satisfying \eqref{initial data} and \eqref{initial data 2}, it follows from the Sobolev embedding $H^1(\R^3)\hookrightarrow L^6(\R^3)$ with the constant $C_1$ that for any $0<\delta<1$, there exists a $R_0=:\frac{\delta^2}{C_1^2L^2M}$ such that
	\begin{align*}\alabel{2.68}
&\sup_{x_0\in\R^3}\int_{B_{R_0}(x_0)}|\na u_{0,\varepsilon}|^3+|v_{0,\varepsilon}|^3+\frac{|1-|u_{0,\varepsilon}|^2|^3}{\varepsilon^3}\,dx
\\
\leq&CR_0^{\frac 32}\left(\sup_{x_0\in\R^3}\int_{B_{R_0}(x_0)}|\na u_{0,\varepsilon}|^6+|v_{0,\varepsilon}|^6+\frac{|1-|u_{0,\varepsilon}|^2|^6}{\varepsilon^6}\,dx\right)^{\frac 12}
\\
\leq &R_0^{\frac 32}C^3_1 \left(\|\na  u_{0,\varepsilon}\|_{H^1(\R^3)}+\|  v_{0,\varepsilon}\|_{H^1(\R^3)}+\varepsilon^{-1}\left\|(|1-|u_{0,\varepsilon}|^2)\right\|_{H^1(\R^3)}\right)^3
\\
\leq &C_1^3(R_0M)^{\frac 32}=\frac{\delta^3}{L^3}.
\end{align*}
	where $L>1$ is an absolute constant independent of $\varepsilon$ and $M$ to be chosen.
	By Lemma \ref{existence of G-L}, there exists a unique strong solution to the system \eqref{G-L1}-\eqref{G-L3}  in $\R^3\times[0,T_\varepsilon]$ with initial data $(u_{0,\varepsilon},v_{0,\varepsilon})$. Since the solution $(u_\varepsilon,v_\varepsilon)$ is continuous, which follows from the Sobolev inequality, there is a time $T_\varepsilon^1\in(0,T_\varepsilon]$ such that
	\begin{align}\label{2.69}
	\frac{1}{2}\leq |u_\varepsilon|\leq \frac{3}{2} \quad\text{in   } \R^3\times[0,T_\varepsilon^1)
	\end{align}
	and
	\begin{align}\label{2.70}
	\sup_{0\leq t\leq T_\varepsilon^1,x\in\R^3}\int_{B_{R_0}(x)}(|v_\varepsilon|^3+|\D u_\varepsilon|^3)\,dx\leq\delta^3.
	\end{align}
	 Now we shall show that \eqref{2.69} and \eqref{2.70} hold true for some uniform time $T$ by using the local energy estimate \eqref{local energy estimates}, \eqref{second order estimates.1} and \eqref{pressure estimate.1}. Let $\phi\in C_{0}^\infty(B_{2R_0}(x_0))$ be a cut-off function with $\phi\equiv 1$ on $B_{R_0}(x_0)$ and $|\nabla \phi|\leq \frac{C}{R_0}$ and $|\nabla^2\phi|\leq\frac{C}{R_0^2}$. It follows from \eqref{local energy estimates}  that
	\begin{align*}\alabel{2.74}
	&\sup_{0\leq t\leq T^1_\varepsilon}\frac{1}{R_0}\int_{B_{2R_0}(x_0)} |v_\varepsilon|^2+|\na u_\varepsilon|^2+\varepsilon^{-2}(|1-|u_\varepsilon|^2|^2)\, dx
	\\
	&+\frac{1}{R_0}\int^{T^1_\varepsilon}_0\int_{B_{2R_0}(x_0)}|\na^2u_\varepsilon|^2+|\na v_\varepsilon|^2+|\p_t u_\varepsilon|^2+\varepsilon^{-2}| \na|u_\varepsilon|^2|^2\,dxdt
	\\
	\leq &\frac{C}{R_0}\int_{B_{2R_0}(x_0)}|v_{0,\varepsilon}|^2 +|\na u_{0,\varepsilon}|^2+\frac{|1-|u_{0,\varepsilon}|^2|^2}{\varepsilon^2}\,dx
	\\
	&+\frac{C}{R^3_0}\int^{T^1_\varepsilon}_0\int_{B_{2R_0}(x_0)}  |\na u_\varepsilon|^2+|v_\varepsilon|^2\,dxdt
	\\
	&+\frac{C}{R_0}\int^{T^1_\varepsilon}_0\int_{B_{2R_0}(x_0)}   |\na u_\varepsilon|^4+|v_\varepsilon|^4\,dxdt  \\
	&+\frac{C}{R^2_0}\int^{T^1_\varepsilon}_0\int_{B_{2R_0}(x_0)}|P_\varepsilon-c_\varepsilon(t)||v_\varepsilon|\,dxdt
	\\
	=:&I_1+I_2+I_3+I_4.
	\end{align*}
	For $I_1$, since the ball $B_{2R_0}(x_0)$ can be covered by finitely many number, which is independent of $R_0$,  balls $B_{R_0}(y)$ with $y\in\R^3$, we obtain from H\"{o}lder's inequality, standard covering arguments and \eqref{2.68} that
	\begin{align*}
	I_1\leq  &\frac{C|B_{R_0}|^{1/3}}{R_0}\left(\int_{B_{2R_0}(x_0)}|u_{0,\varepsilon}|^3+ |v_{0,\varepsilon}|^3+\frac{|1-|u_{0,\varepsilon}|^2|^3}{\varepsilon^3}\,dx\right)^{2/3}
	\\
	\leq&C\left(\sup_{y\in\R^3}\int_{B_{R_0}(y)}|u_{0,\varepsilon}|^3+ |v_{0,\varepsilon}|^3+\frac{|1-|u_{0,\varepsilon}|^2|^3}{\varepsilon^3}\,dx\right)^{2/3}=\frac{C\delta^2}{L^2}.
	\end{align*}
	Similarly, using \eqref{2.70}, one has
	\begin{align*}
	I_2\leq&\frac{C|B_{R_0}|^{1/3}}{R_0^3}\int_{0}^{T_\varepsilon^1}\left(\int_{B_{2R_0}(x_0)}|v_\varepsilon|^3+|\nabla u_\varepsilon|^3\right)^{2/3}\,dt\\
	&\leq \frac{C}{R_0^2}\int_{0}^{T_\varepsilon^1}\left(\sup_{y\in \R^3}\int_{B_{R_0}(y)}|v_\varepsilon|^3+|\nabla u_\varepsilon|^3\right)^{2/3}\,dt\leq \frac{C\delta^2T_\varepsilon^1}{R_0^2}.
	\end{align*}
	Similar to \eqref{pressure estimateP1.3}, we employ the Sobolev inequality for $I_3$ and estimate of $I_2$ to compute
	\begin{align*}
	I_3\leq& \frac{C}{R_0}\int^{T^1_\varepsilon}_0\left(\int_{B_{2R_0}(x_0)}|\na u_\varepsilon|^{3}+|v_\varepsilon|^{3}\,dx\right)^\frac 23\left(\int_{B_{2R_0}(x_0)} |\na u_\varepsilon|^{6}+|v_\varepsilon|^{6} \,dx\right)^\frac13\,dt
	\\
	\leq & \frac{C\delta^2}{R_0}\sup_{y\in\R^3}\int^{T^1_\varepsilon}_0  \int_{B_{R_0}(y)} |\na^2 u_\varepsilon|^{2 }+|\na v_\varepsilon|^{2 } +R^{-2}_0 |\na u_\varepsilon|^{2 }+R^{-2}_0|v_\varepsilon|^{2}\,dxdt
	\\
	\leq &\frac{C\delta^2}{R_0}\sup_{y\in\R^3}\int^{T^1_\varepsilon}_0  \int_{B_{R_0}(y)} |\na^2 u_\varepsilon|^{2 }+|\na v_\varepsilon|^{2 } \,dxdt+\frac{C\delta^2T_\varepsilon^1}{R_0^2}.
	\end{align*}
	
	For $I_4$, it follows from Young's inequality, \eqref{pressure estimate.1} and the estimate of $I_2$ that
	\begin{align*}
	I_4\leq& \frac{\delta^2}{R_0}\int^{T^1_\varepsilon}_0\int_{B_{2R_0}(x_0)}|P_\varepsilon-c_{\varepsilon}(s)|^2\,dxdt+\frac{C_\delta}{R^3_0}\int^{T^1_\varepsilon}_0\int_{B_{2R_0}(x_0)}|v_\varepsilon|^2\,dxdt
	\\
	\leq&\frac{C\delta^2}{R_0}\sup_{y\in\R^3}\int_{0}^{T_\varepsilon^1}\int_{B_{R_0}(y)}\frac{\delta^2}{R_0^2}(|\na u_\varepsilon|^2+|v_\varepsilon|^2)+\delta^2(|\na^2 u_\varepsilon|^2+|\na v_\varepsilon|^2)\,dxdt
	\\
	&+\frac{C\delta^2}{R_0}\sup_{y\in\R^3}\int_{0}^{T_\varepsilon^1}\int_{B_{R_0}(y)}(|\p_t u_\varepsilon|^2+|\na v_\varepsilon|^2)\,dxdt+I_2
	\\
	\leq&\frac{C\delta^2}{R_0}\sup_{y\in\R^3}\int_{0}^{T_\varepsilon^1}\int_{B_{R}(y)}(|\p_t u_\varepsilon|^2+|\na v_\varepsilon|^2+|\na^2 u_\varepsilon|^2)\,dxdt+\frac{C\delta^2T_\varepsilon^1}{R_0^2}.
	\end{align*}
Substituting estimates of $I_i$, $i=1,2,3,4,$ into \eqref{2.74} and taking supremum in $x_0$, we obtain from choosing $\delta<1$ such that $C\delta^2<\frac{1}{4}$
	\begin{align*}\alabel{FHMPufbP4.2}
	&\sup_{0\leq t\leq T^1_\varepsilon,x_0\in\R^3}\frac{1}{R_0}\int_{B_{R_0}(x_0)}|\na u_\varepsilon|^2+|v_\varepsilon|^2+\varepsilon^{-2}(1-|u_\varepsilon|^2)^2\,dx
	\\
	&+\frac{1}{2R_0}\sup_{x_0\in\R^3}\int^{T^1_\varepsilon}_0\int_{B_{R_0}(x_0)}|\na^2u_\varepsilon|^2+|\na v_\varepsilon|^2+|\p_t u_\varepsilon|^2+\frac{| \na|u_\varepsilon|^2|^2}{\varepsilon^2}\,dxdt
	\\
	&\leq\frac{C\delta^2}{L^2}+\frac{C\delta^2T_\varepsilon^1}{R_0^2}.
	\end{align*}
	On the other hand, it follows from \eqref{second order estimates.1} that
	\begin{align*}\alabel{FHMPufbP5.0}
	&R_0\sup_{0\leq t\leq T^1_\varepsilon}\int_{B_{2R_0}(x_0)}(|\na v_\varepsilon(x,t)|^2+|\na^2 u_\varepsilon(x,t)|^2+\frac{1}{\varepsilon^2}|\na(|u_\varepsilon|^2)(x,t)|^2)\,dx
	\\
	&+R_0\int^{T^1_\varepsilon}_0\int_{B_{2R_0}(x_0)}|\na^2v_\varepsilon|^2+|\na \p_t u_\varepsilon|^2+|\na^3u_\varepsilon|^2+\frac{1}{\varepsilon^2}|\nabla^2|u_\varepsilon|^2|^2\,dxdt
	\\
	\leq& CR_0\int_{B_{2R_0}(x_0)}(|\na v_{0,\varepsilon}|^2+|\na^2 u_{0,\varepsilon}|^2+\frac{1}{\varepsilon^2}|\na(|u_{0,\varepsilon}|^2)(x,t)|^2)\,dx
	\\
	&+CR_0\int^{T^1_\varepsilon}_0\int_{B_{2R_0}(x_0)}(|v_\varepsilon|^2+|\na u_\varepsilon|^2)(|\na v_\varepsilon|^2+|\na^2 u_\varepsilon|^2+|\p_t u_\varepsilon|^2)\,dxdt
	\\
	&+\frac{C}{R_0}\int^{T^1_\varepsilon}_0\int_{B_{2R_0}(x_0)}|\na u_\varepsilon|^4+|v_\varepsilon|^4+|\p_t u_\varepsilon|^2+|\na^2 u_\varepsilon|^2+|\na v_\varepsilon|^2\,dxdt
	\\
	&+\frac{C}{R_0}\int^{T^1_\varepsilon}_0\int_{B_{2R_0}(x_0)}|P_\varepsilon-c_\varepsilon(t)|^2\,dxdt
	=:I_5+I_6+I_7+I_8.
	\end{align*}
	Then using the definition of $R_0$ from \eqref{2.68} and the initial condition \eqref{initial data 2} we have
	$$I_5\leq CM R_0\leq \frac{C\delta^2}{C^2_1L^2}.$$
	For $I_6$, we utilize the similar argument as the estimate of $I_3$ to obtain
	\begin{align*}
	I_6\leq&CR_0\delta^2\int^{T^1_\varepsilon}_0\int_{B_{2R_0(x_0)}}\left(|\na^2v_\varepsilon|^2+|\na \p_t u_\varepsilon|^2+|\na^3u_\varepsilon|^2\right)\,dxdt
	\\
	&+\frac{C\delta^2}{R_0}\int^{T^1_\varepsilon}_0\int_{B_{2R_0(x_0)}} (|\na v_\varepsilon|^2+|\na^2 u_\varepsilon|^2+|\p_t u_\varepsilon|^2)\,dxdt
	\\
	\leq&C\delta^2R_0\sup_{x_0\in\R^3}\int^{T^1_\varepsilon}_0\int_{B_{R(x_0)}}\left(|\na^2v_\varepsilon|^2+|\na \p_t u_\varepsilon|^2+|\na^3u_\varepsilon|^2\right)\,dxdt
	+\frac{C\delta^2}{L^2}+\frac{C\delta^2T^1_\varepsilon}{R_0^2},
	\end{align*}
	where \eqref{FHMPufbP4.2} is used in the last step.
	By the estimate of $I_3$ and \eqref{FHMPufbP4.2}, it is clear that
	$$I_7\leq\frac{C\delta^2}{L^2}+\frac{C\delta^2T^1_\varepsilon}{R_0^2}.$$
	For the pressure term $I_8$, it follows from \eqref{pressure estimate.1}, \eqref{FHMPufbP4.2} and the estimate of $I_2$ that
	\begin{align*}
	I_8\leq&\frac{C}{R_0}\sup_{x_0\in\R^3}\int^{T^1_\varepsilon}_0\int_{B_{R_0}(x_0)}\frac{\delta^2}{R_0^2}(|\na u_\varepsilon|^2+|v_\varepsilon|^2)+\delta^2(|\na^2 u_\varepsilon|^2+|\na v_\varepsilon|^2)\,dxdt
	\\
	&+\frac{C}{R_0}\sup_{x_0\in\R^3}\int^{T^1_\varepsilon}_0\int_{B_{R_0}(x_0)}(|\p_t u_\varepsilon|^2+|\na v_\varepsilon|^2)\,dxdt
	\leq\frac{C\delta^2}{L^2}+\frac{C\delta^2T^1_\varepsilon}{R_0^2}.
	\end{align*}
	Substituting estimates of $I_i$, $i=5,6,7,8$ into \eqref{FHMPufbP5.0} and taking supremum in $x_0$, we have
	\begin{align*}
	\alabel{FHMPufbP5.1}
	&R_0\sup_{0\leq t\leq T^1_\varepsilon,x_0\in\R^3}\int_{B_{R_0}(x_0)}(|\na v_\varepsilon(x,t)|^2+|\na^2 u_\varepsilon(x,t)|^2+\frac{1}{\varepsilon^2}|\na(|u_\varepsilon|^2)(x,t)|^2)\,dx
	\\
	&+R_0\sup_{x_0\in \R^3}\int^{T^1_\varepsilon}_0\int_{B_{R_0}(x_0)}|\na^2v_\varepsilon|^2+|\na \p_t u_\varepsilon|^2+|\na^3u_\varepsilon|^2+\frac{1}{\varepsilon^2}|\nabla^2|u_\varepsilon|^2|^2\,dxdt
	\\
	\leq&\frac{C\delta^2}{C^2_1L^2}+\frac{C\delta^2}{L^2}+\frac{C\delta^2T^1_\varepsilon}{R_0^2}.
	\end{align*}
	Therefore, using the Gagliardo-Nirenberg interpolation inequality, \eqref{FHMPufbP5.0} and \eqref{FHMPufbP5.1}, we obtain that
	\begin{align*}
	&\sup_{0\leq t\leq T^1_\varepsilon,x\in\R^3}\int_{B_{R_0}(x_0)} |\na u_\varepsilon|^3+|v_\varepsilon|^3\,dx
	\\
	\leq & C\sup_{0\leq t\leq T^1_\varepsilon,x\in\R^3}\left(\frac{1}{R_0}\int_{B_{R_0}(x_0)}|\na u_\varepsilon|^2+|v_\varepsilon|^2\,dx\right)^{3/2}
	\\
	&+C\sup_{0\leq t\leq T^1_\varepsilon,x\in\R^3}\left(R_0\int_{B_{R_0}(x_0)}|\na^2 u_\varepsilon|^2+|\na v_\varepsilon|^2\,dx\right)^{3/2}
	\\
	\leq&\left(\frac{C_2\delta^2}{L^2}+\frac{C_3\delta^2T^1_\varepsilon}{R_0^2}\right)^{3/2}\leq\frac{\delta^3}{2}
	\end{align*}
	from choosing $L=2\sqrt{C_2+1}$ and $T_\varepsilon^1\leq \sigma R_0^2$ with $\sigma \leq 2C_3^{-1}$. Hence, \eqref{2.70} is verified up to the uniform time $T_M=\sigma R_0^2=:CM^{-2}$.  It remains to verify \eqref{2.69} on $[0,T_M]$ for sufficiently small $\varepsilon$. First, It follows from Lemma \ref{local estimate} that, for any $s\in(0,T_M)$,
	\begin{align*}
	&\int_{\R^3}\left( |v_\varepsilon(x,s)|^2+|\na u_\varepsilon(x,s)|^2+\frac{|1-|u_\varepsilon(x,s)|^2|^2}{\varepsilon^2}\right) \,dx
	\\
	&+\int^s_0\int_{\R^3}\left(|\na v_\varepsilon|^2+|\na^2u_\varepsilon|^2+|\p_t u_\varepsilon|^2+\frac{| \na|u_\varepsilon|^2|^2}{\varepsilon^2}\right)\,dxdt
	\\
	\leq&C\int_{\R^3} |\na u_{0,\varepsilon}|^2+|v_{0,\varepsilon}|^2+\frac{|1-|u_{0,\varepsilon}|^2|^2}{\varepsilon^2} \,dx+C\int^s_0\int_{\R^3}(|v_\varepsilon|^2+|\na u_\varepsilon|^2)|\na u_\varepsilon|^2\,dxdt
	\\
	\leq &CM+C\sum^\infty_i\int^s_0\left(\int_{B_{R_0}(x_i)} (|v_\varepsilon|^{3}+|\na u_\varepsilon|^{3})\, dx\right)^{2/3}\left(\int_{B_{R_0}(x_i)}|\na u_\varepsilon|^6\, dx\right)^{ 1/3}\,dt
	\\
	\leq &CM+C\delta ^{2}\sum^\infty_i\int^s_0\int_{B_{R_0}(x_i)}|\na^2 u_\varepsilon|^2+R_0^{-2}|\na u_\varepsilon|^{2} \,dxdt
	\\
	\leq &CM+C\delta^2\int^s_0 \int_{\R^3}|\na^2 u_\varepsilon|^2\,dxdt+C\frac{\delta ^{2}T^1_\varepsilon}{R^2_0}\sup_{0\leq t\leq s}\int_{\R^3}|\na u_\varepsilon|^{2}\,dx,
	\end{align*}
	Choosing $\delta$ and $T^1_\varepsilon$ sufficiently small such that $C\delta^2\leq 1/2$ and noting $s\leq T_M=\sigma R^2_0$ for $\sigma$ sufficiently small, we have
	\begin{align*}\alabel{2.72}
	&\sup_{0\leq t\leq s}\int_{\R^3}\left(|v_\varepsilon|^2+|\D u_\varepsilon|^2+\varepsilon^{-2}(1-|u_\varepsilon|^2)^2\right)(\cdot,t)\,dx\\
	&+\int_{0}^{s}\int_{\R^3}\left(|\D v_\varepsilon|^2+|\p_t u_\varepsilon|^2+|\D^2u_\varepsilon|^2+\varepsilon^{-2}|\na(|u_\varepsilon|)^2|^2\right)\,dxdt
	\leq CM.\nonumber
	\end{align*}
	 Then, from Lemma \ref{second order estimates}, we derive
	\begin{align*}
	&\int_{\R^3}|\na^2 u_\varepsilon|^2+|\na v_\varepsilon|^2+\frac{|\na|u_\varepsilon|^2|^2}{\varepsilon^2}\,dx
	\\
	&+\int^{s}_0\int_{\R^3}\left(|\na^2v_\varepsilon|^2+|\na \p_t u_\varepsilon|^2+|\na^3u_\varepsilon|^2+\frac{1}{\varepsilon^2}|\na^2|u_\varepsilon|^2|^2\right)\,dxdt
	\\
	\leq &CM +C\int^{s}_0\int_{\R^3}(|\na u_\varepsilon|^2+|v_\varepsilon|^2)(|\na v_\varepsilon|^2+|\na^2 u_\varepsilon|^2+|\p_t u_\varepsilon|^2)\,dxdt
	\\
	\leq &CM+C\delta^2\int^{s}_0\int_{\R^3}\left(|\na^2v_\varepsilon|^2+|\na \p_t u_\varepsilon|^2+|\na^3u_\varepsilon|^2\right)\,dxdt
	\\
	&+\frac{C\delta^2}{R^2_0}\int^{s}_0\int_{\R^3} (|\na v_\varepsilon|^2+|\na^2 u_\varepsilon|^2+|\p_t u_\varepsilon|^2)\,dxdt.
	\end{align*}
	Noting that $C\delta^2<\frac 12$, and using \eqref{2.72} we find
	\begin{align*}
	&\int_{\R^3}|\na^2 u_\varepsilon|^2+|\na v_\varepsilon|^2+\frac{|\na(|u_\varepsilon|)^2|^2}{\varepsilon^2}\,dx\alabel{FHMPufbP6.1}
	\\
	&+\int^{s}_0\int_{\R^3}\left(|\na^2v_\varepsilon|^2+|\na \p_t u_\varepsilon|^2+|\na^3u_\varepsilon|^2+\frac{1}{\varepsilon^2}|\na^2|u_\varepsilon|^2|^2\right)\,dxdt
	\\
	\leq &CM\left(1+\frac{\delta^2}{R^2_0}\right).
	\end{align*}
	Therefore, we obtain from the Gagliardo--Nirenberg interpolation that
	\begin{align*}
	&\|1-|u_\varepsilon|^2\|_{L^\infty(\R^3)}\leq C(\varepsilon^2 M)^{\frac18}(\|\nabla^2 u_\varepsilon\|^{\frac34}_{L^2(\R^3)}+\|\na u_\varepsilon\|^{\frac32}_{L^4(\R^3)})
	\\
	\leq&\varepsilon^{\frac14} [C^4(M^5+\sqrt{(M^2+M^5)^2})]^{\frac14}\leq\frac{9}{16},
	\end{align*}
	for all $\varepsilon<\varepsilon_M=:\frac{3^5}{2^6C^4(M^5+M^2)}$ which gives \eqref{2.69}. In the view of \eqref{2.72} and \eqref{FHMPufbP6.1}, we have proved the assertion \eqref{2.67}.
	\end{proof}

Now we can give the proof of local existence of strong solutions to \eqref{E-L1}-\eqref{E-L3} stated in Theorem \ref{thm1}.

\begin{proof}[\bf Proof of Theorem \ref{thm1}]
	By Proposition \ref{prop2.1}, there exist two positive constants $T_0$   and  $\varepsilon_\ast$  independent of $\varepsilon$ such that  for any $\varepsilon\leq \varepsilon_\ast$, the strong solutions $(u_\varepsilon,v_\varepsilon)$ to \eqref{G-L1}-\eqref{G-L3} satisfy
	\begin{align*}
	&(v_\varepsilon,u_\varepsilon)\text{ in } L^\infty(0,T_0;H^1(\R^3)\times H^2_b(\R^3))\cap L^2(0,T_0;H^2(\R^3)\times H^3_b(\R^3)),\\
	&(\p_t v_\varepsilon,\p_tu_\varepsilon)\text{ in } L^2(0,T_0;L^2(\R^3)\times H^1(\R^3))
	\end{align*}
and  \eqref{2.67} holds. It is clear that multiplying $(u_\varepsilon-b)$ with \eqref{G-L3} and using estimates in Proposition \ref{prop2.1}, we  find $\|(u_\varepsilon-b)\|_{L^\infty(0,T_0;L^2(\R^3))}<C$.

Since the pressure $P_\varepsilon$ satisfies
	\eqref{pressure estimateP1.1}, it follows from using the elliptic estimate and the Sobolev inequality that
	\begin{align*}
	\int_{0}^{T_0}\int_{\R^3}|P_\varepsilon|^2\,dxdt\leq \int_{0}^{T_0}\int_{\R^3}|\na u_\varepsilon|^4+|v_\varepsilon|^4+|\p_t u_\varepsilon|^2+|\na v_\varepsilon|^2\,dxdt\leq C
	\end{align*}
	and
	\begin{align*}
	&\int_{0}^{T_0}\int_{\R^3}|\D P_\varepsilon|^2\,dxdt\leq C\int_{0}^{T_0}\int_{\R^3}\left(|\D\sigma_\varepsilon^E|^2+|\D\sigma_\varepsilon^L|^2+|\D(v_\varepsilon\otimes v_\varepsilon)|^2\right)\,dxdt\\
	\leq& C\int_{0}^{T_0}\|\D u_\varepsilon\|_{L^2}\|\D u_\varepsilon\|_{L^6}\left(\|\D u_\varepsilon\|_{L^6}^2+\|\D v_\varepsilon\|_{L^6}^2+\|\p_tu_\varepsilon\|_{L^6}^2\right)dt\\
	&+C\int_{0}^{T_0}\left(\|v_\varepsilon\|_{L^2}\|v_\varepsilon\|_{L^6}\|\D v_\varepsilon\|_{L^6}^2+\|\D^2 v_\varepsilon\|_{L^2}^2+\|\D\p_t u_\varepsilon\|_{L^2}^2+\|\D u_\varepsilon\|_{L^6}^6\right)dt\leq C.
	\end{align*}
	Then, by  the Aubin-Lions Lemma, there is a subsequence, still denoted by\\ $(v_\varepsilon,u_\varepsilon,P_\varepsilon)$ and a solution $(v,u,p)$ such that for any $R\in(0,\infty)$
	\begin{align*}
	&v_\varepsilon\rightarrow v\,\,\text{in  } L^2(0,T_0;H^1(B_R))\cap C([0,T_0];L^2(B_R))\\
	&v_\varepsilon\rightharpoonup v\,\,\text{in } L^2(0,T_0; H^2(\R^3)),\quad \p_t v_\varepsilon\rightharpoonup \p_t v\,\,\text{in  } L^2(\R^3\times(0,T_0)),\\
	&u_\varepsilon\rightarrow u\,\,\text{in  }L^2(0,T_0;H^2(B_R))\cap C([0,T_0];H^1(B_R)),\\
	&u_\varepsilon\rightharpoonup u\,\,\text{in }L^2(0,T_0;H^3(\R^3)),\quad\p_t u_\varepsilon\rightharpoonup \p_t u\,\,\text{in  }L^2(0,T_0;H^1(\R^3)),\\
	&P_\varepsilon\rightharpoonup P\,\,\text{in }L^2(0,T_0;H^1(\R^3)),
	\end{align*}
	where  $|u|=1$ due to  $\sup_{0\leq t\leq T_0}\int_{\R^3}\frac{(1-|u_\varepsilon|^2)^2}{\varepsilon^2}\leq C$. It can be  checked that $(v,u)$ satisfies \eqref{E-L1}-\eqref{E-L3} based on the above compactness, see \cite{HLX} for more details. Indeed,  \eqref{E-L3} follows from taking cross product with $u_\varepsilon$ twice in \eqref{G-L3} and standard weak convergence argument. The uniqueness of strong solutions to \eqref{E-L1}-\eqref{E-L3} follows from the $L^2$ estimates for the difference between two solutions, we refer to \cite{HX} for more details.

	Next, we check the characterization of the  maximal existence $T^\ast$.  Let  $(u,v)$ be the solution to the Ericksen-Leslie system \eqref{E-L1}-\eqref{E-L3} in $\R^3\times [0, T)$ with $T<T^*$.   Then we have
	\begin{align}
	&\frac{d}{dt}\int_{\R^3}(|\D v|^2+|\D^2 u|^2+\D_{\alpha\beta}^2u_\varepsilon^iW_{p_\alpha^ip_\g^j}\D_{\beta\g}^2u_\varepsilon^j)\,dx\\
&+\int_{\R^3}\left(|\D^2v|^2+|\p_t\D u|^2+|\D^3u|^2\right) \,dx \nonumber\\
	\leq& C\int_{\R^3}(|v|^2+|\D u|^2)(|\D v|^2+|\D^2 u|^2+|\p_t u|^2)\,dx\nonumber
	\end{align}  provided that
\[\sup_{0\leq t\leq T, x\in\R^3}\int_{B_{R}(x)}|\nabla u|^3+|v|^3\,dx\leq \varepsilon_0	
\]
 for some   $\varepsilon_0>0$ and  some $R>0$.

	By a standard open cover $\{ B_{R}(x_i)\}_{i=1}^{\infty}$  of  $\R^3$ (at each $x\in \R^3$, there is at most a fixed  number of  intersection of open balls), we obtain
\begin{align}
&\int_{\R^3}(|\D v(T)|^2+|\D^2 u(T)|^2\;dx+\int_{0}^{T}\int_{\R^3}\left(|\D^2v|^2+|\p_t\D u|^2+|\D^3u|^2\right) \,dxdt\\
\leq& C\hspace{-0.5ex}+\hspace{-0.5ex}C\int_{0}^{T}\sum_{i}\left[\int_{B_{R}(x_i)}\hspace{-0.5ex}(|v|^3+|\D u|^3)dx\right]^\frac{2}{3}\left[\int_{B_R(x_i)}\hspace{-0.5ex}(|\D v|^6+|\D^2 u|^6+|\p_t u|^6)dx\right]^\frac{1}{3}\hspace{-1ex}dt\nonumber\\
\leq& C+C\varepsilon_0^\frac{2}{3}\int_{0}^{T}\int_{\R^3}(|\nabla ^3u|^2+|\nabla^2v|^2+|\nabla\p_t u|^2)+\frac{1}{R^2}(|\nabla ^2u|^2+|\nabla v|^2+|\p_t u|^2)\,dxdt.\nonumber
\end{align}
Choosing $\varepsilon_0$ sufficiently small,  $(u(T^\ast), v(T^\ast))\in H_b^2(\R^3)\times H^1(\R^3)$ and by the local existence result, the solution can be extended passing $T$, so $T^*$ is the maximal existence time.
\end{proof}

\section{Smooth convergence of the Ginzburg-Landau system}
In this section, we prove that the Ginzburg-Landau system smoothly converge to the Ericksen-Leslie system once away from the initial time and  until the maximal existence time.  First, we derive higher order uniform estimates of solutions to the Ginzburg-Landau system. To do that, the following lemma, which are essentially from the Gagliardo–Nirenberg interpolation inequality (c.f. \cite{BKM}, \cite{WZZ}), will be frequently used.
\begin{lemma}\label{pro-com estimates}
	For any index $\alpha,\beta,\g\in \mathbb{N}^3$, it holds that
	\begin{align*}
	&\|\nabla^\a(fg)\|_{L^2}\leq C\sum_{|\g|=|\a|}\left(\|f\|_{L^\infty}\|\D^\g g\|_{L^2}+\|g\|_{L^\infty}\|\D^\g f\|_{L^2}\right),\\
	&\|[\nabla^\a,f]\nabla^\b g\|_{L^2}\leq C\left(\sum_{|\g|=|\a|+|\b|}\|g\|_{L^\infty}\|\nabla^\g f\|_{L^2}+\sum_{|\g|=|\a|+|\b|-1}\|\nabla f\|_{L^\infty}\|\nabla^\g g\|_{L^2}\right),\nonumber
	\end{align*}
	where the commutator $[\nabla^\a,f]\nabla^\b g$ is defined by
	\begin{equation*}
	[\nabla^\a,f]\nabla^\b g=\nabla^\a(f\nabla^\b g)-f\nabla^\a(\nabla^\b g).
	\end{equation*}
\end{lemma}

The following lemma shows that the solution to the Ginzburg-Landau system obtained in Proposition \ref{prop2.1} is uniformly smooth once away from the initial time.
\begin{lemma}\label{higher estimates lemma}
	Let $(v_\varepsilon, u_\varepsilon)$ be the strong solution, obtained in Proposition \ref{prop2.1}, to the system \eqref{G-L1}-\eqref{G-L3} in $\R^3\times[0, T_M]$.
	Then it holds for any $\tau>0$, $s\in(\tau,T_M]$ and any integer $l\geq 0$ that
\begin{align}\label{higher estimates}
&\int_{\R^3}\left(|\D^l v_\varepsilon(s)|^2+|\D^{l+1}u_\varepsilon(s)|^2+\frac{|\D^l(|u_\varepsilon(s)|^2)|^2}{\varepsilon^2}\right)\;dx\\
&+\int_{\tau}^{s}\int_{\R^3}\left(|\D^{l+1}v_\varepsilon|^2+|\D^l\p_t u_\varepsilon|^2+|\D^{l+2}u_\varepsilon|^2+\frac{|\D^{l+1}(|u_\varepsilon|^2)|^2}{\varepsilon^2}\right)\;dxdt\nonumber\\
\leq& C(\tau,l,M,s),\nonumber
\end{align}
where $C$ is a positive constant independent of $\varepsilon$. For  simplicity,  $\D^l$ is denoted as multi-derivatives with index $\alpha$ of order $l$.
\end{lemma}

\begin{proof}
		We prove this lemma by induction. In the view of Proposition \ref{prop2.1}, \eqref{higher estimates} holds for $l=0, 1$. Assume that \eqref{higher estimates} holds for  $l=0,1,\cdots, k$ with $k\geq 1$. Next, we show that \eqref{higher estimates} holds for  $l=k+1$.

		 Firstly, we define the following energy and dissipation  terms
		\begin{equation}
		\mathcal{E}_m:=\|\na^mv_\varepsilon\|_{L^2}^2+\|\na^{m+1}u_\varepsilon\|_{L^2}^2,\quad\mathcal{D}_{m}:=\mathcal{E}_{m+1}+\|\na^m\p_t u_\varepsilon\|_{L^2}^2
		\end{equation}
		for any integer $m$, and
		\begin{align}
		\Lambda_\infty:=\|v_\varepsilon\|_{L^\infty}^2+\|\na u_\varepsilon\|_{L^\infty}^2
		\end{align}
		to simplify notations in the sequel.
		
		Now we prove \eqref{higher estimates} for $l=k+1.$  Applying $\na^\nu$  with index $\nu$ of order $k+1$  to \eqref{G-L1}, multiplying the resulting equation by $\na^\nu v_\varepsilon$, integrating over $\R^3$ and using \eqref{G-L2}, we have
			\begin{align*}\alabel{2.57}
			\frac{1}{2}\frac{d}{dt }\int_{\R^3}|\na^\nu v_\varepsilon|^2  \,dx
			=&-\int_{\R^3}\na^\nu\sigma_{\varepsilon ij}^L\na^\nu \D_jv_\varepsilon^i  \,dx +\int_{\R^3}\na^\nu(v_\varepsilon^j v_\varepsilon^i)\na^\nu\D_j v_\varepsilon^i  \,dx \\
			&+\int_{\R^3}\na^\nu(\D_iu_\varepsilon^\a W_{p_j^\a})\na^\nu\D_j v_\varepsilon^i  \,dx =:I_1+I_2+I_3.\nonumber
			\end{align*}
			To estimate $I_1$, we write
			\begin{equation}
			I_1=-\int_{\R^3}\mathcal{T}_{ij,\nu}^L\na^\nu\na_jv_\varepsilon^i\,dx-\int_{\R^3}\mathcal{R}_{ij,\nu}^L\na^\nu\na_jv_\varepsilon^i\,dx=:I_{1,1}+I_{1,2},
			\end{equation}
			where  $\mathcal{T}_{\varepsilon ij,\nu}^L$ is the highest order derivatives   defined by
			\begin{align*}
			\mathcal{T}_{ij,\nu}^L:=&\alpha_1u_\varepsilon^l\na^\nu A_{\varepsilon lm} u^mu_\varepsilon^iu_\varepsilon^j+\alpha_2\na^\nu N^i_\varepsilon u_\varepsilon^j+\alpha_3u_\varepsilon^i\na^\nu N^j_\varepsilon\\
			&+\alpha_4\na^\nu A_{\varepsilon ij}+\alpha_5\na^\nu A_{\varepsilon il}u_\varepsilon^lu_\varepsilon^j+\alpha_6u_\varepsilon^i\na^\nu A_{\varepsilon jl}u_\varepsilon^l
			\end{align*}
			and the remainder $\mathcal{R}_{ij,\nu}^L$ is given by
			\begin{align*}
			\mathcal{R}_{ij,\nu}^L:=&\alpha_1[\nabla^\nu, u_\varepsilon^lu_\varepsilon^mu_\varepsilon^iu_\varepsilon^j]A_{\varepsilon lm}+\alpha_2[\nabla^\nu,u_\varepsilon^j]N^i_\varepsilon+\alpha_3[\nabla^\nu,u_\varepsilon^i]N^j_\varepsilon\\
			&+\alpha_5[\na^\nu,u_\varepsilon^lu_\varepsilon^j]A_{\varepsilon il}+\alpha_6[\na^\nu,u_\varepsilon^iu_\varepsilon^l]A_{\varepsilon jl}.
			\end{align*}
			Note that $u_\varepsilon^i\na^\nu\Omega_{\varepsilon ij}u_\varepsilon^j=\nabla^\nu A_{\varepsilon ij}\Omega_{\varepsilon ij}=0$.   By a similar argument to one in \eqref{2.3}, we obtain
			\begin{align*}\alabel{2.58a}
						I_{1,1}=&-\int_{\R^3}\left(\alpha_1|u_\varepsilon^i\na^\nu A_{\varepsilon ij} u_\varepsilon^j|^2+\alpha_4|\na^\nu A_\varepsilon |^2+\beta|\na^\nu A_\varepsilon u_\varepsilon|^2\right)  \,dx
						\\
						&+\int_{\R^3}\na^\nu h_\varepsilon ^i\left(\na^\nu\Omega_{\varepsilon ij}u_\varepsilon^j-\frac{\g_2}{\g_1}\na^\nu A_{\varepsilon ij}u_\varepsilon^j\right)  \,dx
						\\
						&+\int_{\R^3}[\na^\nu, u_\varepsilon^l]A_{\varepsilon il}\left(\frac{\g_2^2}{\g_1}\na^{\nu}A_{\varepsilon ij}u_\varepsilon^j-\g_2\na^{\nu}\Omega_{\varepsilon ij}u_\varepsilon^j\right)  \,dx\\
						\leq&-\int_{\R^3}\left(\alpha_1|u_\varepsilon^i\na^\nu A_{\varepsilon ij} u_\varepsilon^j|^2+\alpha_4|\na^\nu A_\varepsilon |^2+\beta|\na^\nu A_\varepsilon u_\varepsilon|^2\right)  \,dx
						\\
						&+\int_{\R^3}\na^\nu h_\varepsilon ^i\left(\na^\nu\Omega_{\varepsilon ij}u_\varepsilon^j-\frac{\g_2}{\g_1}\na^\nu A_{\varepsilon ij}u_\varepsilon^j\right)  \,dx
						\\
						&+\delta_1\|\na^{k+2}v_\varepsilon\|_{L^2}^2+C\|u_\varepsilon\|_{L^\infty}^2\|[\na^\nu, u_\varepsilon]\nabla v_\varepsilon\|_{L^2}^2 \\
						\leq&-\int_{\R^3}\left(\alpha_1|u_\varepsilon^i\na^\nu A_{\varepsilon ij} u_\varepsilon^j|^2+\alpha_4|\na^\nu A_\varepsilon |^2+\beta|\na^\nu A_\varepsilon u_\varepsilon|^2\right)  \,dx
						\\
						&+\int_{\R^3}\na^\nu h_\varepsilon ^i\left(\na^\nu\Omega_{\varepsilon ij}u_\varepsilon^j-\frac{\g_2}{\g_1}\na^\nu A_{\varepsilon ij}u_\varepsilon^j\right)  \,dx
						\\
						&+\delta_1\|\na^{k+2}v_\varepsilon\|_{L^2}^2+C\left(\|v_\varepsilon\|_{L^\infty}^2\|\nabla^{k+2}u_\varepsilon\|_{L^2}^2+\|\nabla u_\varepsilon\|_{L^\infty}^2\|\nabla^{k+1}v_\varepsilon\|_{L^2}^2\right),
						\end{align*}
						where  we have used the estimate from Lemma \ref{pro-com estimates} that
						\begin{align}\label{4-8}
						\|[\na^\nu, u_\varepsilon]\nabla v_\varepsilon\|_{L^2}\leq C(\|v_\varepsilon\|_{L^\infty}\|\nabla^{k+2}u_\varepsilon\|_{L^2}+\|\nabla u_\varepsilon\|_{L^\infty}\|\nabla^{k+1}v_\varepsilon\|_{L^2}).
						\end{align}
						For $I_{1,2}$ involving $\mathcal{R}_{ij,\nu}^L$, we first estimate $\|\mathcal{R}_{ij,\nu}^L\|_{L^2}$.
						By using Lemma \ref{pro-com estimates} several times, we obtain
						\begin{align*}
						\|\alpha_1[\nabla^\nu, u_\varepsilon^lu_\varepsilon^mu_\varepsilon^iu_\varepsilon^j]A_{\varepsilon lm}\|_{L^2}\leq& C\|v_\varepsilon\|_{L^\infty}\|\nabla^{k+2}(u_\varepsilon\#u_\varepsilon\#u_\varepsilon\#u_\varepsilon)\|_{L^2}\\
						&+C\|\nabla(u_\varepsilon\#u_\varepsilon\#u_\varepsilon\#u_\varepsilon)\|_{L^\infty}\|\nabla^{k+1}v_\varepsilon\|_{L^2}\\
						\leq&C(\|v_\varepsilon\|_{L^\infty}\|\nabla^{k+2}u_\varepsilon\|_{L^2}+\|\nabla u_\varepsilon\|_{L^\infty}\|\nabla^{k+1}v_\varepsilon\|_{L^2}),
						\end{align*}
						where the notation $\#$ denotes the multi-linear map with constant coefficients in the sequel. Similarly,
						\begin{align*}
						&\|\alpha_5[\na^\nu,u_\varepsilon^lu_\varepsilon^j]A_{\varepsilon il}+\alpha_6[\na^\nu,u_\varepsilon^iu_\varepsilon^l]A_{\varepsilon jl}\|_{L^2}\\
						\leq& C(\|v_\varepsilon\|_{L^\infty}\|\nabla^{k+2}u_\varepsilon\|_{L^2}+\|\nabla u_\varepsilon\|_{L^\infty}\|\nabla^{k+1}v_\varepsilon\|_{L^2}).
						\end{align*}
						For the commutator involving $N_\varepsilon$ in $\mathcal{R}_{ij,\nu}^L$, we first write
						\begin{align}\label{com-n1}
						[\nabla^\nu,u_\varepsilon^j]N^i_\varepsilon=&[\nabla^\nu,u_\varepsilon^j]\p_t u_\varepsilon^i+[\nabla^\nu,u_\varepsilon^j](v_\varepsilon^l\nabla_lu_\varepsilon^i)-[\nabla^\nu,u_\varepsilon^j](\Omega_{\varepsilon li}u_\varepsilon^l)\\
						=&\sum_{|\mu|=k}\begin{pmatrix}
						\nu \\ \mu
						\end{pmatrix}\na^{\nu-\mu} u_\varepsilon^j\nabla^{\mu}\p_tu_\varepsilon^i+\sum_{|\mu|=0}^{k-1}\begin{pmatrix}
						\nu \\ \mu
						\end{pmatrix}\na^{\nu-\mu} u_\varepsilon^j\nabla^{\mu}\p_tu_\varepsilon^i\nonumber\\
						&+[\nabla^\nu,u_\varepsilon^j]\nabla_l (v_\varepsilon^lu_\varepsilon^i)-[\na^\nu,u_\varepsilon^ju_\varepsilon^l]\Omega_{\varepsilon li}+u_\varepsilon^j[\na^\nu, u_\varepsilon^l]\Omega_{\varepsilon li},\nonumber
						\end{align}
						where we have used \eqref{G-L2} and the fact that
						\begin{align}\label{com123}
						[\na^\nu, f_1](f_2f_3)=[\na^\nu,f_1f_2]f_3-f_1[\nabla^\nu,f_2]f_3
						\end{align}
						for any functions $f_1,f_2$ and $f_3$. Then, we apply Lemma \ref{pro-com estimates} to yield
						\begin{align}\label{com-n2}
						\|[\nabla^\nu,u_\varepsilon^j]\nabla_l (v_\varepsilon^lu_\varepsilon^i)\|_{L^2}\leq& C\left(\|v_\varepsilon u_\varepsilon\|_{L^\infty}\|\na^{k+2}u_\varepsilon\|_{L^2}+\|\nabla u_\varepsilon\|_{L^\infty}\|\nabla^{k+1}(v_\varepsilon u_\varepsilon)\|_{L^2}\right)\\
						\leq& C\|v_\varepsilon\|_{L^\infty}\|\na^{k+2}u_\varepsilon\|_{L^2}+C\|\nabla u_\varepsilon\|_{L^\infty}\|\nabla^{k+1}v_\varepsilon\|_{L^2}\nonumber\\
						&+C\|\nabla u_\varepsilon\|_{L^\infty}\|v_\varepsilon\|_{L^\infty}\|\nabla^{k+1} u_\varepsilon\|_{L^2}\nonumber
						\end{align}
						and
						\begin{align}\label{com-n3}
						&\|-[\na^\nu,u_\varepsilon^ju_\varepsilon^l]\Omega_{\varepsilon li}+u_\varepsilon^j[\na^\nu, u_\varepsilon^l]\Omega_{\varepsilon li}\|_{L^2}\\
						\leq& C(\|v_\varepsilon\|_{L^\infty}\|\nabla^{k+2}u_\varepsilon\|_{L^2}+\|\nabla u_\varepsilon\|_{L^\infty}\|\nabla^{k+1}v_\varepsilon\|_{L^2}).\nonumber
						\end{align}
						It follows from the H\"{o}lder and Sobolev inequalities that
						\begin{align*}\alabel{numu}
						&\left\|\sum_{|\mu|=k}\begin{pmatrix}
						\nu \\ \mu
						\end{pmatrix}\na^{\nu-\mu} u_\varepsilon^j\nabla^{\mu}\p_tu_\varepsilon^i+\sum_{|\mu|=0}^{k-1}\begin{pmatrix}
						\nu \\ \mu
						\end{pmatrix}\na^{\nu-\mu} u_\varepsilon^j\nabla^{\mu}\p_tu_\varepsilon^i\right\|_{L^2}\\
						\leq& C\||\na u_\varepsilon||\na^k\p_t u_\varepsilon|\|_{L^2}+C\sum_{|\mu|=0}^{k-1}\|\na^{\nu-\mu}u_\varepsilon\|_{L^3}\|\na^\mu \p_tu_\varepsilon\|_{L^6}\\
						\leq& C\||\na u_\varepsilon||\na^k\p_t u_\varepsilon|\|_{L^2}+C\sum_{|\mu|=0}^{k-1}\|\na^{\nu-\mu}u_\varepsilon\|_{H^1}\|\na^\mu \na\p_tu_\varepsilon\|_{L^2}\\
						\leq& C\||\na u_\varepsilon||\na^k\p_t u_\varepsilon|\|_{L^2}+C\|\na\p_tu_\varepsilon\|_{L^2}\|\na^{k+2}u_\varepsilon\|_{L^2}+C\|\na\p_tu_\varepsilon\|_{H^{k-1}}\|\na u_\varepsilon\|_{H^k}.
						\end{align*}
						Thus, we obtain
						\begin{align}
						\|\mathcal{R}_{ij,\nu}^L\|_{L^2}\leq &C\|v_\varepsilon\|_{L^\infty}\|\nabla^{k+2}u_\varepsilon\|_{L^2}+C\|\nabla u_\varepsilon\|_{L^\infty}\|\nabla^{k+1}v_\varepsilon\|_{L^2}\\
						&+C\|\na^{k+1}u_\varepsilon\|_{L^2}\|\nabla u_\varepsilon\|_{L^\infty}\|v_\varepsilon\|_{L^\infty}+C\||\na u_\varepsilon||\na^k\p_t u_\varepsilon|\|_{L^2}\nonumber\\
						&+C\|\na\p_tu_\varepsilon\|_{L^2}\|\na^{k+2}u_\varepsilon\|_{L^2}+C\|\na\p_tu_\varepsilon\|_{H^{k-1}}\|\na u_\varepsilon\|_{H^k}.\nonumber
						\end{align}
			Therefore, it follows from Young's inequality that
			\begin{align}\label{2.58b}
			I_{1,2}\leq& \delta_1\|\na^{k+2}v_\varepsilon\|_{L^2}^2+C\||\na u_\varepsilon||\na^k\p_t u_\varepsilon|\|_{L^2}^2+C(\Lambda_\infty+\|\na\p_t u_\varepsilon\|_{L^2}^2)\mathcal{E}_{k+1}\\
			&+C\mathcal{E}_k\Lambda_\infty\|\na u_\varepsilon\|_{L^\infty}^2+C\|\na\p_tu_\varepsilon\|^2_{H^{k-1}}\|\na u_\varepsilon\|^2_{H^k}.\nonumber
			\end{align}
			Hence, we obtain from \eqref{2.58a} and \eqref{2.58b} that
			\begin{align*}\alabel{2.600}
			I_1\leq&-\int_{\R^3}\left(\alpha_1|u_\varepsilon^i\na^\nu A_{\varepsilon ij} u_\varepsilon^j|^2+\alpha_4|\na^\nu A_\varepsilon |^2+\beta|\na^\nu A_\varepsilon u_\varepsilon|^2\right)  \,dx
			\\
			&+\int_{\R^3}\na^\nu h_\varepsilon ^i\left(\na^\nu\Omega_{\varepsilon ij}u_\varepsilon^j-\frac{\g_2}{\g_1}\na^\nu A_{\varepsilon ij}u_\varepsilon^j\right)  \,dx
			\\
			&+2\delta_1\|\na^{k+2}v_\varepsilon\|_{L^2}^2+C\||\na u_\varepsilon||\na^k\p_t u_\varepsilon|\|_{L^2}^2+C(\Lambda_\infty+\|\na\p_t u_\varepsilon\|_{L^2}^2)\mathcal{E}_{k+1}\\
			&+C\mathcal{E}_k\Lambda_\infty\|\na u_\varepsilon\|_{L^\infty}^2+C\|\na\p_tu_\varepsilon\|^2_{H^{k-1}}\|\na u_\varepsilon\|^2_{H^k}.\nonumber
			\end{align*}
		To estimate $I_2,I_3$, we note  that
		\begin{equation*}
		\na_i u_\varepsilon^\a W_{p_j^\a}=u_\varepsilon\#u_\varepsilon\#\na u_\varepsilon\#\na u_\varepsilon+\na u_\varepsilon\#\na u_\varepsilon
		\end{equation*}
	and apply Lemma \ref{pro-com estimates}  to yield
		\begin{align}
		&\|\na^\nu(v_\varepsilon^jv_\varepsilon^i)\|_{L^2}\leq C\|v_\varepsilon\|_{L^\infty}\|\na^{k+1}v_\varepsilon\|_{L^2},\\
		&\|\na^\nu(\na_i u_\varepsilon^\a W_{p_j^\a})\|_{L^2}\leq C(\|\na u_\varepsilon\|_{L^\infty}\|\na^{k+2}u_\varepsilon\|_{L^2}+\|\na u_\varepsilon\|_{L^\infty} \|\na^{k+1} u_\varepsilon\|_{L^2}).
		\end{align}
		Thus, we obtain from Young's inequality that
		\begin{align}\label{2.61}
		I_2+I_3\leq& \delta_1\int_{\R^3}|\na^{k+2}v_\varepsilon|^2\,dx+C\|v_\varepsilon\|_{L^\infty}^2\|\nabla^{k+1}v_\varepsilon\|_{L^2}^2\\
		&+C\|\nabla u_\varepsilon\|_{L^\infty}^2\|\nabla^{k+2}u_\varepsilon\|_{L^2}^2+C\|\na^{k+1} u_\varepsilon\|_{L^2}^2\|\na u_\varepsilon\|_{L^\infty}^4.\nonumber
		\end{align}
			Substituting \eqref{2.600} and \eqref{2.61} into \eqref{2.57}, and using \eqref{Leslie coe-2}, one has
			\begin{align*}\alabel{2.62}
			&\frac{1}{2}\frac{d}{dt }\int_{\R^3}|\na^\nu v_\varepsilon|^2\,dx+\alpha_4\int_{\R^3}|\na^\nu A_\varepsilon |^2  \,dx\\
			\leq&\int_{\R^3}\na^{k+1}h^i\left(\na^{k+1}\Omega_{\varepsilon  ij}u_\varepsilon^j-\frac{\g_2}{\g_1}\na^{k+1}A_{ij}u_\varepsilon^j\right)  \,dx\\
			&+3\delta_1\|\na^{k+2}v_\varepsilon\|_{L^2}^2+C\||\na u_\varepsilon||\na^k\p_t u_\varepsilon|\|_{L^2}^2+C(\Lambda_\infty+\|\na\p_t u_\varepsilon\|_{L^2}^2)\mathcal{E}_{k+1}\\
			&+C\mathcal{E}_k\Lambda_\infty\|\na u_\varepsilon\|_{L^\infty}^2+C\|\na\p_tu_\varepsilon\|^2_{H^{k-1}}\|\na u_\varepsilon\|^2_{H^k}.\nonumber
			\end{align*}
			Applying $\na^\nu$, with index $\nu$ of order $k+1$, to \eqref{G-L3}, multiplying the resulting equation by $\frac{1}{\g_1}\na^\nu h_\varepsilon$ and integrating over $\R^3$ give
			\begin{align*}\alabel{2.63}
			&-\int_{\R^3}\p_t\na^\nu u_\varepsilon^i\na^\nu h_\varepsilon^i  \;dx +\frac{1}{\g_1}\|\na^\nu h_\varepsilon\|_{L^2}^2
			\\
			=&\int_{\R^3}\na^\nu h_\varepsilon^i\left(\frac{\g_2}{\g_1}\na^{\nu}A_{\varepsilon ij}u_\varepsilon^j-\na^{\nu}\Omega_{\varepsilon ij} u_\varepsilon^j\right)  \,dx\\
			&-\int_{\R^3}\left([\na^\nu, u_\varepsilon^j]\Omega_{\varepsilon ij}-\frac{\g_2}{\g_1}[\na^\nu, u_\varepsilon^j]A_{\varepsilon ij}+\na^{\nu}(v_\varepsilon\cdot \D u_\varepsilon^i)\right)\na^{\nu}h_\varepsilon^i  \,dx
			\\
			\leq&\frac{1}{2\g_1}\|\na^\nu h_\varepsilon\|_{L^2}^2  +\int_{\R^3}\na^\nu h_\varepsilon^i\left(\frac{\g_2}{\g_1}\na^{\nu}A_{\varepsilon ij}u_\varepsilon^j-\na^{\nu}\Omega_{\varepsilon ij} u_\varepsilon^j\right)  \,dx\nonumber\\
			&+C(\|v_\varepsilon\|_{L^\infty}^2\|\nabla^{k+2}u_\varepsilon\|_{L^2}^2+\|\nabla u_\varepsilon\|_{L^\infty}^2\|\nabla^{k+1}v_\varepsilon\|_{L^2}^2),
			\end{align*}
			where, in the last step, we have used \eqref{4-8} and
			\begin{align}
			\|\na^\nu(v_\varepsilon\cdot\na u_\varepsilon)\|_{L^2}\leq C(\|v_\varepsilon\|_{L^\infty}\|\nabla^{k+2}u_\varepsilon\|_{L^2}+\|\nabla u_\varepsilon\|_{L^\infty}\|\nabla^{k+1}v_\varepsilon\|_{L^2}).
			\end{align}
		For the term $$J_0:=-\int_{\R^3}\p_t\na^{\nu} u_\varepsilon^i\na^{\nu}h_\varepsilon^i\, dx $$
		on the left hand side of \eqref{2.63}, integration by parts yields
		\begin{align*}
		&-\int_{\R^3}\p_t\na^{\nu} u_\varepsilon^i\na^{\nu}\na_\a W_{p^i_\a}\, dx=\int_{\R^3}\p_t\na^{\nu}\na_\a u_\varepsilon^i\na^{\nu}W_{p^i_\a}\, dx\\
		=&\int_{\R^3}\p_t\na^{\nu-e_\b}\na_{\b\a}^2 u_\varepsilon^i\na^{\nu-e_\b}\left(W_{p^i_\a p^j_\g}\na^2_{\b\g}u_\varepsilon^j+W_{p_\a^iu^j}\na_\b u_\varepsilon^j\right)\,dx\\
		=&\int_{\R^3}\p_t\na^{\nu}\na_{\a} u_\varepsilon^iW_{p^i_\a p^j_\g}\na^{\nu}\na_{\g} u_\varepsilon^j\,dx-\int_{\R^3}\p_t\na^{\nu} u_\varepsilon^i\na_{\a}\left([\na^{\nu-e_\b},W_{p^i_\a p^j_\g}]\na^2_{\b\g}u_\varepsilon^j\right)\,dx\\
		&-\int_{\R^3}\p_t\na^{\nu} u_\varepsilon^i\na^{\nu-e_\b}\na_{\a}\left(W_{p_\a^iu^j}\na_\b u_\varepsilon^j\right)\,dx\\
		=&\frac{1}{2}\frac{d}{dt}\int_{\R^3}\na^{\nu}\na_{\a} u_\varepsilon^iW_{p^i_\a p^j_\g}\na^{\nu}\na_{\g} u_\varepsilon^j\,dx-\frac{1}{2}\int_{\R^3}\na^{\nu}\na_{\a} u_\varepsilon^i\p_tW_{p^i_\a p^j_\g}\na^{\nu}\na_{\g} u_\varepsilon^j\,dx\\
		&-\int_{\R^3}\p_t\na^{\nu} u_\varepsilon^i\na_{\a}\left([\na^{\nu-e_\b},W_{p^i_\a p^j_\g}]\na^2_{\b\g}u_\varepsilon^j+\na^{\nu-e_\b}\left(W_{p_\a^iu^j}\na_\b u_\varepsilon^j\right)\right)\,dx,
		\end{align*}
		where $e_\b$ denotes the index of taking one derivative with respective to $x_\beta$. Therefore, in view of \eqref{M-F1}, we have
		\begin{align}\label{2.81}
	    J_0=&\frac{1}{2}\frac{d}{dt}\int_{\R^3}\na^{\nu}\na_{\a} u_\varepsilon^iW_{p^i_\a p^j_\g}\na^{\nu}\na_{\g} u_\varepsilon^j\,dx -\int_{\R^3}\p_t\na^{\nu}u_\varepsilon^i\na^{\nu}\left(\frac{(1-|u_\varepsilon|^2)}{\varepsilon^2}u_\varepsilon^i\right)  \,dx \\
		&-\frac{1}{2}\hspace{-.5ex}\int_{\R^3}\hspace{-.5ex}\na^{\nu}\na_{\a} u_\varepsilon^i\p_tW_{p^i_\a p^j_\g}\na^{\nu}\na_{\g} u_\varepsilon^jdx\hspace{-.5ex}-\hspace{-.5ex}\int_{\R^3}\hspace{-.5ex}\p_t\na^{\nu} u_\varepsilon^i\na_{\a}\left([\na^{\nu-e_\b}\hspace{-.5ex},W_{p^i_\a p^j_\g}]\na^2_{\b\g}u_\varepsilon^j\right)dx \nonumber\\
		&+\int_{\R^3}\p_t\na^{\nu} u_\varepsilon^i\left(\na^{\nu}W_{u_\varepsilon^i}-\na^{\nu-e_\b}\na_{\a}\left(W_{p_\a^iu_\varepsilon^j}\na_\b u_\varepsilon^j\right))\right)  \,dx\nonumber
		\\
		=&:J_{0,1}+J_{0,2}+J_{0,3}+J_{0,4}+ J_{0,5}.\nonumber
		\end{align}
		Using the fact
		\begin{equation*}
		\na^\nu(|u_\varepsilon|^2)=2\na^\nu u_\varepsilon^iu_\varepsilon^i+\sum_{|\mu|=1}^{k}\begin{pmatrix}
		\nu \\ \mu
		\end{pmatrix}\na^{\nu-\mu} u_\varepsilon^j\nabla^{\mu}u_\varepsilon^i,
		\end{equation*}
		we can rewrite $J_{0,2}$ as
		\begin{align}\label{J02}
		J_{0,2}&=-\int_{\R^3}\left\{\p_t(\na^\nu u_\varepsilon^i u_\varepsilon^i)-\na^\nu u_\varepsilon\p_t u_\varepsilon\right\}\na^\nu\left(\frac{1-|u_\varepsilon|^2}{\varepsilon^2}\right)\,dx\\
		&\quad-\int_{\R^3}\p_t\na^\nu u_\varepsilon^i[\na^\nu,u_\varepsilon^i]\left(\frac{1-|u_\varepsilon|^2}{\varepsilon^2}\right)\,dx\nonumber\\
		&=\frac{1}{4}\frac{d}{dt}\int_{\R^3}\left|\na^\nu(|u_\varepsilon|^2)\right|^2\,dx-\int_{\R^3}\p_t\na^\nu u_\varepsilon^i[\na^\nu,u_\varepsilon^i]\left(\frac{1-|u_\varepsilon|^2}{\varepsilon^2}\right)\,dx\nonumber\\
		&\quad+\int_{\R^3}\left(\na^\nu u_\varepsilon\p_t u_\varepsilon-\sum_{|\mu|=1}^{k}\begin{pmatrix}
		\nu \\ \mu
		\end{pmatrix}\na^{\nu-\mu} u_\varepsilon^j\nabla^{\mu}\p_tu_\varepsilon^i\right)\na^\nu\left(\frac{1-|u_\varepsilon|^2}{\varepsilon^2}\right)\,dx\nonumber\\
		&=:\frac{1}{4}\frac{d}{dt}\int_{\R^3}\left|\na^\nu(|u_\varepsilon|^2)\right|^2\,dx+B_1+B_2.\nonumber
		\end{align}
		To estimate $J_{0,2}$, we have to control
		\begin{equation*}
		\left\|\varepsilon^{-2}[\na^\nu,u_\varepsilon]\left(1-|u_\varepsilon|^2\right)\right\|_{L^2},\quad\left\|\varepsilon^{-2}\na^\nu\left(1-|u_\varepsilon|^2\right)\right\|_{L^2}.
		\end{equation*}
		Due to the fact that $\frac{3}{4}\leq|u_\varepsilon|\leq \frac{5}{4}$, the equation \eqref{G-L3} gives
		\begin{equation*}
		\varepsilon^{-2}(1-|u_\varepsilon|^2)=|u_\varepsilon|^{-2}\left(\g_1N_\varepsilon+\g_2A_\varepsilon u_\varepsilon-u_\varepsilon^i\na_\a W_{p_\a^i}+u_\varepsilon^iW_{u_\varepsilon^i}\right).
		\end{equation*}
		Note that
		\begin{equation*}
		\na_\alpha W_{p^i_\a}=u_\varepsilon\#u_\varepsilon\#\na^2u_\varepsilon+u_\varepsilon\#\na u_\varepsilon\#\na u_\varepsilon+\na^2 u_\varepsilon
		\end{equation*}
		and
				\begin{equation*}
		 W_{u_\varepsilon^i}=u_\varepsilon\#\na u_\varepsilon\#\na u_\varepsilon.
		\end{equation*}
	    Then we  use \eqref{com123} to write
		\begin{align*}\label{J02-1}
		\varepsilon^{-2}[\na^\nu,u_\varepsilon]\left(1-|u_\varepsilon|^2\right)=&\g_1[\na^\nu, u_\varepsilon|u_\varepsilon|^{-2}]N_\varepsilon-\g_1u_\varepsilon[\na^\nu,|u_\varepsilon|^{-2}]N_\varepsilon\\
		&+\g_2[\na^\nu,u_\varepsilon \#u_\varepsilon]A_{\varepsilon }-\g_2u_\varepsilon[\na^\nu,u_\varepsilon]A_{\varepsilon }\nonumber\\
		&+[\na^\nu,|u_\varepsilon|^{-2}u_\varepsilon\#u_\varepsilon\#u_\varepsilon\#u_\varepsilon+|u_\varepsilon|^{-2}u_\varepsilon\#u_\varepsilon]\na^2 u_\varepsilon\\
		&-u_\varepsilon\#[\na^\nu,|u_\varepsilon|^{-2}u_\varepsilon\#u_\varepsilon\#u_\varepsilon+|u_\varepsilon|^{-2}u_\varepsilon]\na^2 u_\varepsilon\nonumber\\
		&+[\na^\nu,|u_\varepsilon|^{-2}u_\varepsilon\#u_\varepsilon\#u_\varepsilon]\na u_\varepsilon\#\na u_\varepsilon\\
		&-u_\varepsilon[\na^\nu,|u_\varepsilon|^{-2}u_\varepsilon\#u_\varepsilon]\na u_\varepsilon\#\na u_\varepsilon \nonumber\\
=&:\sum_{i=1}^{8}G_i.\nonumber
		\end{align*}
		We apply Lemma \ref{pro-com estimates} to estimate $G_i$, $i=1,\cdots,8$. Firstly, we  claim that for any $u_\varepsilon$ satisfying $\frac{3}{4}\leq|u_\varepsilon|\leq \frac{5}{4}$ and any $k\geq 0$
        \begin{equation}\label{u-m}
       \|\nabla^{k+1}(|u_\varepsilon|^{-2})\|_{L^2}\leq C\|\nabla^{k+1} u_\varepsilon\|_{L^2}.
        \end{equation}
       Indeed, by direct calculations, we have
       \begin{align*}
       |\na^{k+1}(|u_\varepsilon|^{-2})|\leq C\sum_{|\mu_1|+\cdots|\mu_m|=k+1}|\na^{\mu_1}u_\varepsilon\#\cdots\#\na^{\mu_m}u_\varepsilon|.
       \end{align*}
       Then, it follows from the H\"{o}lder and Gagliardo–Nirenberg interpolation inequalities that
       \begin{align*}
        \|\nabla^{k+1}(|u_\varepsilon|^{-2})\|_{L^2}&\leq C\sum_{|\mu_1|+\cdots|\mu_m|=k+1}\|\na^{\mu_1}u_\varepsilon\#\cdots\#\na^{\mu_m} u_\varepsilon\|_{L^2}\\
        &\leq C\sum_{|\mu_1|+\cdots|\mu_m|=k+1}\|\na^{\mu_1}u_\varepsilon\|_{L^{r_1}}\cdots\|\na^{\mu_m}u_\varepsilon\|_{L^{r_m}}\\
        &\leq C\|\na^{k+1}u_\varepsilon\|_{L^2}^{\a_1}\|u_\varepsilon\|_{L^\infty}^{1-\alpha_1}\cdots\|\na^{k+1}u_\varepsilon\|_{L^2}^{\a_m}\|u\|_{L^\infty}^{1-\alpha_m}\leq C\|\na^{k+1}u_\varepsilon\|_{L^2},
       \end{align*}
       where $r_i$ and $\a_i$, $i=1\cdots m$ satisfies
       \begin{equation*}
       \sum_{i=1}^n\frac{1}{r_i}=\frac{1}{2}, \quad\frac{1}{r_i}=\frac{|\mu_i|}{3}+(\frac{1}{2}-\frac{k+1}{3})\a_i+\frac{1-\a_i}{\infty}\, \mbox{, so that }\, \sum_{i=1}^{n}\a_i=1.
       \end{equation*}
       Then, it follows from an expression like \eqref{com-n1} and arguments like \eqref{com-n1}-\eqref{numu} that
       \begin{align*}
       \|G_1\|_{L^2}\leq &C\||\na (u_\varepsilon|u_\varepsilon|^{-2})||\na^k \p_t u_\varepsilon|\|_{L^2}+C\|\na \p_t u_\varepsilon\|_{L^2}\|\na^{k+2}(u_\varepsilon|u_\varepsilon|^{-2})\|_{L^2}\\
       &+C\|\na \p_t u_\varepsilon\|_{H^{k-1}}\|\na (u_\varepsilon|u_\varepsilon|^{-2})\|_{H^k}+C\|v_\varepsilon\|_{L^\infty}\|\na^{k+2}(u_\varepsilon|u_\varepsilon|^{-2})\|_{L^2}\nonumber\\
       &+C\|\na(u_\varepsilon|u_\varepsilon|^{-2})\|_{L^\infty}\|\na^{k+1} v_\varepsilon\|_{L^2}+C\|\na(u_\varepsilon|u_\varepsilon|^{-2})\|_{L^\infty}\|v_\varepsilon\|_{L^\infty}\|\na ^{k+1}u_\varepsilon\|_{L^2}\nonumber\\
       &+C\|v_\varepsilon\|_{L^\infty}\|\na^{k+2}(u_\varepsilon\#u_\varepsilon\#|u_\varepsilon|^{-2})\|_{L^2}+C\|\na u_\varepsilon\|_{L^\infty}\|\na^{k+1}v_\varepsilon\|_{L^2}\nonumber\\
       \leq&C\|v_\varepsilon\|_{L^\infty}\|\nabla^{k+2}u_\varepsilon\|_{L^2}+C\|\nabla u_\varepsilon\|_{L^\infty}\|\nabla^{k+1}v_\varepsilon\|_{L^2}\nonumber\\
       &+C\|\na^{k+1}u_\varepsilon\|_{L^2}\|\nabla u_\varepsilon\|_{L^\infty}\|v_\varepsilon\|_{L^\infty}+C\||\na u_\varepsilon||\na^k\p_t u_\varepsilon|\|_{L^2}\nonumber\\
       &+C\|\na\p_tu_\varepsilon\|_{L^2}\|\na^{k+2}u_\varepsilon\|_{L^2}+C\|\na\p_tu_\varepsilon\|_{H^{k-1}}\|\na u_\varepsilon\|_{H^k}. \nonumber
       \end{align*}
   Similarly, we  have the same estimate for $\|G_2\|_{L^2}$ as $\|G_1\|_{L^2}$.
       Here we have used
       \begin{align*}
       &\|\na^{k+2}(u_\varepsilon|u_\varepsilon|^{-2})\|_{L^2}+\|\na^{k+2}(u_\varepsilon\#u_\varepsilon\#|u_\varepsilon|^{-2})\|_{L^2}\\
       \leq& C\|\na^{k+2}(|u_\varepsilon|^{-2})\|_{L^2}+C\|\nabla^{k+2}u_\varepsilon\|_{L^2}\leq C\|\na^{k+2}u_\varepsilon\|_{L^2}\nonumber
       \end{align*}
       from \eqref{u-m}. By using Lemma \ref{pro-com estimates} again, it is easy to derive
       \begin{align*}
       &\|G_3+G_4\|_{L^2}\leq C\|v_\varepsilon\|_{L^\infty}\|\nabla^{k+2}u_\varepsilon\|_{L^2}+C\|\nabla u_\varepsilon\|_{L^\infty}\|\nabla^{k+1}v_\varepsilon\|_{L^2},\\
       &\|G_5+G_6\|_{L^2}\leq C\|\na u_\varepsilon\|_{L^\infty}\|\nabla^{k+2}u_\varepsilon\|_{L^2},\\
       &\|G_7+G_8\|_{L^2}\leq C\|\na u_\varepsilon\|_{L^\infty} \|\na^{k+1}u_\varepsilon\|_{L^2}.
       \end{align*}
       Therefore, we obtain
       \begin{align}\label{G0}
       &\left\|\varepsilon^{-2}[\na^\nu,u_\varepsilon^i](1-|u_\varepsilon|^2)\right\|_{L^2}\\
       \leq& C\||\na u_\varepsilon||\na^k\p_t u_\varepsilon|\|_{L^2}^2+C(\Lambda_\infty+\|\na\p_t u_\varepsilon\|_{L^2}^2)\mathcal{E}_{k+1}\nonumber\\
       &+C\mathcal{E}_k\Lambda_\infty\|\na u_\varepsilon\|_{L^\infty}^2+C\|\na\p_tu_\varepsilon\|^2_{H^{k-1}}\|\na u_\varepsilon\|^2_{H^k}.\nonumber
       \end{align}
      To estimate  $\left\|\varepsilon^{-2}\na^\nu(1-|u_\varepsilon|^2)\right\|_{L^2}$, we extract terms of higher order derivatives and write
       \begin{align*}
       \na^\nu\left(\frac{1-|u_\varepsilon|^2}{\varepsilon^2}\right)=&\g_1|u_\varepsilon|^{-2}\na^\nu N_\varepsilon+\g_1[\na^\nu,|u_\varepsilon|^{-2}]N_\varepsilon+\g_2|u_\varepsilon|^{-2}u_\varepsilon^j\na^\nu A_{\varepsilon ij}\\
       &+\g_2[\na^\nu, |u_\varepsilon|^{-2}u_\varepsilon^j]A_{\varepsilon ij}+u_\varepsilon\#u_\varepsilon\#u_\varepsilon\#\na^\nu\na^2 u_\varepsilon\nonumber\\
       &+[\na^\nu,u_\varepsilon\#u_\varepsilon\#u_\varepsilon]\na^2u_\varepsilon+u_\varepsilon\#\na^\nu\na^2 u_\varepsilon+[\na^\nu,u_\varepsilon]\na^2u_\varepsilon\nonumber\\
       &+u_\varepsilon\#u_\varepsilon\na^\nu(\na u_\varepsilon\#\na u_\varepsilon)+[\na^\nu, u_\varepsilon\# u_\varepsilon]\na u_\varepsilon\#\na u_\varepsilon.\nonumber
       \end{align*}
       Then, by using similar arguments to derive \eqref{G0}, it is clear that
       \begin{align}\label{G00}
       \left\|\na^\nu\left(\frac{1-|u_\varepsilon|^2}{\varepsilon^2}\right)\right\|^2_{L^2}\leq &C\|\na^{k+1}\p_t u_\varepsilon\|^2_{L^2}+C\|\na^{k+2}v_\varepsilon\|^2_{L^2}+C\|\na^{k+3}u_\varepsilon\|^2_{L^2}\nonumber\\
       &+C\||\na u_\varepsilon||\na^k\p_t u_\varepsilon|\|_{L^2}^2+C(\Lambda_\infty+\|\na\p_t u_\varepsilon\|_{L^2}^2)\mathcal{E}_{k+1}\\
       &+C\mathcal{E}_k\Lambda_\infty\|\na u_\varepsilon\|_{L^\infty}^2+C\|\na\p_tu_\varepsilon\|^2_{H^{k-1}}\|\na u_\varepsilon\|^2_{H^k}.\nonumber
       \end{align}
      Therefore, using \eqref{numu}, \eqref{G00} and Young's inequality, we have
       \begin{align}\label{B2}
       |B_1+B_2|\leq &\delta_1\|\na^{k+2}v_\varepsilon\|_{L^2}^2+\delta_2\|\na^{k+1}\p_t u_\varepsilon\|^2_{L^2}+\delta_3\|\na^{k+3}u_\varepsilon\|_{L^2}^2\\
       &+C\||\na u_\varepsilon||\na^k\p_t u_\varepsilon|\|_{L^2}^2+C(\Lambda_\infty+\|\na\p_t u_\varepsilon\|_{L^2}^2)\mathcal{E}_{k+1}\nonumber\\
       &+C\mathcal{E}_k\Lambda_\infty\|\na u_\varepsilon\|_{L^\infty}^2+C\|\na\p_tu_\varepsilon\|^2_{H^{k-1}}\|\na u_\varepsilon\|^2_{H^k}.\nonumber
       \end{align}
       Substituting \eqref{B2} into \eqref{J02}, one has
       \begin{align}\label{J-02}
       J_{0,2}\geq& \frac{1}{4}\frac{d}{dt}\int_{\R^3}\left|\na^\nu(|u_\varepsilon|^2)\right|^2\,dx-\delta_1\|\na^{k+2}v_\varepsilon\|_{L^2}^2-\delta_2\|\na^{k+1}\p_t u_\varepsilon\|_{L^2}^2\\
       &-\delta_3\|\na^{k+3}u_\varepsilon\|_{L^2}^2-C\||\na u_\varepsilon||\na^k\p_t u_\varepsilon|\|_{L^2}^2-C\mathcal{E}_k\Lambda_\infty\|\na u_\varepsilon\|_{L^\infty}^2\nonumber\\
       &-C(\Lambda_\infty+\|\na\p_t u_\varepsilon\|_{L^2}^2)\mathcal{E}_{k+1}-C\|\na\p_tu_\varepsilon\|^2_{H^{k-1}}\|\na u_\varepsilon\|^2_{H^k}.\nonumber
       \end{align}
		For $J_{0,3}$, it follows from the H\"{o}lder, Sobolev and Young inequalities that
		\begin{align}\label{J-03}
		|J_{0,3}|&\leq C\int_{\R^3}|\p_tu_\varepsilon||\na^{k+2}u_\varepsilon|^2\,dx\leq C\|\na^{k+2}u_\varepsilon\|_{L^6}\|\p_tu_\varepsilon\|_{L^3}\|\na^{k+2}u_\varepsilon\|_{L^2}\\
		&\leq \delta_3\|\na^{k+3}u_\varepsilon\|_{L^2}^2+C\|\p_tu_\varepsilon\|_{H^1}^2\|\na^{k+2}u_\varepsilon\|_{L^2}^2.\nonumber
		\end{align}
		For $J_{0,4}$, we first rewrite the commutator in the integral as
		\begin{align*}
		\na_{\a}\left([\na^{\nu-e_\b},W_{p^i_\a p^j_\g}]\na^2_{\b\g}u_\varepsilon^j\right)
		=[\na^{\nu-e_\beta}\na_\a,W_{p^i_\a p^j_\g}]\na^2_{\b\g}u_\varepsilon^j-\na_\a W_{p^i_\a p^j_\g}\na^\nu\na_\g u_\varepsilon^j.
		\end{align*}
		Since $W_{p^i_\a p^j_\g}=u_\varepsilon\#u_\varepsilon+1$, then, utilizing Lemma \ref{pro-com estimates} gives
		\begin{align*}
		\|\na_{\a}\left([\na^{\nu-e_\b},W_{p^i_\a p^j_\g}]\na^2_{\b\g}u_\varepsilon^j\right)\|_{L^2}\leq C\|\na u_\varepsilon\|_{L^\infty}\|\na^{k+2}u_\varepsilon\|_{L^2}.
		\end{align*}
		Thus, it follows from Young's inequality that
		\begin{align*}\alabel{higher estimates J04}
		|J_{0,4}|\leq\delta_2\|\na^{k+1}\p_t u_\varepsilon\|_{L^2}^2+C\|\na u_\varepsilon\|_{L^\infty}^2\|\na^{k+2}u_\varepsilon\|_{L^2}^2.
		\end{align*}
		For $J_{0,5}$, note that $W_{u_\varepsilon^i}\equiv W_{p_\a^iu_\varepsilon^j}\na_\b u_\varepsilon^j=u_\varepsilon\#\na u_\varepsilon\#\na u_\varepsilon,$
	    then Lemma \ref{pro-com estimates} gives
		\begin{equation}\label{uuu}
		\|\na^\nu(u_\varepsilon\#\na u_\varepsilon\#\na u_\varepsilon)\|_{L^2}\leq C(\|\na u_\varepsilon\|_{L^\infty}\|\na^{k+2}u_\varepsilon\|_{L^2}+\|\nabla u_\varepsilon\|_{L^\infty}\|\na^{k+1}u_\varepsilon\|_{L^2}).
		\end{equation}
		Hence, we can obtain that
		\begin{equation}\label{higher estimates J05}
		|J_{0,5}|\leq \delta_2\|\na^{k+1}\p_t u_\varepsilon\|_{L^2}^2+C(\|\na u_\varepsilon\|_{L^\infty}^2\|\na^{k+2}u_\varepsilon\|_{L^2}^2+\|\nabla u_\varepsilon\|_{L^\infty}^4\|\na^{k+1}u_\varepsilon\|_{L^2}^2).
		\end{equation}
		Substitute \eqref{J-02} - \eqref{higher estimates J05} into \eqref{2.81}, we compute
		\begin{align*}\alabel{2.84}
        J_0\geq&\frac{d}{dt}\int_{\R^3}\frac{1}{2}\na^{\nu}\na_{\a} u_\varepsilon^iW_{p^i_\a p^j_\g}\na^{\nu}\na_{\g} u_\varepsilon^j+\frac{1}{4\varepsilon^2}|\na^\nu(|u_\varepsilon|^2)|^2\,dx
		\\
		&-\delta_1\|\na^{k+2}v_\varepsilon\|_{L^2}^2-3\delta_2\|\na^{k+1}\p_t u_\varepsilon\|_{L^2}^2-2\delta_3\|\na^{k+3}u_\varepsilon\|_{L^2}^2\nonumber\\
		&-C\||\na u_\varepsilon||\na^k\p_t u_\varepsilon|\|_{L^2}^2-C(\Lambda_\infty+\|\p_tu_\varepsilon\|_{H^1}^2)\mathcal{E}_{k+1}\nonumber\\
		&-C\mathcal{E}_k\Lambda_\infty\|\na u_\varepsilon\|_{L^\infty}^2-C\|\na\p_tu_\varepsilon\|^2_{H^{k-1}}\|\na u_\varepsilon\|^2_{H^k}.\nonumber
		\end{align*}
		Plugging \eqref{2.84} into \eqref{2.63}, we obtain
		\begin{align*}\alabel{2.85}
		&\frac{d}{dt}\int_{\R^3}\frac{1}{2}\na^{\nu}\na_{\a} u_\varepsilon^iW_{p^i_\a p^j_\g}\na^{\nu}\na_{\g} u_\varepsilon^j+\frac{1}{4\varepsilon^2}|\na^\nu(|u_\varepsilon|^2)|^2\,dx+\frac{1}{2\g_1}\|\na^\nu h_\varepsilon\|_{L^2}^2
		\\
		\leq&\int_{\R^3}\na^\nu h_\varepsilon^i\left(\frac{\g_2}{\g_1}\na^{\nu}A_{\varepsilon ij}u_\varepsilon^j-\na^{\nu}\Omega_{\varepsilon ij} u_\varepsilon^j\right)  \,dx
		\\
			&+\delta_1\|\na^{k+2}v_\varepsilon\|_{L^2}^2+3\delta_2\|\na^{k+1}\p_t u_\varepsilon\|^2_{L^2}+2\delta_3\|\na^{k+3}u_\varepsilon\|_{L^2}^2
			\\
			&+C\||\na u_\varepsilon||\na^k\p_t u_\varepsilon|\|_{L^2}^2+C(\Lambda_\infty+\|\p_tu_\varepsilon\|_{H^1}^2)\mathcal{E}_{k+1}\nonumber\\
			&+C\mathcal{E}_k\Lambda_\infty\|\na u_\varepsilon\|_{L^\infty}^2+C\|\na\p_tu_\varepsilon\|^2_{H^{k-1}}\|\na u_\varepsilon\|^2_{H^k}.\nonumber
		\end{align*}
			Summing \eqref{2.62} with \eqref{2.85} yields
			\begin{align}\label{2.86}
			&\frac{1}{2}\frac{d}{dt }\int_{\R^3}|\na^\nu v_\varepsilon|^2\,dx+\frac{1}{2}\frac{d}{dt}\int_{\R^3}\na^{\nu}\na_{\a} u_\varepsilon^iW_{p^i_\a p^j_\g}\na^{\nu}\na_{\g} u_\varepsilon^j\\
			&+\frac{1}{4\varepsilon^2}\frac{d}{dt}\int_{\R^3}|\na^\nu(|u_\varepsilon|^2)|^2\,dx+\alpha_4\int_{\R^3}|\na^\nu A_\varepsilon |^2  \,dx+\frac{1}{2\g_1}\int_{\R^3}|\na^\nu h_\varepsilon|^2\,dx\nonumber\\
			\leq& 3\delta_1\|\na^{k+2}v_\varepsilon\|_{L^2}^2+3\delta_2\|\na^{k+1}\p_t u_\varepsilon\|^2_{L^2}+2\delta_3\|\na^{k+3}u_\varepsilon\|_{L^2}^2
			\nonumber\\
			&+C\||\na u_\varepsilon||\na^k\p_t u_\varepsilon|\|_{L^2}^2+C(\Lambda_\infty+\|\p_tu_\varepsilon\|_{H^1}^2)\mathcal{E}_{k+1}\nonumber\\
			&+C\mathcal{E}_k\Lambda_\infty\|\na u_\varepsilon\|_{L^\infty}^2+C\|\na\p_tu_\varepsilon\|^2_{H^{k-1}}\|\na u_\varepsilon\|^2_{H^k}.\nonumber
			\end{align}
			Using integration by parts and \eqref{G-L2}, we note that
			\begin{equation}\label{anti-v}
			\alpha_4\int_{\R^3}|\na^\nu A_\varepsilon |^2  \,dx=\frac{\alpha_4}{2}\int_{\R^3}|\na^\nu \na v_\varepsilon|^2\,dx.
			\end{equation}
			Then, it remains to estimate terms involving $\na^{k+3}u_\varepsilon$ and $\na^{k+1}\p_t u_\varepsilon$. Applying $\na^\nu$, with index $\nu$ of order $k+1$, to \eqref{G-L3} and multiplying the resulting equation by $\na^{\nu}\p_t u_\varepsilon$, we have
			\begin{align}\alabel{2.87}
			&-\int_{\R^3}\na^{\nu}h_\varepsilon\cdot\na^{\nu}\p_t u_\varepsilon  \,dx +\g_1\int_{\R^3}|\na^\nu \p_tu_\varepsilon|^2\,dx \\
			=&\int_{\R^3}\left(\na^{\nu}(\g_1\Omega_\varepsilon u_\varepsilon-\g_2A_\varepsilon u_\varepsilon)-\g_1\na^{\nu}(v_\varepsilon\cdot \D u_\varepsilon)\right)\cdot\na^{\nu}\p_t u_\varepsilon  \;dx \nonumber\\
			\leq&\frac{\g_1}{4}\int_{\R^3}|\na^\nu \p_tu_\varepsilon|^2\,dx +C\|(\na^\nu\Omega_\varepsilon )u_\varepsilon\|_{L^2}^2+C\|(\na^\nu A_\varepsilon)u_\varepsilon\|_{L^2}^2\nonumber\\
			&+C\|[\na^\nu,u_\varepsilon]\Omega_\varepsilon\|_{L^2}^2+C\|[\na^\nu,u_\varepsilon]A_\varepsilon\|_{L^2}^2+C\|\na^\nu\na(v_\varepsilon\cdot u_\varepsilon)\|_{L^2}^2\nonumber\\
			\leq& \frac{\g_1}{4}\int_{\R^3}|\na^\nu \p_tu_\varepsilon|^2\,dx+C_1\|\na^{k+2}v_\varepsilon\|_{L^2}^2\nonumber\\
			&+C\|v_\varepsilon\|_{L^\infty}^2\|\nabla^{k+2}u_\varepsilon\|_{L^2}^2+C\|\nabla u_\varepsilon\|_{L^\infty}^2\|\nabla^{k+1}v_\varepsilon\|_{L^2}^2,\nonumber
			\end{align}
			where we have used \eqref{4-8} and Lemma \ref{pro-com estimates} in the last step. Plugging \eqref{2.84}  into \eqref{2.87} yields		
			\begin{align*}\alabel{2.88}
			&\frac{d}{dt}\int_{\R^3}\frac{1}{2}\na^{\nu}\na_{\a} u_\varepsilon^iW_{p^i_\a p^j_\g}\na^{\nu}\na_{\g} u_\varepsilon^j+\frac{1}{4\varepsilon^2}|\na^\nu(|u_\varepsilon|^2)|^2\,dx+\frac{3\g_1}{4}\|\na^\nu\p_tu_\varepsilon\|_{L^2}^2
			\\
			\leq&(C_1+\delta_1)\|\na^{k+2}v_\varepsilon\|_{L^2}^2+3\delta_2\|\na^{k+1}\p_t u_\varepsilon\|^2_{L^2}+2\delta_3\|\na^{k+3}u_\varepsilon\|_{L^2}^2
			\\
			&\quad+C\||\na u_\varepsilon||\na^k\p_t u_\varepsilon|\|_{L^2}^2+C(\Lambda_\infty+\|\p_tu_\varepsilon\|_{H^1}^2)\mathcal{E}_{k+1}\nonumber\\
			&\quad+C\mathcal{E}_k\Lambda_\infty\|\na u_\varepsilon\|_{L^\infty}^2+C\|\na\p_tu_\varepsilon\|^2_{H^{k-1}}\|\na u_\varepsilon\|^2_{H^k}.\nonumber
			\end{align*}
		Applying $\na^\nu\na_\beta$, with index $\nu$ of order $k+1$, to \eqref{G-L3}, multiplying by $\na^{\nu}\na _\beta u_\varepsilon$ and integrating by parts,   it follows from a similar argument as the one in \eqref{2.87} that
			\begin{align*}\alabel{3.35}
			&\frac12\frac{d}{dt}\int_{\R^3}|\na^{\nu}\na u_\varepsilon|^2\,dx-\frac{1}{\g_1}\int_{\R^3}\na^{\nu}\na_\b h_\varepsilon\cdot\na^{\nu}\na_\beta u_\varepsilon  \,dx \\
			=&-\int_{\R^3}\na^{\nu}\na_\b((v_\varepsilon\cdot\D u_\varepsilon)-\Omega_\varepsilon  u_\varepsilon+\frac{\g_2}{\g_1}A_\varepsilon u_\varepsilon) \cdot\na^{\nu}\na_\b u_\varepsilon  \,dx \nonumber
			\\
			=&\int_{\R^3}\na^{\nu}\na_\b((v_\varepsilon\cdot\D u_\varepsilon)-\Omega_\varepsilon  u_\varepsilon+\frac{\g_2}{\g_1}A_\varepsilon u_\varepsilon) \cdot\na^{\nu}\Delta u_\varepsilon  \,dx \nonumber
			\\
			\leq&\delta_3\|\na^{k+3}u_\varepsilon\|_{L^2}^2+C_2\|\na^{k+2}v_\varepsilon\|_{L^2}^2+C\|v_\varepsilon\|_{L^\infty}^2\|\nabla^{k+2}u_\varepsilon\|_{L^2}^2+C\|\nabla u_\varepsilon\|_{L^\infty}^2\|\nabla^{k+1}v_\varepsilon\|_{L^2}^2.\nonumber
			\end{align*}
			To estimate the term $$ K_0:=-\frac{1}{\g_1}\int_{\R^3}\na^{\nu}\na_\beta h_\varepsilon\cdot\na^{\nu}\na_\b u_\varepsilon  \;dx,$$ we use \eqref{M-F1} and integration by parts to get
			\begin{align*}\alabel{2.90}
		    K_0=&\frac{1}{\g_1}\int_{\R^3}\na^{\nu}\na^2_{\beta\g}u_\varepsilon^jW_{p_\alpha^ip_\g^j}\na^{\nu}\na^2_{\alpha\beta}u_\varepsilon^i  \;dx \\
			&-\frac{1}{\g_1}\int_{\R^3}\na^{\nu}\na_\b\left(\frac{(1-|u_\varepsilon|^2)u_\varepsilon^i}{\varepsilon^2}\right)\na^{\nu}\na_\beta u_\varepsilon^i  \;dx \nonumber\\
			&+\frac{1}{\g_1}\int_{\R^3}[\na^\nu,W_{p_\alpha^ip_\g^j}]\na^2_{\beta\g}u_\varepsilon^j\na^{\nu}\na^2_{\alpha\beta}u_\varepsilon^i  \,dx \nonumber\\
			&-\frac{1}{\g_1}\int_{\R^3}\left(\na^{\nu}W_{u_\varepsilon^i}\na^{\nu}\Delta u_\varepsilon^i-\nabla^{\nu}(W_{p_\alpha^i u_\varepsilon^j}\nabla_\beta u_\varepsilon^j)\nabla^{\nu}\nabla^2_{\alpha\beta}u_\varepsilon^i\right)  \,dx \nonumber\\
			=:&K_{0,1}+K_{0,2}+K_{0,3}+ K_{0,4}.\nonumber
			\end{align*}
			It follows from \eqref{Co2} that
			\begin{equation}\label{2.90-0}
			K_{0,1}\geq \frac{a}{\g_1}\|\na^{\nu}\na^2u_\varepsilon\|_{L^2}^2.
			\end{equation}
			For $K_{0,2}$, it can be rewritten as follows
			\begin{align*}
			K_{0,2}=&\frac{1}{\g_1\varepsilon^2}\int_{\R^3}\left(\na^\nu\na_\b\left(|u_\varepsilon|^2\right)u_\varepsilon^i-[\na^\nu\na_\b, u_\varepsilon^i]\left(1-|u_\varepsilon|^2\right)\right)\na^\nu\na_\beta u_\varepsilon^i\,dx\\
			=&\frac{1}{2\g_1\varepsilon^2}\int_{\R^3}|\na^\nu \na(|u_\varepsilon|^2)|^2\,dx-\frac{1}{\g_1\varepsilon^2}\int_{\R^3}\na^\nu\na_\b\left(|u_\varepsilon|^2\right)[\na^\nu,u_\varepsilon^i]\na_\beta u_\varepsilon^i\,dx\\
			&-\frac{1}{\g_1\varepsilon^2}\int_{\R^3}\left(\na_\b[\na^\nu, u_\varepsilon^i]\left(1-|u_\varepsilon|^2\right)+\na_\b u_\varepsilon^i\na^\nu(1-|u_\varepsilon|^2)\right)\na^\nu\na_\beta u_\varepsilon^i\,dx\\
			=&\frac{1}{2\g_1\varepsilon^2}\int_{\R^3}|\na^\nu \na(|u_\varepsilon|^2)|^2\,dx+\frac{1}{\g_1\varepsilon^2}\int_{\R^3}\na^\nu\left(|u_\varepsilon|^2\right)[\na^\nu\na_\beta,u_\varepsilon^i]\na_\beta u_\varepsilon^i\,dx\\
			&+\frac{1}{\g_1\varepsilon^2}\int_{\R^3}\left([\na^\nu, u_\varepsilon^i]\left(1-|u_\varepsilon|^2\right)\right)\na^\nu\Delta u_\varepsilon^i\,dx,
			\end{align*}
			where we have used the fact
			\begin{equation*}
			\na_\b[\na^\nu,f]g=[\na^\nu\na_\b,f]g-\na_\b f\na^\nu g
			\end{equation*}
			for   two functions $f$ and $g$. Then, it follows from \eqref{G00}, \eqref{G0} and Lemma \ref{pro-com estimates} that
			\begin{align*}
			&\left|\frac{1}{\g_1\varepsilon^2}\int_{\R^3}\na^\nu\left(|u_\varepsilon|^2\right)[\na^\nu\na_\beta,u_\varepsilon^i]\na_\beta u_\varepsilon^i\,dx\right|\\
			\leq& \|\varepsilon^{-2}\na^\nu(|u_\varepsilon|^{-2})\|_{L^2}\|[\na^\nu\na_\beta,u_\varepsilon^i]\na_\beta u_\varepsilon^i\|_{L^2}\\
			\leq& \|\varepsilon^{-2}\na^\nu(|u_\varepsilon|^{-2})\|_{L^2}\|\na u_\varepsilon\|_{L^\infty}\|\na^{k+2}u_\varepsilon\|_{L^2}\\
			\leq&\delta_1\|\na^{k+2}v_\varepsilon\|_{L^2}^2+\delta_2\|\na^{k+1}\p_t u_\varepsilon\|_{L^2}^2+\delta_3\|\na^{k+3}u_\varepsilon\|_{L^2}^2+C\||\na u_\varepsilon||\na^k\p_t u_\varepsilon|\|_{L^2}^2\\
			&+C(\Lambda_\infty+\|\na\p_t u_\varepsilon\|_{L^2}^2)\mathcal{E}_{k+1}+C\mathcal{E}_k\Lambda_\infty\|\na u_\varepsilon\|_{L^\infty}^2+C\|\na\p_tu_\varepsilon\|^2_{H^{k-1}}\|\na u_\varepsilon\|^2_{H^k}.\nonumber
			\end{align*}
			and
						\begin{align*}
			&\left|\frac{1}{\g_1\varepsilon^2}\int_{\R^3}\left([\na^\nu, u_\varepsilon^i]\left(1-|u_\varepsilon|^2\right)\right)\na^\nu\Delta u_\varepsilon^i\,dx\right|\\
			\leq& C\|[\na^\nu, u_\varepsilon^i]\left(\varepsilon^{-2}(1-|u_\varepsilon|^2)\right)\|_{L^2}\|\na^{k+3}u_\varepsilon\|_{L^2}\\
			\leq&\delta_3\|\na^{k+3}u_\varepsilon\|_{L^2}^2+C\||\na u_\varepsilon||\na^k\p_t u_\varepsilon|\|_{L^2}^2+C(\Lambda_\infty+\|\na\p_t u_\varepsilon\|_{L^2}^2)\mathcal{E}_{k+1}\nonumber\\
			&+C\mathcal{E}_k\Lambda_\infty\|\na u_\varepsilon\|_{L^\infty}^2+C\|\na\p_tu_\varepsilon\|^2_{H^{k-1}}\|\na u_\varepsilon\|^2_{H^k}.\nonumber
			\end{align*}
			Hence, we have
			\begin{align}
			K_{0,2}\geq& \frac{1}{2\g_1\varepsilon^2}\int_{\R^3}|\na^\nu \na_\b(|u_\varepsilon|^2)|^2\,dx-\delta_1\|\na^{k+2}v_\varepsilon\|_{L^2}^2-\delta_2\|\na^{k+1}\p_t u_\varepsilon\|_{L^2}^2\\
			&-\delta_3\|\na^{k+3}u_\varepsilon\|_{L^2}^2-C\||\na u_\varepsilon||\na^k\p_t u_\varepsilon|\|_{L^2}^2-C\mathcal{E}_k\Lambda_\infty\|\na u_\varepsilon\|_{L^\infty}^2\nonumber\\
			&-C(\Lambda_\infty+\|\na\p_t u_\varepsilon\|_{L^2}^2)\mathcal{E}_{k+1}-C\|\na\p_tu_\varepsilon\|^2_{H^{k-1}}\|\na u_\varepsilon\|^2_{H^k}.\nonumber
			\end{align}
			By using \eqref{uuu} and Young's inequality that
			\begin{align*}\alabel{2.92}
			|K_{0,3}+K_{0,4}|\leq& \delta_3\|\na^{k+3}u_\varepsilon\|_{L^2}^2+C\|\na u_\varepsilon\|_{L^\infty}^2\|\na^{k+2}u_\varepsilon\|_{L^2}^2\\
			&+C\|\nabla u_\varepsilon\|_{L^\infty}^4\|\na^{k+1}u_\varepsilon\|_{L^2}^2.
			\end{align*}
			Substituting \eqref{2.90-0}-\eqref{2.92} into \eqref{2.90}, the inequality \eqref{3.35} reads as
			\begin{align*}\alabel{2.94}
			&\frac12\frac{d}{dt}\int_{\R^3}|\na^{\nu}\na_\beta u_\varepsilon|^2\,dx+ \frac{a}{\g_1}\|\na^{\nu}\na^2u_\varepsilon\|_{L^2}^2 +\frac{1}{2\g_1\varepsilon^2}\|\na^{\nu}\na_\beta(|u_\varepsilon|^2)\|_{L^2}^2
			\\
		   \leq&(C_2+\delta_1)\|\na^{k+2}v_\varepsilon\|_{L^2}^2+\delta_2\|\na^{k+1}\p_t u_\varepsilon\|_{L^2}^2+3\delta_3\|\na^{k+3}u_\varepsilon\|_{L^2}^2\nonumber\\
		   &+C\||\na u_\varepsilon||\na^k\p_t u_\varepsilon|\|_{L^2}^2+C(\Lambda_\infty+\|\na\p_t u_\varepsilon\|_{L^2}^2)\mathcal{E}_{k+1}\nonumber\\
		   &+C\mathcal{E}_k\Lambda_\infty\|\na u_\varepsilon\|_{L^\infty}^2+C\|\na\p_tu_\varepsilon\|^2_{H^{k-1}}\|\na u_\varepsilon\|^2_{H^k}.\nonumber
			\end{align*}
			Summing \eqref{2.94} with \eqref{2.88} yields
			\begin{align*}\alabel{2.95}
						&\frac{d}{dt}\int_{\R^3}\frac{1}{2}\na^{\nu}\na_{\a} u_\varepsilon^iW_{p^i_\a p^j_\g}\na^{\nu}\na_{\g} u_\varepsilon^j+\frac{1}{2}|\na^{\nu}\na_\beta u_\varepsilon|^2+\frac{1}{4\varepsilon^2}|\na^\nu(|u_\varepsilon|^2)|^2\,dx
						\\
						&+\frac{3\g_1}{4}\|\na^{\nu}\p_t u_\varepsilon\|_{L^2}^2  + \frac{a}{\g_1}\|\na^{\nu}\na^2u_\varepsilon\|_{L^2}^2  +\frac{1}{2\g_1\varepsilon^2}\|\na^{\nu}(|u_\varepsilon|^2)\|_{L^2}^2  \nonumber
						\\
						\leq&\tilde{C}_1\|\na^{k+2} v_\varepsilon\|^2_{L^2} +4 \delta_2\|\na^{k+1}\p_t u_\varepsilon\|^2_{L^2} +5 \delta_3\|\na^{k+3} u_\varepsilon\|^2_{L^2}
						\\
						&+C\||\na u_\varepsilon||\na^k\p_t u_\varepsilon|\|_{L^2}^2+C(\Lambda_\infty+\|\p_tu_\varepsilon\|_{H^1}^2)\mathcal{E}_{k+1}\nonumber\\
						&+C\mathcal{E}_k\Lambda_\infty\|\na u_\varepsilon\|_{L^\infty}^2+C\|\na\p_tu_\varepsilon\|_{H^{k-1}}\|\na u_\varepsilon\|_{H^k},\nonumber
						\end{align*}
						where $\tilde{C_1}=C_1+C_2+2\delta_1$.
					 Multiplying \eqref{2.86} by $ \tilde C_2:=\max\{1,4\alpha_4^{-1}(\tilde{C}_1+1)\}$, adding with \eqref{2.95} and take a summation over all the index $\nu$ of order $k+1$, we can obtain
					 \begin{align*}\alabel{2.96}
					 &\frac{1}{2}\frac{d}{dt}\mathcal{E}_{k+1}(t)+\frac{d}{dt}\int_{\R^3}\na^{k+1}\D_{\alpha}u_\varepsilon^iW_{p_\alpha^ip^j_\gamma}\na^{k+1}\D_{\gamma}u_\varepsilon^j  +\frac{1}{2\varepsilon^2}|\na^{k+1}(|u_\varepsilon|^2)|^2\,dx\\
					 &+\tilde{a}\mathcal{D}_{k+1}(t)+(4\g_1\varepsilon^2)^{-1}\|\na^{k+1}(|u_\varepsilon|^{-2})\|_{L^2}^2
					 \\
					 \leq&C\||\na u_\varepsilon||\na^k\p_t u_\varepsilon|\|_{L^2}^2+C(\Lambda_\infty+\|\p_tu_\varepsilon\|_{H^1}^2)\mathcal{E}_{k+1}+C\mathcal{E}_k\Lambda_\infty\|\na u_\varepsilon\|_{L^\infty}^2\nonumber\\
					 &+C\|\na\p_tu_\varepsilon\|_{H^{k-1}}\|\na u_\varepsilon\|_{H^k},\nonumber
					 \end{align*}
			where we have chosen   $\delta_1=\a_4/12$, $\delta_2=\g_1(16(\tilde C_2+1))^{-1}$, $\delta_3=a(2\g_1(2\tilde C_2+5))^{-1}$ and $\tilde{a}=\min\{1,\frac{a}{2\g_1},\frac{\g_1}{2}\}$.  Furthermore, it follows from using uniform estimates of  the strong solution in Proposition \ref{prop2.1} that for any $\delta>0$ there exists a $R_0$  depending only on $M$ such that
			\begin{align}\label{higher estimates P1.0}
			\sup_{0\leq t\leq T_M,x_0\in\R^3}\int_{B_{R_0}(x_0)}|v_\varepsilon|^3+|\na u_\varepsilon|^3  \,dx\leq \delta^3.
			\end{align}
			By standard covering argument, we have
			\begin{align*}
				&C\||\na u_\varepsilon||\na^k\p_t u_\varepsilon|\|_{L^2}^2
				\leq C\sum_i\left(\int_{B_{R_0}(x_i)}|\na^k\p_t u_\varepsilon|^6\;dx\right)^{1/3}\delta^2
				\leq C\delta^2\|\na^{k+1}\p_t u_\varepsilon\|_{L^2}^2
			\end{align*}
			so that we can choose $\delta$ small enough satisfying $C\delta^2=\frac{\tilde{a}}{2}$ to conclude
			\begin{align*}\alabel{2.97} &\frac{1}{2}\frac{d}{dt}\mathcal{E}_{k+1}(t)+\frac{d}{dt}\int_{\R^3}\na^{k+1}\D_{\alpha}u_\varepsilon^iW_{p_\alpha^ip^j_\gamma}\na^{k+1}\D_{\gamma}u_\varepsilon^j  +\frac{1}{2\varepsilon^2}|\na^{k+1}(|u_\varepsilon|^2)|^2\,dx\\
			&+\frac{\tilde{a}}{2}\mathcal{D}_{k+1}(t)+(4\g_1\varepsilon^2)^{-1}\|\na^{k+1}(|u_\varepsilon|^{-2})\|_{L^2}^2
			\\
			\leq&C(\Lambda_\infty+\|\p_tu_\varepsilon\|_{H^1}^2)\mathcal{E}_{k+1}+C\mathcal{E}_k\Lambda_\infty\|\na u_\varepsilon\|_{L^\infty}^2+C\left(\sum_{m=0}^{k}\mathcal{E}_m\right)\left(\sum_{m=1}^{k}\mathcal{D}_m\right).\nonumber
			\end{align*}
			By inductive assumptions, it holds for any $l=0,1,\cdots,k$ with $k\geq 1$, any $\tau>0$ and any $s\in(\tau,T_M]$ that
			\begin{align}\label{Hyp}
			\mathcal{E}_l(s)+\int_{\tau}^{s} \mathcal{D}_{l}(t)+\varepsilon^{-2}\|\na^{l+1}(|u_\varepsilon(t)|^2)\|_{L^2}^2\,dt\leq C(\tau,l).
			\end{align}
			Applying the mean value theorem in \eqref{Hyp} for $l=k$,   there exists a $\tau_{\varepsilon}\in(\tau, 2\tau)$ such that
			\begin{equation}\label{higher estimates P1.3}
			\mathcal{E}_{k+1}(\tau_\varepsilon)+\varepsilon^{-2}\|\na^{k+1}(|u_\varepsilon(\tau_\varepsilon)|^2)\|_{L^2}^2\leq C(\tau,k).
			\end{equation}
			On the other hand, by the Sobolev embedding and Proposition 3.1, one has
			\begin{align}\label{LL1}
		    (\Lambda_\infty+\|\p_tu_\varepsilon\|_{H^1}^2)&\leq C(\|v_\varepsilon\|_{H^2}^2+\|\na u_\varepsilon\|_{H^2}^2+\|\p_tu_\varepsilon\|_{H^1}^2)\leq C(\mathcal{D}_0+\mathcal{D}_1).
			\end{align}
			Moreover, the Sobolev embedding and \eqref{Hyp} imply
			\begin{equation}\label{LL2}
			\|\na u_\varepsilon\|_{L^\infty}^2\leq C\|\na u_\varepsilon\|_{H^2}^2\leq C(\|\na^{k+2}u_\varepsilon\|_{L^2}^2+1)
			\end{equation}
			when $k\geq 1$.
			Therefore, we apply \eqref{Hyp}-\eqref{LL2} into \eqref{2.97} 
			\begin{align*}\alabel{2.98} &\frac{1}{2}\frac{d}{dt}\mathcal{E}_{k+1}(t)+\frac{d}{dt}\int_{\R^3}\na^{k+1}\D_{\alpha}u_\varepsilon^iW_{p_\alpha^ip^j_\gamma}\na^{k+1}\D_{\gamma}u_\varepsilon^j  +\frac{1}{2\varepsilon^2}|\na^{k+1}(|u_\varepsilon|^2)|^2\,dx\\
						&+\frac{\tilde{a}}{2}\mathcal{D}_{k+1}(t)+(4\g_1\varepsilon^2)^{-1}\|\na^{k+1}(|u_\varepsilon|^{-2})\|_{L^2}^2
						\\
						\leq&C(\mathcal D_0+\mathcal D_1)\mathcal{E}_{k+1}+C\sum_{m=1}^{k}\mathcal{D}_m.\nonumber
						\end{align*}
			We apply the Gronwall  inequality in \eqref{2.98} for $t\in (\tau_\varepsilon,s)$ and conclude that \eqref{higher estimates} holds for $l=k+1$ on the $(2\tau,s)$. Since $\tau$ is an arbitrary positive constant,  we prove \eqref{higher estimates} for any $s\in (\tau,T_M]$ and $l=k+1$  which completes a proof of this lemma.
			\end{proof}

Next, we have the following strong convergence lemma
\begin{lemma}\label{lem4.2}
	Let $(v_\varepsilon, u_\varepsilon)$ be the strong solution, obtained in Proposition \ref{prop2.1}, to the system \eqref{G-L1}-\eqref{G-L3} in $\R^3\times[0, T_M]$. Then, for any $t\in[0,T_M]$, we have
	\begin{align}
	&\lim\limits_{\varepsilon\rightarrow 0}\int_{\R^3}|v_\varepsilon(t)|^2\,dx=\int_{\R^3}|v(t)|^2\,dx,\label{4.45}\\
	&\lim\limits_{\varepsilon\rightarrow 0}\int_{\R^3}|\nabla u_\varepsilon(t)|^2\,dx=\int_{\R^3}|\nabla u(t)|^2\,dx,\label{4.46}\\
	& \lim\limits_{\varepsilon\rightarrow 0}\int_{\R^3}\varepsilon^{-2}|(1-|u_\varepsilon(t)|^2)|^2\,dx=0.\label{4.47}
	\end{align}
	
\end{lemma}
\begin{proof}
	It follows from Lemma \ref{energy-identity} that
	\begin{align}\label{ene-id1}
	&\int_{\R^3}\left(\frac{|v_\varepsilon(t)|^2}{2}+W(u_\varepsilon,\D u_\varepsilon)(t)+\frac{1}{4\varepsilon^2}(1-|u_\varepsilon(t)|^2)^2\right)\,dx\\
	&+\alpha_4\int_{0}^{t}\int_{\R^3}|A_\varepsilon|^2\;dxdt+\alpha_1\int_{0}^{t}\int_{\R^3}|u_\varepsilon^TA_\varepsilon u_\varepsilon|^2\,dxdt\nonumber\\
	&+\beta\int_{0}^{t}\int_{\R^3}|A_\varepsilon u_\varepsilon|^2\,dxdt+\frac{1}{\g_1}\int_{0}^{t}\int_{\R^3}|\g_1N_\varepsilon+\g_2A_\varepsilon u_\varepsilon|^2\,dxdt\nonumber\\
	=&\int_{\R^3}\left(\frac{|v_0|^2}{2}+W(u_0,\nabla u_0)\right)\,dx.\nonumber
	\end{align}
	By the lower semi-continuity, we have
	\begin{align}\label{low}
	&\int_{\R^3}|v(t)|^2\,dx\leq\liminf\limits_{\varepsilon\rightarrow 0}\int_{\R^3}|v_\varepsilon(t)|^2\,dx,\\
	&\int_{\R^3}W(u,\nabla u)(t)\,dx\leq\liminf\limits_{\varepsilon\rightarrow 0}\int_{\R^3}W(u_\varepsilon,\nabla u_\varepsilon)(t)\,dx,\nonumber\\
	&\int_{0}^{t}\int_{\R^3}|A|^2\,dxdt\leq\liminf\limits_{\varepsilon\rightarrow 0}\int_{0}^{t}\int_{\R^3}|A_\varepsilon|^2\,dxdt,\nonumber\\
	&\int_{0}^{t}\int_{\R^3}|u^TA u|^2\,dxdt\leq\liminf\limits_{\varepsilon\rightarrow 0}\int_{0}^{t}\int_{\R^3}|u_\varepsilon^TA_\varepsilon u_\varepsilon|^2\,dxdt,\nonumber\\
	&\int_{0}^{t}\int_{\R^3}|A u|^2\,dxdt\leq\liminf\limits_{\varepsilon\rightarrow 0}\int_{0}^{t}\int_{\R^3}|A_\varepsilon u_\varepsilon|^2\,dxdt,\nonumber\\
	&\int_{0}^{t}\int_{\R^3}|\g_1N+\g_2A u|^2\,dxdt\leq\liminf\limits_{\varepsilon\rightarrow 0}\int_{0}^{t}\int_{\R^3}|\g_1N_\varepsilon+\g_2A_\varepsilon u_\varepsilon|^2\,dxdt.\nonumber
	\end{align}
	On the other hand, using a similar argument in Lemma \ref{energy-identity} (c.f. \cite{WZZ}), one has
	\begin{align*}
	&\int_{\R^3}\left(\frac{|v(t)|^2}{2}+W(u,\D u)(t)\right)\;dx+(\alpha_1+\frac{\g_2^2}{\g_1})\int_{0}^{t}\int_{\R^3}|u^TA u|^2\,dxdt\\
	&+\alpha_4\int_{0}^{t}\int_{\R^3}|A|^2\,dxdt+\beta\int_{0}^{t}\int_{\R^3}|Au|^2\,dxdt+\frac{1}{\g_1}\int_{0}^{t}\int_{\R^3}|h-(u\cdot h)u|^2\,dxdt\nonumber\\
	=&\int_{\R^3}\left(\frac{|v_0|^2}{2}+W(u_0,\nabla u_0)\right)\,dx.\nonumber
	\end{align*}
	It follows from \eqref{E-L3} that
	\begin{align*}
	\int_{0}^{t}\int_{\R^3}|h-(u\cdot h)u|^2\,dxdt=&\int_{0}^{t}\int_{\R^3}|\g_1N+\g_2 \left(Au-(u^TAu)u\right)|^2\,dxdt\\
	=&\int_{0}^{t}\int_{\R^3}|\g_1N+\g_2 Au|^2\,dxdt-\g_2^2\int_{0}^{t}\int_{\R^3}|u^TAu|^2\,dxdt,
	\end{align*}
	where we have used $u\cdot N=0$ due to the fact that $|u|=1$. Hence, we have the energy identity for the Ericksen-Leslie system \eqref{E-L1}-\eqref{E-L3} that
	\begin{align}\label{ener-id2}
	&\int_{\R^3}\left(\frac{|v(t)|^2}{2}+W(u,\D u)(t)\right)\,dx+\alpha_1\int_{0}^{t}\int_{\R^3}|u^TA u|^2\,dxdt+\alpha_4\int_{0}^{t}\int_{\R^3}|A|^2\,dxdt\\
	&+\beta\int_{0}^{t}\int_{\R^3}|Au|^2\,dxdt+\frac{\g_2^2}{\g_1}\int_{0}^{t}\int_{\R^3}|\g_1N+\g_2Au|^2\,dxdt\nonumber\\
	&=\int_{\R^3}\left(\frac{|v_0|^2}{2}+W(u_0,\nabla u_0)\right)\,dx.\nonumber
	\end{align}
	Comparing \eqref{ene-id1} with \eqref{ener-id2} and using \eqref{low}, we first obtain \eqref{4.47}. Repeating the comparison of \eqref{ene-id1} and \eqref{ener-id2}, we have \eqref{4.45} and
	\begin{equation*}
	\int_{\R^3}W(u,\nabla u)\,dx=\lim\limits_{\varepsilon\rightarrow 0}\int_{\R^3}W(u_\varepsilon,\nabla u_\varepsilon)\,dx
	\end{equation*}
	which implies \eqref{4.46}, since $W(u,\nabla u)$ satisfies \eqref{Co2}.
\end{proof}
Now we give a proof of Theorem \ref{thm2}.
\begin{proof}[\bf Proof of Theorem \ref{thm2}]
	Let $(u,v)$ be the strong solution  to the Ericksen-Leslie system \eqref{E-L1}-\eqref{E-L3} in $\R^3\times[0,T^\ast)$ with initial data $(u_0,v_0)\in H^2_b(\R^3)\times H^1(\R^3)$, where  $T^\ast$ is its maximal existence time. Given any $T\in (0,T^\ast)$, set
	\begin{equation*}
	M= 2\sup_{0\leq t\leq T}\|(\nabla u,v)\|_{H^1(\R^3)}^2.
	\end{equation*}
	By Proposition \ref{prop2.1}, the system \eqref{G-L1}-\eqref{G-L3}  with initial data $(u_0,v_0)$ has a unique strong solution $(u_\varepsilon,v_\varepsilon)$ in $\R^3\times[0,T_M]$ satisfying
	\begin{equation}\label{4.1}
	\frac{3}{4}\leq |u_\varepsilon|\leq\frac{5}{4}
	\end{equation}
	and
	\begin{align}\label{4.2}
	&\sup_{0\leq t\leq T_M}\left(\|v_\varepsilon\|_{H^1}^2+\|\D u_\varepsilon\|_{H^1}^2+\varepsilon^{-2}\|(1-|u_\varepsilon|^2)\|_{H^1}^2\right)+\|\D v_\varepsilon\|_{L^2(0,T_0;H^1)}^2\\
	&+\|\D^2 u_\varepsilon\|_{L^2(0,T_0;H^1)}^2+\|\p_t u_\varepsilon\|_{L^2(0,T_0;H^1)}^2+\varepsilon^{-2}\|\D(|u_\varepsilon|^2)\|_{L^2(0,T_0;H^1)}^2\leq C_M\nonumber
	\end{align}
	for any $\varepsilon\leq \varepsilon_M$.
	Next, applying Lemma \ref{higher estimates lemma}, one has the following higher estimates
	\begin{align*}
	&(v_\varepsilon,\nabla u_\varepsilon)\in L^\infty(\tau,T_M;H^{k}(\R^3))\cap L^2(\tau,T_M;H^{k+1}(\R^3)),\\
	&(\p_tv_\varepsilon,\p_t\nabla u_\varepsilon)\in L^2(\tau,T_M;H^{k-1}(\R^3)),
	\end{align*}
	which have uniform bounds   in $\varepsilon$, for any $k\geq 2$.
	It follows from the Aubin-Lions Lemma that there exists a subsequence such that
	\begin{align*}
	&v_{\varepsilon_i}\rightarrow v\quad \text{in }C([\tau,T_M];H^{k-1}(B_R(0)))\\
	&u_{\varepsilon_i}\rightarrow u\quad \text{in } C([\tau,T_M];H^k(B_R(0)))
	\end{align*}
	for any $k\geq 2$ and $R\in(0,\infty)$. This together with \eqref{4.2} implies that
	\begin{equation*}
	(u_{\varepsilon_i},v_{\varepsilon_i})\rightarrow(u,v)\quad\text{in } C([\tau,T_M];C^{\infty}_{loc}(\R^3))
	\end{equation*}
	with $|u|=1$. By Theorem \ref{thm1}, $(u,v)$ must be the unique solution to the Ericksen-Leslie system \eqref{E-L1}-\eqref{E-L3}. Since $(u,v)$ is unique and any sequence $(u_\varepsilon,v_\varepsilon)$ has a convergent subsequence $(u_{\varepsilon_i},v_{\varepsilon_i})$, then the sequence  $(u_\varepsilon,v_\varepsilon)$ converges to $(u,v)$ in $C([\tau,T_M];C^{\infty}_{loc}(\R^3))$. Then, using the equations \eqref{E-L1}-\eqref{E-L3}, it is not difficult to prove the smooth convergence in $t$; that is
	\begin{equation*}
	(u_\varepsilon,v_\varepsilon)\rightarrow(u,v)\quad\text{in } C^\infty([\tau,T_M];C^{\infty}_{loc}(\R^3)).
	\end{equation*}
	Now, we prove that $T_M$ can be extended to $T$.
	
	Suppose that $T_M<T$. Then it follows from Lemma \ref{higher estimates}, \ref{lem4.2} and integration by parts that
	\begin{equation*}
	\lim\limits_{\varepsilon\rightarrow 0}\|(\nabla v_\varepsilon-\nabla v)(T_M)\|_{L^2}^2\leq C\lim\limits_{\varepsilon\rightarrow 0}\|(v_\varepsilon-v)(T_M)\|_{L^2}^2\|\nabla^2 v_\varepsilon-\nabla^2 v)(T_M)\|_{L^2}^2=0.
	\end{equation*}
	Similarly,
	\begin{equation*}
	\lim\limits_{\varepsilon\rightarrow 0}\|(\nabla^2 u_\varepsilon-\nabla^2 u)(T_M)\|_{L^2}^2=0,\quad\lim\limits_{\varepsilon\rightarrow 0}\varepsilon^{-2}\|\nabla (1-|u_\varepsilon|^2)(T_M)\|_{L^2}^2=0
	\end{equation*}
	Therefore, we obtain
	\begin{align*}
	&\lim\limits_{\varepsilon\rightarrow 0}\left(\|v_\varepsilon(T_M)\|_{H^1(\R^3)}^2+\|\nabla u_\varepsilon(T_M)\|_{H^1(\R^3)}^2+\varepsilon^{-2}\|(1-|u_\varepsilon|^2)^2(T_M)\|_{H^1(\R^3)}^2\right)\\
	=&\|v(T_M)\|_{H^1(\R^3)}^2+\|\nabla u(T_M)\|_{H^1(\R^3)}^2\leq\frac{M}{2}.\nonumber
	\end{align*}
	Hence,  for sufficiently small $\varepsilon$, one has
	\begin{equation*}
	\|v_\varepsilon(T_M)\|_{H^1(\R^3)}^2+\|\nabla u_\varepsilon(T_M)\|_{H^1(\R^3)}^2+\varepsilon^{-2}\|(1-|u_\varepsilon|^2)^2(T_M)\|_{H^1(\R^3)}^2\leq M.
	\end{equation*}
	Moreover, \eqref{4.47} and Lemma \ref{higher estimates lemma} imply $\frac{3}{4}\leq |u_\varepsilon(T_M)|\leq\frac{5}{4}$ for sufficiently small $\varepsilon$. Therefore, using $(v_\varepsilon(T_M),u_\varepsilon(T_M))$ as a new initial data at $t=T_M$ and applying Proposition \ref{prop2.1} again, we can extend the strong solution $(u_\varepsilon,v_\varepsilon)$ to the time  $T_1=:\min\{T,2T_M\}$. By the same argument above, it is obvious that
	\begin{equation*}
	(u_\varepsilon,v_\varepsilon)\rightarrow(u,v)\quad\text{in } C^\infty([\tau,T_1];C^{\infty}_{loc}(\R^3)).
	\end{equation*}
	We repeat the above two steps and establish the convergence up to $T$ for any $T<T^\ast$.
	This completes the proof of Theorem \ref{thm2}.
\end{proof}

\noindent{\bf Acknowledgments:} The research  of the second
author was supported by the Australian Research Council
grant DP150101275. The third author was supported by the Australian Research Council
grant DP150101275 as a  postdoctoral fellow.

\end{document}